\documentclass[letterpaper, 11pt,  reqno]{amsart}

\usepackage{amsmath,amssymb,amscd,amsthm,amsxtra, esint, bbm, enumerate, relsize, xcolor}
\usepackage{mathtools} 
\usepackage{accents, cancel}
\usepackage{mathrsfs,bbm}
\usepackage[implicit=true]{hyperref}
\usepackage[initials, nobysame]{amsrefs}
\usepackage{enumitem}

\DefineSimpleKey{bib}{primaryClass}
\DefineSimpleKey{bib}{archivePrefix}

\usepackage{stmaryrd}

\usepackage{mparhack}

\allowdisplaybreaks[2]

\newtheorem{theorem}{Theorem} [section]

\newtheorem{lemma}[theorem]{Lemma}
\newtheorem{proposition}[theorem]{Proposition}

\theoremstyle{definition}
\newtheorem{remark}[theorem]{Remark}

\DeclareMathOperator{\Law}{Law}

\newcommand{\1}{\mathbbmss 1}

\newcommand{\noi}{\noindent}
\newcommand{\Z}{\mathbb{Z}}
\newcommand{\R}{\mathbb{R}}
\newcommand{\C}{\mathbb{C}}
\newcommand{\T}{\mathbb{T}}

\DeclareMathOperator{\WP}{WP}

\newcommand{\n}{\mathbf n}
\newcommand{\m}{\mathbf m}
\renewcommand{\l}{\mathbf l}
\newcommand{\p}{\mathbf p}

\let \div \relax
\DeclareMathOperator{\div}{div}

\newcommand{\bb}{\mathbb}

\let\Re=\undefined\DeclareMathOperator*{\Re}{Re}
\let\Im=\undefined\DeclareMathOperator*{\Im}{Im}

\newcommand{\E}{\mathbb{E}}

\newcommand{\al}{\alpha}

\newcommand{\dl}{\delta}

\newcommand{\ep}{\varepsilon}

\newcommand{\s}{\sigma}

\newcommand{\ft}{\widehat}

\newcommand{\wt}{\widetilde}
\newcommand{\cj}{\overline}

\newcommand{\ta}{\theta}

\newcommand{\les}{\lesssim}

\newcommand{\jb}[1]
{\langle #1 \rangle}

\newcommand{\N}{\mathbb{N}}

\newcommand{\eps}{\ep}

\renewcommand{\ft}{\widehat}

\DeclareMathOperator{\BMO}{BMO_+}

\numberwithin{equation}{section}
\numberwithin{theorem}{section}

\makeatletter
\@namedef{subjclassname@2020}{%
  \textup{2020} Mathematics Subject Classification}
\makeatother

\begin{document}
\baselineskip = 14pt

\title[Sharp quasi-invariance for cubic Szeg\"o]
{Sharp quasi-invariance threshold for the cubic Szeg\H{o} equation}

\author[J.~Coe and L.~Tolomeo]
{James Coe and Leonardo Tolomeo}

\address{
James Coe \\
School of Mathematics\\
The University of Edinburgh\\
and The Maxwell Institute for the Mathematical Sciences\\
James Clerk Maxwell Building\\
The King's Buildings\\
Peter Guthrie Tait Road\\
Edinburgh\\ 
EH9 3FD\\
 United Kingdom
 }

\email{j.coe@ed.ac.uk}

 \address{
 Leonardo Tolomeo\\
School of Mathematics\\
The University of Edinburgh\\
and The Maxwell Institute for the Mathematical Sciences\\
James Clerk Maxwell Building\\
The King's Buildings\\
Peter Guthrie Tait Road\\
Edinburgh\\ 
EH9 3FD\\
 United Kingdom
}

\email{l.tolomeo@ed.ac.uk}


\begin{abstract}
We consider the 1-dimensional cubic Szeg\H{o} equation with data distributed according to the Gaussian measure with inverse covariance operator $(1-\partial_x^2)^s$, where $s>\frac12$. We show that, for $s>1$, this measure is quasi-invariant under the flow of the equation, while for $s<1$, $s\neq \frac34$, the transported measure and the initial Gaussian measure are mutually singular for almost every time. This is the first observation of a transition from quasi-invariance to singularity in the context of the transport of Gaussian measures under the flow of Hamiltonian PDEs.  
\end{abstract}

\maketitle

\tableofcontents

\section{Introduction}
\subsection{Main result}
In this work, we consider the Cauchy problem for the cubic Szeg\H{o} equation on the 1-dimensional torus $\T=\R/(2\pi\Z)$, 
\begin{equation}\label{1}
    \begin{cases}
        i\partial_tu=\Pi(|u|^2u), \quad (t,x)\in \mathbb{R}\times\mathbb{T},\\
        u(0) = u_0 = \Pi(u_0),
    \end{cases}
\end{equation}
where $\Pi$ is the Szeg\H{o} projection onto non-negative Fourier frequencies:
$$\Pi\Big(\sum_{n\in \mathbb{Z}}\widehat{u}(n)e^{inx}\Big)=\sum_{n\geq 0}\widehat{u}(n)e^{inx},$$
and the condition $u_0 = \Pi(u_0)$ simply means that we restrict our attention to initial data which has Fourier support on nonnegative frequencies. 
Our goal is to study certain statistical properties of the evolution of \eqref{1}, when its initial data $u_0$ is randomly distributed, according to a particular family of Gaussian measures $\mu_s$.  

More precisely, we concern ourselves with the following question: denoting with $\Phi_t(u_0)$ the solution of \eqref{1} at time $t$, does it hold that 
$$ (\Phi_t)_{\#} \mu_s \ll \mu_s ? $$
When this phenomenon happens, we say that the measure is \emph{quasi-invariant} under the flow of \eqref{1}. 

The study of quasi-invariance for the flow of Hamiltonian PDEs has been initiated by Tzvetkov in \cite{tzvetkov_2015}. 
In particular, in this pioneering work, Tzvetkov showed that in the particular case of the regularized long wave (BBM) equation, and of Gaussian probability measures of the form 
\begin{equation} \label{musdef}
    \mu_s \sim \exp\Big(- \| u \|_{H^s}^2\Big)du,\footnote{For a rigorous definition of these measures in the setting of this paper, see Section \ref{sec:prelim}.}
\end{equation}
quasi-invariance holds in regime that is much wider than what one would expect from the classical results of Cameron-Martin \cite{CM} and Ramer \cite{Ramer}. Throughout the years, the technology to show quasi-invariance has been greatly developed, especially in the context of dispersive PDEs. 
In particular, we signal quasi-invariance results for the BBM and Benjamin-Ono equations \cites{tzvetkov_2015, GLT1,GLT2}, KdV type equations~\cite{PTV2}, wave equations \cites{OTz2, GOTW, STX}, and Schr\"odinger equations \cites{OTz, OST, OTT, PTV, FT, DT2020, OS, FS, GLT1, forlano2022quasiinvariance}. 

In the results above, the dispersive nature of the equation plays a significant role in the proof, and indeed, if one considers the ODE 
$$i\partial_t u = |u|^2u, $$
it is possible to show that quasi-invariance fails in a dramatic fashion, and $(\Phi_t)_{\#}\mu_s \not\ll  \mu_s$, see \cites{OST}.

However, the relationship between the dispersive nature of the equation and quasi-invariance is not fully understood. In particular, when the regularity of the initial data (distributed according to the measure $\mu_s$) becomes low enough, it is unclear if quasi-invariance is preserved or not. 
On the one hand, from the study of \emph{invariant} measures, one could expect that the evolution eventually becomes singular with respect to the underlying Gaussian measure.\footnote{
Famously, the $\Phi^4_3$ measure of quantum field theory
$$ \Phi^4_3 \sim \exp\Big( -\frac14 \int u^4 -\frac12 \int u^2 - \frac12 \int |\nabla u|^2\Big) $$
is singular with respect to the underlying Gaussian measure, and it is (formally) an invariant measure for the Schr\"odinger equation on $\T^3$ 
$$i\partial_t u -\Delta u - u^3 = 0, $$
and in \cite{BDNY}, it has been shown to be invariant for the wave equation on $\T^3$ (up to renormalisation)
$$ \partial_t^2 u -\Delta u - u^3 = 0.$$} On the other hand, no such example is available in the literature, and only positive results are known. 

Here lies our interest into \eqref{1}. This
equation was introduced by G\'{e}rard-Grellier in \cite{ASENS_2010_4_43_5_761_0}, as a toy model for non-dispersive Hamiltonian dynamics. 
More precisely, this equation informally serves as a toy model for the ``$0$-dispersion" limit of fractional NLS/ half-wave equation
$$ i\partial_t u - (-\Delta)^\frac12 u = |u|^2u, $$
and similar dispersionless models.
Indeed, equation \eqref{1} shares many features with such models. To begin with, it is formally a Hamiltionian equation on the phase space $L^2_+=\Pi(L^2(\T))$,\footnote{Hence our choice $u_0 = \Pi(u_0)$.} with Hamiltonian given by 
$$ E(u) = \|u\|_{L^4(\T)}^4.$$
Moreover, it has a number of conserved quantities, including the $L^2$-norm and the $\dot H^\frac12$ norm of the solutions, i.e.\ 
\begin{equation*}
M(u)=\|u\|_{L^2(\bb{T})}^2,\quad P(u)=\|u\|_{\dot H^\frac12}^2=: \int_\T u (-\Delta)^\frac12 \cj{u}.
\end{equation*}
In \cite{ASENS_2010_4_43_5_761_0}, G\'erard and Grellier showed that this equation admits a Lax pair, which in particular implies 
that the equation is completely integrable and the existence of infinitely many conservation laws. 
Exploiting these conservation laws, in \cite{GeKo17}, G\'erard and Koch showed that equation \eqref{1} is globally well-posed in the space $\BMO = \Pi(L^\infty(\T)).$ 

From a statistical point of view, the conservation of $M$ and $P$ above suggest that the white-noise measure $\mu_0$ and (a variant of) the measure 
$\mu_{\frac12}$ should be invariant (and hence quasi-invariant) for the flow of \eqref{1}. While these measures are concentrated on distributions of too low regularity to be able to show strong invariance statements, 
a weak form of invariance for $\mu_{\frac12}$ has been proven in \cite{AFST_2018_6_27_3_527_0}. For $s>\frac12$, one has that $\mu_s(\BMO) = 1$, and so the flow map $\Phi_t$ and the push-forward measure $(\Phi_t)_{\#}\mu_s$ are well-defined for every $t\in \R$. 
In view of the conservation laws for the Szeg\H{o} equation and of the analogous results for dispersive equations, one might conjecture that $(\Phi_t)_{\#}\mu_s \ll \mu_s$.

On the other hand, as discussed, equation \eqref{1} enjoys no dispersion nor multilinear smoothing, even in the probabilistic setting \cite{2011}. 
Consequently, one could expect its behaviour to be similar to the one of the ODE
$$ i\partial_t u = |u|^2u, $$
for which we recall that no measure $\mu_s$ (for $s>\frac12$) is quasi-invariant, see \cite{OST}. 

The main goal of this paper is to provide a clear picture of quasi-invariance and lack thereof for equation \eqref{1}, and develop a better understanding of when one should expect quasi-invariance to hold/fail (respectively). In particular, our main result is the following.
\begin{theorem}\label{thm:main}
    Let $s>\frac12$, and let $\Phi_t$ be the flow map of equation \eqref{1} on the space $\BMO$. We have the following. 
    \begin{enumerate}[label={\rm{(\roman*)}}]
        \item If $s>1$, then the measure $\mu_s$ is quasi-invariant. More precisely, for every $t \in \R$, we have that
        $(\Phi_t)_{\#}\mu_s \ll \mu_s.$
        \item If $\frac12 < s < 1$ and $s\neq \frac34$, then the transported measure $(\Phi_t)_{\#} \mu_s$ is singular with respect to $\mu_s$ for almost every time $t\in \R$. More precisely, there exists a countable set $\mathscr N_s \subset \R$, such that for every $t \in \R\setminus \mathscr N_s$,
        $(\Phi_t)_{\#}\mu_s \perp \mu_s $.
    \end{enumerate}
\end{theorem}
To the authors' knowledge, this is the first time that such a transition from quasi-invariance to singularity has been observed. 

\subsection{A heuristic for quasi-invariance and a strategy for singularity} In order to understand the numerology of our result in Theorem \ref{thm:main}, consider (as a toy model) the dynamical system on $\R^d$
\begin{equation} \label{ODE}
\begin{cases}
    \dot y = b(y), \\
    y(0) = x,
\end{cases}
\end{equation}
where $b: \R^d \to \R^d$ is a smooth vector field. For the purpose of this subsection, denote by $\Phi_t(x)$ the solution of \eqref{ODE}. 
For every probability measure $\nu_0$ on $\R^d$, 
the transported measure $\nu_t := (\Phi_t(x))_{\#} \nu_0$ will satisfy the Liouville's equation
\begin{equation}\label{Liouville}
    \partial_t \nu_t = - \div{(b\, \nu_t)}.
\end{equation}
When $\nu_0$ is a probability measure of the form 
$$ \nu_0 = \frac{1}{Z}\exp(-E(y)) dy $$
for some smooth function $ E: \R^d \to \R$ with $\exp(-E) \in L^1(\R^d)$, then we can solve \eqref{Liouville} for 
$$ f_t:= \log \frac{d\nu_t}{dx} - \log \frac{d\nu_0}{dx} $$
explicitly with the method of characteristics, and obtain that 
$$ f_t = E(x) -E(\Phi_{-t}(x)) + \int_0^t \div{(b)}\big(\Phi_{-t'}(x)\big) dt'. $$ 
For convenience, define 
\begin{equation}
    Q(x) = \left.\frac{d}{dt} E(y(t))\right|_{t=0} + \div{(b)},
\end{equation}
from which we obtain 
\begin{equation} \label{densityformula}
    \nu_t = \exp\Big(\int_0^t Q(\Phi_{-t'}(x)) dx \Big) \nu_0.
\end{equation}
One can now apply (formally) this formula to equation \eqref{1} and $\nu_0 = \mu_s$ in \eqref{musdef}, for which $E(u) = \|u\|_{H^s}^2, $ obtaining that 
$$ (\Phi_t)_{\#}\mu_s(u_0) = \exp\Big( \int_0^t Q\big(\Phi_{-t'}(u_0)\big) dt' \Big)\mu_s(u_0),  $$
where 
\begin{equation} \label{Qdef}
\begin{aligned}
        Q(u) &:= 2\Im\Big(\int \Pi(|u|^2 u) \cj{\jb{\nabla}^{2s}u}\Big) \\
        &= \frac i 2\sum_{n_1-n_2+n_3-n_4 = 0, n_j \ge 0 } \big(\jb{n_1}^{2s} - \jb{n_2}^{2s} + \jb{n_3}^{2s} - \jb{n_4}^{2s}\big) \ft{u}(n_1)\cj{\ft{u}(n_2)}\ft{u}(n_3)\cj{\ft{u}(n_4)},
\end{aligned}
\end{equation}
and the last equality is obtained by exploiting Plancherel and the symmetry between the Fourier coefficients of $u$ and $\cj u$.
Then, a natural benchmark to test quasi-invariance is the following: denoting by $S(t)$ the propagator for the \emph{linear} flow for \eqref{1}, is it true that 
$$\E_{\mu_s}\Big| \int_0^t Q\big(S(-t')(u_0)\big) dt' \Big|^2 < \infty \,? $$
In the case of equation \eqref{1}, one has that $S(t)$ is the identity, 
and this condition translates to 
$$ \E_{\mu_s}\big|Q(u_0)\big|^2< \infty. $$
Exploiting the cancellations of the symbol $\big(\jb{n_1}^{2s} - \jb{n_2}^{2s} + \jb{n_3}^{2s} - \jb{n_4}^{2s}\big)$, one can check that for $s>\frac12$, this is the case if and only if $s> 1$,
which is exactly the regime in which we prove quasi-invariance in Theorem \ref{thm:main}. 
Actually, the finiteness of this quantity is what allows us to perform the proof of quasi-invariance in this regime. We essentially rely on the strategy of \cite{PTV} adapted to the Szeg\H{o} equation, together with some recent technology (namely, the variational formula \eqref{BDtoest}) in order to estimate 
$$ \int \exp(TQ(u_0)) d\mu_s(u_0).$$
In this paper, we re-introduce the strategy of \cite{PTV} starting from a different point of view, that 
relies on the formula \eqref{densityformula} in a more crucial way, see the proof of Proposition \ref{prop:fNLp}.\footnote{A similar argument appeared also in \cite{AmFi}.} 
While in the particular case of equation \eqref{1}, the results we obtain are completely equivalent, we believe that the argument presented in this paper is far more flexible, and indeed it was necessary to exploit some of this flexibility in the work by Forlano and the second author \cite{forlano2022quasiinvariance}.

We now move to discussing the case $s<1$, where instead we have $\E_{\mu_s}\big|Q(u_0)\big|^2 = \infty. $
As it turns out, the heuristic above not only suggests that singularity should hold, but also a strategy for its proof. Indeed, if two probability measures $\nu, \nu'$ are mutually singular, 
there exists a set $E$ such that 
$\nu(E) = 1$, $\nu'(E^c) =1$. Therefore, if we formally write 
$$ f = \frac{d\nu'}{d\nu}, $$
we must have that 
$$ \int_E f d \nu = \nu'(E) = 0, $$
which implies that $f = 0$ $\nu$-almost surely. By flipping the roles of $\nu,\nu'$, we also get 
$$ \frac{1}{f} = 0 \hspace{5pt} \nu'-\text{a.s.\ }\Leftrightarrow f = + \infty \hspace{5pt}\nu'-\text{a.s.} $$
In the case of the Szeg\H{o} equation, since (formally) $\div(b) =0$, we can rewrite \eqref{densityformula} as 
$$ (\Phi_t)_{\#}\mu_s(u_0) = \exp\Big( \| u_0\|_{H^s}^2 - \| \Phi_{-t}(u_0)\|_{H^s}^2 \Big)\mu_0(u_0). $$
Therefore, the conditions above can be rewritten as
\begin{gather*}
\| u_0\|_{H^s}^2 - \| \Phi_{-t}(u_0)\|_{H^s}^2 = - \infty \hspace{5pt} \mu_s-\text{a.s.\ },
\end{gather*}
and
\begin{gather*}
\| u_0\|_{H^s}^2 - \| \Phi_{-t}(u_0)\|_{H^s}^2 = + \infty \hspace{5pt} (\Phi_t)_\#\mu_s-\text{a.s.\ } \Leftrightarrow 
\| \Phi_{t}(u_0)\|_{H^s}^2 -\| u_0\|_{H^s}^2  = + \infty \hspace{5pt} \mu_s-\text{a.s.}
\end{gather*}
In other words, singularity should correspond exactly to the fact that for a generic $u_0$ distributed according to $\mu_s$, then $t=0$ is a (local) \emph{minimum} for the (infinite) quantity $\|\Phi_t(u_0)\|_{H^s}^2$. 
This suggests the following strategy for showing singularity. 
\begin{enumerate}
    \item Pick a well-defined approximation $R_N(u_0)$ of $ \|u_0\|_{H^s}^2$.
    \item Show that for some $c>0$, $\alpha >0$,
    $$\left.\frac{d^2}{dt^2} R_N(\Phi_t(u_0))\right|_{t=0} = cN^\alpha(1+o(1))$$ 
    for $\mu_s$-a.e.\ $u_0$, and that 
    $$ \left.\frac{d}{dt} R_N(\Phi_t(u_0))\right|_{t=0} = o(N^\alpha). $$
    \item Exploiting the local theory for equation \eqref{1}, show that for $t\ll 1$,
    $$ R_N(\Phi_t(u_0)) - t\left.\frac{d}{dt} R_N(\Phi_t(u_0))\right|_{t=0} - \frac{t^2}{2} \left.\frac{d^2}{dt^2} R_N(\Phi_t(u_0))\right|_{t=0} \ll t^2 N^\alpha,   $$
    which formally corresponds to the fact that $t=0$ is a local minimum for 
    $$ \lim_{N \to \infty} N^{-\alpha} R_N(\Phi_t(u_0)).$$
    \item Since functions on $\R$ can have at most countably many local minimum points, deduce that for up to countably many times, $(\Phi_t)_{\#}\mu_s \perp \mu_s$.
\end{enumerate}
This list of steps is essentially the strategy that we perform in Section \ref{sec:4} in order to show Theorem \ref{thm:main}, (ii).

\noi
More precisely, we first show an abstract formulation of (4) in Proposition \ref{sing_abstract}.

\noi
In Section \ref{sec:time_der}, we then consider an approximation $R_N$ of the square of the $H^1$ norm,\footnote{Choosing the $H^1$ norm instead of the $H^s$ norm makes no difference in the subsequent steps, but some of the formulas become slightly easier.} and proceed to estimating  
\begin{gather*}
    \Bigg|\left.\frac{d}{dt} R_N(\Phi_t(u_0))\right|_{t=0}\Bigg| \les N^{2-2s}\\
    \left.\frac{d^2}{dt^2} R_N(\Phi_t(u_0))\right|_{t=0} \sim N^{4-4s}
\end{gather*}  
This is where our restrictions on $s$, namely $s\neq 1, \frac34$, originate from. Indeed, we have that 
\begin{itemize}
    \item For $s=1$, both $\left.\frac{d}{dt} R_N(\Phi_t(u_0))\right|_{t=0}$ and $\left.\frac{d}{dt^2} R_N(\Phi_t(u_0))\right|_{t=0}$ are $O(1)$, and similarly the error terms coming from (3) are also of the same size,
    \item When $s=\frac34$, the quantity 
    $$\left.\frac{d^2}{dt^2} R_N(\Phi_t(u_0))\right|_{t=0} $$
    transitions from being positive when $\frac 34 < s < 1$, to being \emph{negative} when $\frac12 < s<\frac 34$. When $s= \frac34$, the sign of this quantity is unclear, but in any case it is of smaller order than $N^{4-4s}$, which is the minimum of what the rest of the argument can handle.
\end{itemize}

\noi
We then proceed to estimate the error of the term 
$$ R_N(\Phi_t(u_0)) - t\left.\frac{d}{dt} R_N(\Phi_t(u_0))\right|_{t=0} - \frac{t^2}{2} \left.\frac{d^2}{dt^2} R_N(\Phi_t(u_0))\right|_{t=0}. $$
This requires a very precise decomposition of the solution as $u(t) = X(t) + Y(t)$, where $X(t)$ is a semi-explicit, pseudo-Gaussian term of regularity $s-\frac12$ (the same as a typical data sampled according to $\mu_s$), and $Y(t)$ has regularity $2s-1$. We perform this in Section \ref{sec:paradec}. 
This decomposition is closely related to the ones used by Bringmann in \cite{Bring0} and Deng, Nahmod and Yue in \cites{DNY1,DNY2}. 
However, we point out that the main properties that we need to exploit are completely determinstic, and are not improved by the random structure of the initial data $u_0$. 

Finally, in Section \ref{sec:estimates}, we use this decomposition in order to show a series of (sharp) error bounds. 
We note that the only term that does not have a satisfactory deterministic bound is \eqref{QNrandom}, 
for which a random estimate is necessary. 

\begin{remark}
    It might be possible that, in the case $s=1$, performing sharper estimates and choosing a better approximation $R_N$ of $\|u_0\|_{H^1}^2$, would allows us to conclude singularity also in this case. Similarly, it might be possible to show that for $s=\frac34$,
$$ \left.\frac{d^4}{dt^4} R_N(\Phi_t(u_0))\right|_{t=0} \sim c N^{4-4s}  $$
for some $c\neq 0$. However, this would require substantial more work and ideas in Section \ref{sec:estimates} and \ref{sec:time_der} respectively, and at this stage, it is unclear if such results are true.
\end{remark} 

\begin{remark}\label{remark:ODE}
As a sanity check, we can also test the heuristic above for the ODE 
$$ i\partial_t u = |u|^2u, $$
for which we know that 
quasi-invariance fails for every $s>\frac12$, see \cite{OST}. In this case, the condition above for quasi-invariance translates into 
$$\E_{\mu_s}\Big| \sum_{n_1-n_2+n_3-n_4 = 0 } \big(\jb{n_1}^{2s} - \jb{n_2}^{2s} + \jb{n_3}^{2s} - \jb{n_4}^{2s}\big) \ft{u}(n_1)\cj{\ft{u}(n_2)}\ft{u}(n_3)\cj{\ft{u}(n_4)}\Big|^2 < \infty, $$
where the only difference with respect to \eqref{1} is that we dropped the condition $n_j\ge 0$ from the sum.
For this sum, we have that 
\begin{align*}
    &\E_{\mu_s}\Big| \sum_{n_1-n_2+n_3-n_4 = 0 } \big(\jb{n_1}^{2s} - \jb{n_2}^{2s} + \jb{n_3}^{2s} - \jb{n_4}^{2s}\big) \ft{u}(n_1)\cj{\ft{u}(n_2)}\ft{u}(n_3)\cj{\ft{u}(n_4)}\Big|^2 \\
    &\sim \sum_{n_1-n_2+n_3-n_4 = 0 } \frac{\big(\jb{n_1}^{2s} - \jb{n_2}^{2s} + \jb{n_3}^{2s} - \jb{n_4}^{2s}\big)^2}{\jb{n_1}^{2s}\jb{n_2}^{2s}\jb{n_3}^{2s}\jb{n_4}^{2s}} \\
    &\gtrsim \sum_{n_1=-n_3, n_2=n_4 = 0}  1 \\
    &= \infty.
\end{align*}
Therefore, according to our heuristic, we do not expect quasi-invariance to hold for any $s>\frac12$, which is indeed the case.
\end{remark}

\subsection{Some reductions}
In practice, while working with the Szeg\H{o} equation, we will need to keep track of the regularity of solutions emanating from an initial data sampled according to $\mu_s$. Consequently, we will not be able to work in the full space $\BMO$, where the global flow for \eqref{1} is defined, but we will rather have to consider some space $X_s \subseteq \BMO$.

Recall from \cite{ASENS_2010_4_43_5_761_0} that \eqref{1} is globally well-posed in the space $H^{\s}_+$ for every $\s > \frac12$. In view of Lemma \ref{u0Besov}, we have that for $\mu_s$-a.e.\ initial data $u_0$, we have $u_0 \in H^\s_+$ for every $\s < s -\frac12$. 
Therefore, when $s > 1$, one recovers a good global well-posedness statement by referring to well known results. 

When $s\le 1$, the situation is a bit more complicated. Indeed, while global well-posedness in $L^2_+$ has recently been shown in \cite{GP23}, 
it is still unclear if the equation \eqref{1} is well-posed in $H^\s_+$ for any $0 < \s < \frac 12$. 
It is indeed possible (in principle) that for a general initial data in $H^\s\cap \BMO$, the solution will belong to a space $H^{\s(t)}_+$ for some $\s(t)$ which is strictly decreasing in time, see \cite{GeKo17}. On the other hand, for the purpose of Section \ref{sec:4}, we need to keep track of the optimal regularity of the solution. The considerations above lead to the following choice for $X_s$.
\begin{equation}\label{Xsdef}
    X_s = 
    \begin{cases}
        H^{\s}_{+} \text{ for any } \frac 12 < \s < s-\frac 12, &\text{ when }s>1,\\
        B^{s-\frac12,+}_{p, \infty} \text{ for any } p > \min\big(100, \frac{1}{s-\frac12}\big), &\text{ when } \frac 12 < s \le 1,
    \end{cases}
\end{equation}
where $B^{s-\frac12,+}_{p, \infty}$ is a Besov space for functions supported on non-negative frequencies, defined in Section \ref{sec:prelim}. In view of Lemma \ref{u0Besov}, this space is well-adapted to the Gaussian measure $\mu_s$. Moreover, it satisfies the following.

\begin{proposition}\label{LWP}
Let $s > \frac12$, then the Banach space $X_s\subset \BMO$ defined above satisfies $\mu_s(X_s) = 1$ and the equation \eqref{1} is locally well posed in $X_s$.
\end{proposition}
Unfortunately, the restriction to the space $X_s$ makes it unclear if a generic $X_s \subset \BMO$ solution $\Phi_t(u_0)$ actually belongs to $X_s$ for infinite time. 
What is worse, is that in principle it is possible that the set of times in which the solution belongs to $X_s$, i.e.\ 
$$ \mathcal G(u_0) :=\{ t\in \R: \Phi_t(u_0) \in X_s \} $$
can be \emph{disconnected}. 

To deal with these (potential) issues, we need to introduce some notation. For $t \in \R$, we denote by $\WP(t)$ the set of ``good" initial data for which the solution to \eqref{1} exists in $X_s$ for every time $\tau$ between $0$ and $t$, or more precisely 
\begin{align*}
 \WP(t) :=
\{ u_0 \in X_s: \text{the equation \eqref{1} admits a solution } u \in C([0,t], X_s) \text{ with } u(0)=u_0\} 
\end{align*}
when $t \ge 0$, and 
\begin{align*}
 \WP(t) :=
\{ u_0 \in X_s: \text{the equation \eqref{1} admits a solution } u \in C([t,0], X_s) \text{ with } u(0)=u_0\} 
\end{align*}
when $t \le 0$. By the local well posedness in the space $X_s$, and the global well posedness in the space $\BMO \supseteq X_s$, we must have that for $u_0 \in\WP(t)$, the (local) solution in $X_s$ must coincide with $\Phi_t(u_0)$ .
Moreover, by Proposition \ref{LWP}, we obtain that 
$$ \Phi_t : X_s \supseteq \WP(t) \to X_s $$
is a continuous map, and that $\WP(t)\subseteq X_s$ is an open set for every $t \in \R$.

Then, the main results of Sections \ref{sec:3}, \ref{sec:4} are the following (respectively).

\begin{proposition}\label{thm:QI} 
Let $s>1$. 
Then for every $t \in \R$, we have that $\WP(t) = X_s$. Moreover, 
we have that
$$ (\Phi_t)_\# \mu_s \ll \mu_s  $$
for every $t \in \R$.
\end{proposition}
\begin{proposition}\label{thm:sing}
Let $\frac 12 < s < 1$, with $s \neq \frac 34$. Then there exists a countable set $\mathscr N \subseteq \R$ such that for every $t \in \R \setminus \mathscr N$ with $\mu_s(\WP(t)) > 0$,
$$ (\Phi_t)_\# (\1_{\WP(t)} \mu_s) \perp \mu_s.$$
\end{proposition}

It is clear that Proposition \ref{thm:QI} immediately implies Theorem \ref{thm:main}, (i). 
The situation in the singular case is more unclear, but it can be proven via a soft argument that Proposition \ref{thm:sing} automatically implies Theorem \ref{thm:main}, (ii). We will show this at the end of Section \ref{sec:4}.

\section{Preliminaries}\label{sec:prelim}

\subsection{Notation}
To fix notation, we will always interpret $N,M$ and $N_j$ as dyadic integers (and sums over such as just sums over dyadic integers), and given a family $N_1,...,N_k$ of dyadic integers, we denote $N^{(1)}\geq N^{(2)}\geq...\geq N^{(k)}$ to be a decreasing arrangement of the $N_j$. We take as notation $M\ll N$ to mean $M< 2^{-20}N$, $M\gtrsim N$ to mean $M\geq 2^{-20}N$, $M\sim N$ to mean $M\gtrsim N\gtrsim M$, and $N\approx M$ to mean $M/4\leq N\leq 4M$. This (unconventional) choice of symbols will help with probabilistic decoupling of different frequency scales.

When $A,B$ are not dyadic numbers, we will instead use the notation $A\les B$ to denote that there exists a constant $C \in (0,+\infty)$ such that $A \le C B$. 

We also take the convention $\bb{T}=\bb{R}/(2\pi\bb{Z})$ with normalized integral
\begin{equation*}
\int_{\bb{T}}f(x)dx\vcentcolon=\frac{1}{2\pi}\int_0^{2\pi}f(x)dx,
\end{equation*}
and we define the Fourier transform of a function $f$ to be 
$$\widehat{f}(n)\vcentcolon=\mathcal{F}[f](n)\vcentcolon=\int_{\bb{T}}f(x)e^{-inx}dx,$$
where $n \in \Z$.
\hfill\break

\subsection{Function spaces}\label{sec:spaces}

Let $s \in \R$ and $1 \leq p \leq \infty$.
We define the $L^2$-based Sobolev space $H^s(\T)$
by the norm:
\begin{align*}
\| f \|_{H^s} = \| \jb{n}^s \ft f (n) \|_{\ell^2_n},
\end{align*}
and the space restricted to non-negative frequencies $$H^{s}_+=\Pi(H^s(\T)).$$

\noi

Let $\varphi:\R \to [0, 1]$ be a smooth  bump function supported on $[-\frac{8}{5}, \frac{8}{5}]$ 
and $\varphi\equiv 1$ on $\big[-\frac 54, \frac 54\big]$.
We set $\phi_1(\xi) = \varphi(|\xi|)$
and 
\begin{align}
\phi_{N}(\xi) = \varphi\big(\tfrac{\xi}{2^j}\big)-\varphi\big(\tfrac{\xi}{2^{j-1}}\big)
\label{phi1}
\end{align}

\noi
where $N=2^j$, and $j \in \N$.
Then, for $N=2^j$ for $j \in \N_0 := \N \cup\{0\}$, 
we define  the Littlewood-Paley projector  $P_N$ 
as the Fourier multiplier operator with a symbol $\phi_N$.
Note that we have 
\begin{align*}
\sum_{N\in 2^{\N_0}} \phi_N (\xi) = 1
\end{align*}

\noi
 for each $\xi \in \R$. 
Thus, 
we have 
\[ f = \sum_{N\in 2^{\N_0}} P_N f. \]

We will also denote by $\phi_{\ll N}, \phi_{\les N}, \phi_{\sim N}, \phi_{\approx N}, \phi_{\gtrsim N}, \phi_{\gg N}$ the multipliers
\begin{align*}
\phi_{\ll N} = \sum_{M \ll N} \phi_M, && \phi_{\les N} = \sum_{M \les N} \phi_M, && \phi_{\sim N} = \sum_{M \sim N} \phi_M,  \\
\phi_{\approx N} = \sum_{M \approx N} \phi_M,  &&\phi_{\gtrsim N} = \sum_{M \gtrsim N} \phi_M, && \phi_{\gg N}= \sum_{M \gg N} \phi_M,
\end{align*}
and $P_{\ll N}, P_{\les N}, P_{\sim N}, P_{\approx N}, P_{\gtrsim N}, P_{\gg N}$ will denote the corresponding Fourier multiplier operators.
\noi

Next, we  recall the basic properties of the Besov spaces $B^s_{p, q}(\T)$
defined by the norm:
\begin{equation*}
\| u \|_{B^\s_{p,q}} = \Big\| N^\sigma \| P_{N} u \|_{L^p_x} \Big\|_{\ell^q_N(2^{\N_0})}.
\end{equation*}

\noi

Essentially, $\s$ denotes the regularity of a function $u \in B^\s_{p,q}$, and $p$ the integrability of such a function, with the extra parameter $q$ allowing for some extra flexibility at a ``logarithmic" scale.  \\
We shall also denote
\begin{equation*}
B^{\sigma,+}_{p,q}\vcentcolon=\Pi\Big(B^\sigma_{p,q}(\bb{T})\Big),
\end{equation*}
i.e. the Besov space on non-negative frequencies. Note that, from the boundedness of the Hilbert transform of $L^p(\T)$ for $1<p<\infty$, we also have that $B^{\sigma,+}_{p,q} \subset B^{\sigma}_{p,q} $
whenever $p\neq 1,\infty$.

We recall the basic estimates in Besov spaces.
See \cites{BCD, GOTW} for example, for \rm{(i) -- (iii)}, while \rm{(iv)} is a simple consequence of the well-known boundedness of the Hilbert transform on $L^p$ for $p \in (1,\infty)$. 

\begin{lemma}\label{LEM:Bes}
The following estimates hold.

\begin{enumerate}
\item[\rm(i)] Let $s_1, s_2 \in \R$ and $p_1, p_2, q_1, q_2 \in [1,\infty]$.
Then, we have
\begin{align} 
\begin{split}
\| u \|_{B^{s_1}_{p_1,q_1}} 
&\les \| u \|_{B^{s_2}_{p_2, q_2}} 
\qquad \text{for $s_1 \leq s_2$, $p_1 \leq p_2$,  and $q_1 \geq q_2$},  \\
\| u \|_{B^{s_1}_{p_1,q_1}} 
&\les \| u \|_{B^{s_2}_{p_1, \infty}}
\qquad \text{for $s_1 < s_2$},\\
\| u \|_{B^0_{p_1, \infty}}
 &  \les  \| u \|_{L^{p_1}}
 \les \| u \|_{B^0_{p_1, 1}}.
\end{split}
\label{embed}
\end{align}
\item[\rm(ii)] Let $1\leq p_2 \leq p_1 \leq \infty$, $q \in [1,\infty]$,  and  $s_2 \ge s_1 + \big(\frac{1}{p_2} - \frac{1}{p_1}\big)$. Then, we have
\begin{equation} \label{Sobolev}
 \| u \|_{B^{s_1}_{p_1,q}} \les \| u \|_{B^{s_2}_{p_2,q}}.
\end{equation}
\item[\rm(iii)] Let $p, p_1, p_2, p_3, p_4 \in [1,\infty]$ such that 
$\frac1{p_1} + \frac1{p_2} 
= \frac1{p_3} + \frac1{p_4} = \frac 1p$. 
Then, for every $\s>0$, we have
\begin{equation}
\| uv \|_{B^{\s}_{p,q}} \les  \| u \|_{B^{\s}_{p_1,q}}\| v \|_{L^{p_2}} + \| u \|_{L^{p_3}} \| v \|_{B^\s_{p_4,q}} .
\label{prod}
\end{equation}
\item[\rm(iv)] For $p\in(1,\infty)$, 
\begin{equation}\label{projectionbound}
\|\Pi(u)\|_{L^p}\les \|u\|_{L^p}.
\end{equation}
\end{enumerate}
\end{lemma}

In particular, in view of \eqref{prod}, \eqref{Sobolev} and \eqref{embed} we have that whenever $p>\frac1\s$, the space $B^{\s}_{p,q}$ has the algebra property, i.e.\ 
\begin{equation} \label{algebra}
    \| fg\|_{B^{\s}_{p,q}} \les \| f\|_{B^{\s}_{p,q}} \| g\|_{B^{\s}_{p,q}}.
\end{equation}
Combining with \eqref{projectionbound}, we have for $\max(1,\frac1\s)<p<\infty$, the space $B^{\s,+}_{p,q}$ has the following algebra property:
\begin{equation} \label{algebra+}
    \| \Pi(fg)\|_{B^{\s,+}_{p,q}} \les \| f\|_{B^{\s,+}_{p,q}} \| g\|_{B^{\s,+}_{p,q}}.
\end{equation}
Next we recall the Coifman-Meyer multiplier theorem, which we shall use extensively. See \cite{CK}*{Theorem 3.3}.
\begin{proposition}\label{prop:CM}
Let $R\geq 1$ and let $m:\bb{R}^d\to \bb{C}$ be smooth away from the origin such that
\begin{equation*}
\|m\|_{CM,d}\vcentcolon=\sup_{|\alpha|\leq d(d+3)}\sup_{x\in\bb{R}^d}\|x\|^{|\alpha|}|\partial^\alpha m(x)|\leq R<\infty.
\end{equation*}
Consider the multilinear map given by
\begin{equation*}
\mathcal{F}[T_m(f_1,...,f_d)](n)=\sum_{\substack{n_1,...,n_d\in\bb{Z} \\ n_1+...n_d=n}}m(n_1,...,n_d)\prod_{j=1}^d\widehat{f_j}(n_j).
\end{equation*}
Then for all $p\in[1,\infty)$ and $p_1,...,p_d\in(1,\infty]$ with $\frac{1}{p}=\frac{1}{p_1}+...+\frac{1}{p_d}$, there is a $C(R,p,p_i)<\infty$ such that
\begin{equation*}
\|T_m(f_1,...,f_d)\|_{L^p}\leq C(R,p,p_i)\prod_{j=1}^d\|f_j\|_{L^{p_j}}.
\end{equation*}
 for all $f_1,...,f_d$.
\end{proposition} 
We refer to $\|m\|_{CM,d}$ as the Coifman-Meyer norm of $m$. A key feature of this result is for any family of smooth multipliers $\{m_N\}_N$ with uniformly bounded Coifman-Meyer norms, we may choose a uniform operator bound $C(R,p,p_i)$ for the family $\{T_{m_N}\}_N$.


\subsection{On the Gaussian measures \texorpdfstring{$\mu_s$}{$\mu_s$}} 

We now define the Gaussian measures $\mu_s$ that we are going to consider throughout the paper, 
conditioned to be concentrated on spaces of functions which only have nonnegative frequencies.  

To this goal, for $s>\frac12$ we consider random initial data of the form
\begin{equation}\label{Gaussrepresentation}
    u_0(x)=\sum_{n\geq 0}\frac{g_n}{\jb{n}^s}e^{inx},
\end{equation}
where $\{g_n\}_{n\geq 0}$ are independent standard complex Gaussian random variables\footnote{i.e. $\Re(g_n) \sim \Im(g_n) \sim N(0,\frac12)$, and $\Re(g_n), \Im(g_n)$ are independent.}, and $\jb{n}=(1+n^2)^{1/2}$. We denote
\begin{equation*}
    \mu_s\vcentcolon=\text{Law}(u_0),
\end{equation*}
which is the unique centred Gaussian measure on $L^2_+$ with covariance operator $(1-\Delta)^{-s}$ (restricted to non-negative frequencies). This measure may be written formally as
\begin{equation*}
    d\mu_s=\frac{1}{Z_s}e^{-\|u\|_{H^s_+}^2}\Pi(u)du,
\end{equation*}
where $du$ is the formal Lebesgue measure on $L^2(\T)$. We recall for $\s<s-\frac12$, $$\mu_s\left(H^{\sigma}\setminus H^{s-\frac12}\right)=1,$$ and so one may view $u_0$ as random initial data of regularity $s-\frac12-$ (in Sobolev norm).

We now verify that $\mu_s$ is concentrated on a Besov space with critical regularity $s-\frac12$.
\begin{lemma} \label{u0Besov}
Let $s > \frac12$, and let $p\in[1,\infty)$. Then  $\mu_s(B^{s-\frac{1}{2},+}_{p,\infty}) = 1$.
\end{lemma}

\begin{proof}
We shall show the result for $p=2k$ for $k\in \bb{N}$, and the general case follows by interpolation.
    We note that
    \begin{equation*}
        \|u\|_{B^{s-\frac{1}{2},+}_{p,\infty}}=\sup_{N\in 2^{\bb{N}_0}} N^{s-\frac{1}{2}}\|P_Nu\|_{L^p},
    \end{equation*}
and so we define the random variable $X_N$ for $N\in 2^{\bb{N}_0}$ as
\begin{equation*}
    X_N=N^{2k\big(s-\frac{1}{2}\big)}\|P_Nu\|_{L^{2k}}^{2k},
\end{equation*}
where $u = u_0$ is given by \eqref{Gaussrepresentation}.
We see that
\begin{align*}
    X_N= &\ N^{k(2s-1)}\int_\bb{T} \Big|\sum_{n\in\N_0}\frac{\phi_N(n)g_n}{\jb{n}^s}e^{inx}\Big|^{2k} dx \\
    = &\ N^{k(2s-1)}\int_\bb{T} \sum_{\n,\m\in\N_0^k}\prod_{j=1}^k \frac{\phi_N(n_j)\phi_N(m_j)g_{n_j}\overline{g_{m_j}}}{\jb{n_j}^s\jb{m_j}^s}e^{i(n_j-m_j)x} dx \\
    = &\ N^{k(2s-1)}\sum_{\mathclap{\substack{ \n,\m\in\N_0^k \\ n_1+...+n_k=m_1+...+m_k}}} C_{\n,\m} \prod_{j=1}^k g_{n_j}\overline{g_{m_j}},
\end{align*}
where $\n = (n_1, \dotsc, n_k)$, $\m= (m_1,\dotsc, m_k)$, and  
\begin{equation*}
    C_{\n,\m}=\prod_{j=1}^k\frac{\phi_N(n_j)\phi_N(m_j)}{\jb{n_j}^s\jb{m_j}^s}.
\end{equation*}
    We note that 
\begin{align*}
    C_{\n,\m}&=0, \text{  for  } (\n,\m)\notin [N/2,2N]^{2k}, \\
    C_{\n,\m} &\les \frac{1}{N^{2ks}}  \text{  for  } (\n,\m)\in [N/2,2N]^{2k}.
\end{align*}
Now, for $j=0,...,k$, let $E_j$ denote the set of $(\n,\m)$ such that $\n$ and $\m$ have exactly $j$ entries in common, i.e.
\begin{equation*}
    E_j=\Big\{(\n,\m)\in [N/2,2N]^{2k} : \max_{\sigma\in S_k} \Big[\# \{l \text{ s.t. } n_l=m_{\sigma(l)}\}\Big]=j \Big\},
\end{equation*}
where $S_k$ denotes the set of permutations of $\{1,\dotsc,k\}$.  We define
\begin{equation*}
    Y_{N,j}=N^{k(2s-1)}\sum_{\mathclap{\substack{ (\n,\m)\in E_j \\ n_1+...+n_k=m_1+...+m_k}}} C_{\n,\m} \prod_{j=1}^k g_{n_j}\overline{g_{m_j}}.
\end{equation*}
Since $[N/2,2N]^{2k}=\bigsqcup_{j=0}^k E_j$, we have that
\begin{equation}\label{XNdec}
    X_N=\sum_{j=0}^kY_{N,j}.
\end{equation}
Since the $\{g_j\}$ are independent standard complex Gaussians, we observe that
\begin{equation} \label{Ecancellation}
    \bb{E}\Big[\prod_{j=1}^k g_{n_j}\overline{g_{m_j}}\Big]=0, \text{ whenever } (\n,\m)\not\in E_k.
\end{equation}
Therefore, for $0\leq j<k$,
\begin{equation*}
    \bb{E}[Y_{N,j}]=0.
\end{equation*}
For convenience of notation, denote
$$\Delta =\Big\{(\n,\m)\in [N/2,2N]^{2k} : n_1+...+n_k=m_1+...+m_k\Big\}.$$ 
For $0\leq j<k$, we have that
\begin{align*}
    \bb{E}[Y_{N,j}^2]=N^{2k(2s-1)}\sum_{\mathclap{\substack{ (\n,\m)\in E_j\cap \Delta \\ (\l,\p)\in E_j\cap \Delta}}} C_{\n,\m}C_{\l,\p} \bb{E}\Big[\prod_{j=1}^k g_{n_j}g_{l_j}\cj{g_{m_j}}\cj{g_{p_j}}\Big].
\end{align*}
By similar considerations to the ones leading to \eqref{Ecancellation}, we have that  
$$ \bb{E}\Big[\prod_{j=1}^k g_{n_j}g_{l_j}\overline{g_{m_j}}\overline{g_{p_j}}\Big] = 0 \text{ unless } (\m,\p)=\sigma(\n,\l) \text{ for some } \sigma\in S_{2k},$$
and by H\"older, 
$$ \E\Big[\prod_{j=1}^k g_{n_j}g_{l_j}\cj{g_{m_j}}\cj{g_{p_j}}\Big] \le \E|g_{n_j}|^{4k} \les_k 1.  $$
For $\s \in S_{2k}$, denote 
$$\s_1(\n,\l) := (\s(\n,\l)_1,\dotsc, \s(\n,\l)_k), \quad \s_2(\n,\l) := (\s(\n,\l)_{k+1},\dotsc, \s(\n,\l)_{2k}).$$ 
We then obtain  
\begin{align*}
    \bb{E}[Y_{N,j}^2]= &\  N^{2k(2s-1)}\sum_{{\substack{ (\n,\m)\in E_j\cap \Delta \\ (\l,\p)\in E_j\cap \Delta}}} C_{\n,\m}C_{\l,\p} \bb{E}\Big[\prod_{j=1}^k g_{n_j}g_{l_j}\cj{g_{m_j}}\cj{g_{p_j}}\Big]\\
    \le&\ N^{2k(2s-1)}\sum_{\s \in S_{2k}} \sum_{{\substack{ (\n,\m)\in E_j\cap \Delta \\ (\l,\p)\in E_j\cap \Delta \\ (\m,\p)=\sigma(\n,\l)}}} C_{\n,\m}C_{\l,\p} \Big|\bb{E}\Big[\prod_{j=1}^k g_{n_j}g_{l_j}\cj{g_{m_j}}\cj{g_{p_j}}\Big]\Big| \\
    \les &\  N^{-2k}\sum_{\s \in S_{2k}} \sum_{\substack{ (\n,\s_1(\n,\l))\in E_j\cap \Delta \\ (\l,\s_2(\n,\l))\in E_j\cap \Delta }} 1\\
    =&\ 
N^{-2k} \Big|\{ (\s, \n, \l) \in S_{2k} \times [N/2,2N]^{k} \times [N/2,2N]^{k}: \\
&\phantom{N^{-2k} \Big|\{ (\s, \n, \l)\}}(\n,\s_1(\n,\l))\in E_j\cap \Delta, (\l,\s_2(\n,\l))\in E_j\cap \Delta\}\Big|    
\end{align*}
We now bound the size of this set. 

Firstly, we note that since $(\n,\s_1(\n,\l))\in E_j, (\l,\s_2(\n,\l))\in E_j$, there exist permutations $\tau_\n, \tau_\l,\tau_1,\tau_2 \in S_k$ such that 
\begin{gather*}
 \tau_\n(\n)_1 =  \tau_1(\s_1(\n,\l))_1, \dotsc,  \tau_\n(\n)_j =  \tau_1(\s_1(\n,\l))_j, \\
   \tau_\l(\l)_1 =  \tau_2(\s_2(\n,\l))_1, \dotsc,  \tau_\l(\l)_j =  \tau_2(\s_2(\n,\l))_j, \\
  \tau_\n(\n)_{j+1} = \tau_2(\s_2(\n,\l))_{j+1}, \dotsc, \tau_\n(\n)_k = \tau_2(\s_2(\n,\l))_{k}, \\
   \tau_\l(\l)_{j+1} =  \tau_1(\s_1(\n,\l))_{j+1}, \dotsc,  \tau_\l(\l)_k =  \tau(\s_1(\n,\l))_{k}. 
\end{gather*}
Namely, the permutations $\tau_\n, \tau_\l,\tau_1,\tau_2 \in S_k$ put the common terms between $\n$ and $\s_1(\n,\l)$ and between $\m$ and $\s_2(\n,\l)$ in corresponding positions $1,\dotsc, j$, and the common terms between $\n$ and $\s_2(\n,\l)$ and between $\m$ and $\s_1(\n,\l)$ in corresponding positions $j+1,\dotsc, k$.

Since $|S_k|^4 = (k!)^4 \les_k 1$, we obtain that 
\begin{align*}
&\Big|\{ (\s, \n, \l) \in S_{2k} \times [N/2,2N]^{k} \times [N/2,2N]^{k}: 
(\n,\s_1(\n,\l))\in E_j\cap \Delta, (\l,\s_2(\n,\l))\in E_j\cap \Delta\}\Big|    \\
&\les_k \Big| \{ \n,\l \in [N/2,2N]^{k} \times [N/2,2N]^{k}: n_{j+1} + \dotsb + n_k = l_{j+1} + \dotsb + l_k\}\Big|\\
&\les_k N^{2k-1}.
\end{align*}
Therefore, we deduce that 
\begin{equation*}
    \bb{E}[Y_{N,j}^2]\les N^{-2k}N^{2k-1}=\frac{1}{N}.
\end{equation*}
Summing over $N\in 2^{\bb{N}_0}$, we obtain that for every $\eps > 0$, and every $j<k$,
\begin{equation}
 |Y_{N,j}| \le C_\eps N^{-(\frac12-\eps)} \label{YNj bdd}
\end{equation}
holds $\mu$-a.s., for a (random) constant $C_\eps < \infty$.

It remains to consider the $j=k$ case. Noting that $E_k$ is the set of $(\n,\m)\in [N/2,2N]^{2k}$ such that $\m$ is a permutation of $\n$, we see that
\begin{align*}
Y_{N,k} &= k! N^{k(2s-1)} \sum_{\n \in \Z^{k}} C_{\n,\n} \prod_{j=1}^k |g_{n_j}|^2\\
& \les N^{-k} \sum_{\n \in [N/2,2N]^{k}} \prod_{j=1}^k |g_{n_j}|^2 \\
& = \Big( N^{-1} \sum_{n \in [N/2,2N]} |g_{n}|^2\Big)^k =: (\wt {Y}_N)^k.
\end{align*}

\noi
We have that 
\begin{align*}
\E[\wt {Y}_N] = \E\Big[N^{-1} \sum_{n \in [N/2,2N]} |g_{n}|^2\Big] \les 1,
\end{align*}
and, by independence of the Gaussians $g_n, g_m$ for $n\neq m$, 
\begin{align*}
\E\Big[\big|\wt {Y}_N - \E[\wt {Y}_N]\big|^2\Big] &= N^{-2} \E\bigg[\Big|\sum_{n \in [N/2,2N]} |g_{n}|^2 - 1\Big|^2\bigg] \\
& = N^{-2} \E\Big[\sum_{n \in [N/2,2N]} \big(|g_{n}|^2 - 1\big)^2 \Big]\\
&\les N^{-1}.
\end{align*}
Therefore, we obtain that $\mu_s$-a.s., 
$$ \lim_{N \to \infty} \wt {Y}_N - \E[\wt {Y}_N] = 0, $$
and so for some deterministic $C = C(k) > 0$, we have that 
$$ \limsup_{N\to \infty} Y_{N,k} \le \limsup_{N \to \infty} (\wt {Y}_N)^k =  \limsup_{N \to \infty} (\E[\wt {Y}_N])^k \le C. $$
Therefore, together with \eqref{XNdec} and \eqref{YNj bdd}, we obtain that $\mu_s$-a.s.,
$$ \sup_{N} X_N < \infty, $$
which implies that $u_0\in B^{s-\frac{1}{2},+}_{p,\infty}$ $\mu_s$-a.s.
\end{proof}

\subsection{Local well posedness}
Recalling our choice of $X_s$:
\begin{equation}
    X_s = 
    \begin{cases}
        H^{\s}_{+} \text{ for any } \frac 12 < \s < s-\frac 12, &\text{ when }s>1,\\
        B^{s-\frac12,+}_{p, \infty} \text{ for any } p > \min\big(100, \frac{1}{s-\frac12}\big), &\text{ when } \frac 12 < s \le 1.
    \end{cases}
\end{equation}
We have the following local well-posedness statement.
\begin{proposition} \label{prop:LWP}
Let $s> \frac 12$. Then the equation \eqref{1} is locally well-posed in $X_s$. More precisely, for every $u_0 \in X_s$, there exists $T_\ast = T_\ast(\|u_0\|_{X_s}) > 0$ such that the \eqref{1} 
admits a unique solution $u$ belonging to the space 
$u \in C([0,T_\ast], X_s).$ Moreover, this solution satisfies 
$$ u \in C^\infty([0,T_\ast], X_s). $$
\end{proposition}
\begin{proof}
    The proof is a standard application of Banach fixed-point theorem, where we only need the algebra property, i.e.\ the boundedness of the maps  
    $$ H^\s \times H^\s \ni (f,g) \mapsto fg \in H^s \quad B^{s-\frac12}_{p, \infty}\times B^{s-\frac12}_{p, \infty} \ni (f,g) \mapsto fg \in B^{s-\frac12}_{p, \infty} $$
    respectively (for our choice of parameters), together with boundedness of 
    $$ \Pi: H^\s \to H^\s_+, \quad \Pi: B^{s-\frac12}_{p, \infty} \to B^{s-\frac12,+}_{p, \infty}. $$
    Note that the fact that the solution is going to be infinitely smooth in time follows from the (formal) identity 
    $$ i^k \partial_t^k u = \Pi(|u|^2\Pi(|u|^2 \dots\Pi(|u^2|u) \dots )),$$
    and the fact that the right-hand-side of this expression is bounded in $X_s$ due to the algebra property again.
\end{proof}

\section{Quasi-invariance: \texorpdfstring{$s>1$}{s>1}} \label{sec:3}

For $N\in\bb{N}$, we introduce the sharp Fourier truncation $\pi_N$ on $L^2_+$ by
\begin{equation*}
\pi_N\left(\sum_{n\geq 0}\widehat{u}(n)\right)=\sum_{0\leq n <N}\widehat{u}(n),
\end{equation*}
and the truncated flow
\begin{equation}\label{Truncated}
\begin{cases}
i\partial_tu_N=\pi_N\left(\Pi\left(|u_N|^2u_N\right)\right), \\
u(0)=u_0\in \pi_N(L^2_+).
\end{cases}
\end{equation}
We recall from \cite{ASENS_2010_4_43_5_761_0} that on the space $\pi_N(L^2_+)\cong \bb{R}^{2N}$ endowed with the symplectic form
\begin{equation*}
\omega(u,v)=4\text{Im}(u,v)_{L^2},
\end{equation*}
this truncated flow is the equation for a Hamiltonian equation associated with the Hamiltonian 
\begin{equation*}
E_N(u)=\|\pi_N(u)\|_{L^4}^4.
\end{equation*}
The Hamiltonian vector field $X_{E_N}(u)=\pi_N\left(\Pi\left(|u|^2u\right)\right)$ on $\pi_N(L^2_+)$ is smooth, and by the conservation of the Hamiltonian, we can define a global flow map $\Phi_{t,N}$ for (\ref{Truncated}) on this space. Moreover, using standard arguments, one can easily show the following approximation statement.
\begin{proposition} \label{prop:LWPconvergence}
    Let $\s > \frac12$, and let $u_0 \in H^{\s}$. Let $T \ge 0.$ Then for every $\s' < \s$,
    $$ \lim_{N\to\infty} \|\Phi_{t,N}(u_0) - \Phi_{t}(u_0)\|_{C([-T,T],H^{\s'})} = 0, $$
    and for every $R>0$, convergence is uniform on the $H^\s$ ball
    $$ B_R:= \{ u_0 \in H^{\s}: \|u_0\|_{H^{\s}} \le R\}. $$
\end{proposition}

We next define the truncated measure $$\mu_{s,N}\vcentcolon=(\pi_N)_\# \mu_s,$$ which is a Gaussian measure on $\pi_N(L^2_+)$, given by
\begin{equation*}
d\mu_{s,N}=\frac{1}{Z_{s,N}}e^{-\|\pi_N(u)\|^2_{H^s_+}}du\overline{du},
\end{equation*}
where $du\overline{du}$ is the natural Lebesgue measure on $\pi_N(L^2_+)\cong \C^N\cong \bb{R}^{2N}$. 
By Liouville's theorem, the Lebesgue measure is invariant under the flow map, and so we see that for all $t\in\bb{R}$:
\begin{align*}
d(\Phi_{t,N})_\#\mu_{s,N}=\exp\left(-\|\Phi_{-t,N}(u)\|^2_{H^s_+}+\|u\|^2_{H^s_+}\right)d\mu_{s,N},
\end{align*}
and so we define $f_{t,N}$ to be the density in the expression above, i.e.\ 
\begin{equation} \label{fNformula}
f_{t,N}\vcentcolon=\frac{d\Phi_{t,N}\#\mu_{s,N}}{d\mu_{s,N}}=\exp\left(-\|\Phi_{-t,N}(u)\|^2_{H^s_+}+\|u\|^2_{H^s_+}\right).
\end{equation}
By the fundamental theorem of calculus, we can write for $u\in\pi_N(L^2_+)$:
\begin{equation}\label{fNintegral}
-\|\Phi_{-t,N}(u)\|^2_{H^s_+}+\|u\|^2_{H^s_+}=-\int_0^{-t}Q_{\pi_N}(\Phi_{\tau,N}(u))d\tau,
\end{equation}
where (with an abuse of notation) we denote
\begin{equation*}
Q_{\pi_N}(u)=Q_{\pi_N}(u,u,u,u),
\end{equation*}
and $Q_{\pi_N}$ is the multi-linear map
\begin{equation*}
Q_{\pi_N}(u_1,u_2,u_3,u_4)\vcentcolon=\sum_{\substack{n_1-n_2+n_3-n_4=0 \\ 0\leq n_i<N}}\frac{i}{2}\Psi_s(\underline{\mathbf{n}})\widehat{u_1}(n_1)\overline{\widehat{u_2}(n_2)}\widehat{u_3}(n_3)\overline{\widehat{u_4}(n_4)},
\end{equation*}
with $\Psi_s$ being the multiplier
\begin{equation*}
\Psi_s({\n})=\jb{n_1}^{2s}-\jb{n_2}^{2s}+\jb{n_3}^{2s}-\jb{n_4}^{2s}.
\end{equation*}
We have the following bound on $\Psi_s$.
\begin{lemma}
    Let $n_j\in\N_0$ with $n_j \sim N_j$ and $n_1-n_2+n_3-n_4 =0$. We have
    \begin{equation}\label{Psisbd}
        |\Psi_s(\n)| \les (N^{(1)})^{2s-1} N^{(3)}. 
    \end{equation}
\end{lemma}
\begin{proof}
    Without loss of generality, we can assume $N_1\sim N_2 \sim N^{(1)}$. Indeed, if otherwise we had (say) $N_1\sim N_3 \sim N^{(1)} \gg N_2,N_4$, then the condition $n_j \ge 0$ prevents us from having $n_1-n_2+n_3-n_4 =0$. Therefore, from the mean value theorem, we obtain that 
    \begin{align*}
        |\Psi_s({\n})| & \le \big|\jb{n_1}^{2s}-\jb{n_2}^{2s}\big| + \jb{n_3}^{2s}+ \jb{n_4}^{2s} \\
        &\les (N^{(1)})^{2s-1}|n_1-n_2| + (N^{(3)})^{2s} + (N^{(4)})^{2s} \\
        &= (N^{(1)})^{2s-1}|n_3-n_4| + (N^{(3)})^{2s} + (N^{(4)})^{2s} \\
        &\les (N^{(1)})^{2s-1} N^{(3)}.
    \end{align*}
\end{proof}
\begin{remark}
    Recalling the calculations of Remark \ref{remark:ODE} for the ODE setting, we see the diverging term of $Q(u)$ is given by the interaction of $|n_1|,|n_3|\gg|n_2|,|n_4|$. By restricting to non-negative frequencies no such interaction can exist, from which we deduce this improved bound on $\Psi_s$ which we use to prove integrability of $Q$.
\end{remark}

We wish to show (local) exponential integrability (with respect to $\mu_{s,N}$) of $Q_{\pi_N}$, uniformly in $N$. To do this we first need a deterministic multi-linear bound on $Q_{\pi_N}$.
\begin{lemma} 
For all $p_1,p_2,p_3,p_4\in(1,\infty]$ with $\frac{1}{p_1}+\frac{1}{p_2}+\frac{1}{p_3}+\frac{1}{p_4}=1$, there is a constant $C(p_i,s)>0$, such that for all $N\in\bb{N}$ and $N_i$ dyadic,
\begin{equation} \label{QpiN_CM}
|Q_{\pi_N}(P_{N_1}u_1,P_{N_2}u_2,P_{N_3}u_3,P_{N_4}u_4)|\leq C(p_i,s) \left(N^{(1)}\right)^{2s-1}N^{(3)}\prod_{j=1}^4\|P_{N_j}u_j\|_{L^{p_j}},
\end{equation}
for all $u_j\in \pi_N(L^2_+)$.
\end{lemma}
\begin{proof}
Without loss of generality it suffices to consider the case $N_1=N^{(1)}$ and $N_1\approx N_2$. For $M\ge1$ dyadic, we let
\begin{equation*}
\Psi_{M,s}(\underline{\mathbf{n}})=\big(\phi_{\les M}(n_1)\jb{n_1}^{2s}-\phi_{\les M}(n_2)\jb{n_2}^{2s}+\phi_{\les M}(n_3)\jb{n_3}^{2s}-\phi_{\les M}(n_4)\jb{n_4}^{2s}\big) \prod_{j=1}^4 \phi_{\les M}(n_j),
\end{equation*}
and we write
\begin{equation*}
\Psi_{M,s}(\underline{\mathbf{n}})=(n_4-n_3)(n_2-n_3)m_M(n_1,n_2,n_3,n_4).
\end{equation*}
By the double mean-value theorem, $m_M$ corresponds to the multiplier
\begin{equation*}
m_M(n_1,n_2,n_3,n_4)=\prod_{j=1}^4 \phi_{\les M}(n_j) \int_0^1\int_0^1(\jb{\cdot}^{2s}\phi_M)''(n_3+s(n_2-n_3)+t(n_4-n_3))dsdt.
\end{equation*}
We easily see that  the Coifman-Meyer norm of $m_M$ is bounded by $M^{2s-2}$. Denote by $T_{m_M}$ the quadrilinear operator given by the multiplier $m_M$, i.e.
$$T_{m_M}(u_1,u_2,u_3,u_4) = \sum_{n_1-n_2+n_3-n_4 = 0} m_M(n_1,n_2,n_3,n_4) \ft{u_1(n_1)}\cj{\ft{u_2(n_2)}}\ft{u_3(n_3)}\cj{\ft{u_4(n_4)}}. $$
We calculate
\begin{align*}
&Q_{\pi_N}(P_{N_1}u_1,P_{N_2}u_2,P_{N_3}u_3,P_{N_4}u_4)\\
&= {\sum_{M\lesssim N_1}\sum_{\substack{n_1-n_2+n_3-n_4=0 \\ 0\leq n_i<N}}}\frac{i}{2}\Psi_{M,s}(\underline{\mathbf{n}})\widehat{P_{N_1}u_1}(n_1)\overline{\widehat{P_{N_2}u_2}(n_2)} 
\widehat{P_{N_3}u_3}(n_3)\overline{\widehat{P_{N_4}u_4}(n_4)} \\
&={\sum_{M\lesssim N_1}}\frac{i}{2}T_{m_M}(P_{N_1}u_1,i\partial_xP_{N_2}u_2,P_{N_3}u_3,i\partial_xP_{N_4}u_4) + \text{ similar terms} ,
\end{align*}
where each of the similar terms is obtained by writing each of the terms of $(n_4-n_3)(n_2-n_3) = n_2n_4 - n_2n_3 -n_3n_4 + n_3^2 $ as derivatives. Applying the Coifman-Meyer multiplier theorem on $m_M$ and using Bernstein's inequality yields (with a constant $C(p_i,s)$ that changes line-by-line):
\begin{align*}
|Q_{\pi_N}(P_{N_1}u_1,P_{N_2}u_2,P_{N_3}u_3,P_{N_4}u_4)|& \leq C(p_i,s) \sum_{M\lesssim N_1} M^{2s-2}N_1N^{(3)}\prod_{j=1}^4\|P_{N_j}u_j\|_{L^{p_j}} \\
&\leq C(p_i,s)N_1^{2s-1}N^{(3)}\prod_{j=1}^4\|P_{N_j}u_j\|_{L^{p_j}}.
\end{align*}
\end{proof}

From now on, we fix $s>1$, and $\frac12 < \s < s-\frac12$ as in the definition \eqref{Xsdef} of $X_s$, with the understanding that $\s$ is sufficiently close to $s-\frac{1}{2}$. Also, we let 
$$B_R := \{ u_0 \in H^{\s}: \|u_0\|_{H^{\s}} \le R\}.$$ 
\begin{proposition} \label{fNexpbound}
There exist constants $A(s,\sigma),B(s,\sigma),C(s,\sigma)>0$ such that for every $M> 0$, $R\geq 1$, and $N\in\N$ we have
\begin{equation*}
\int \exp\left(M|Q_{\pi_N}(u)|\right)\mathbbm{1}_{B_R(0)}(u)d\mu_{s,N}(u)\leq \exp(CM\max\{1,M^A\}(1+R)^B).
\end{equation*}
\end{proposition}
\begin{proof}
We recall from \cite{forlano2022quasiinvariance}*{Lemma 2.4} the simplified Bou\'e-Dupuis variational formula, which in this context reads
\begin{equation} \label{BD}
\begin{aligned}
\log&\int \exp\left(M|Q_{\pi_N}(u)|\right)\mathbbm{1}_{B_R(0)}(u)d\mu_{s,N}(u) \\
&\leq \bb{E}\left[\sup_{V\in H^s}M|Q_{\pi_N}(u_0+V)|\mathbbm{1}_{B_R(0)}(\pi_N(u_0+V))-\frac{1}{2}\|V\|_{H^s}^2\right],
\end{aligned}
\end{equation}
where $u_0$ is a random variable distributed according to $\mu_s$. By symmetry we may expand
\begin{align*}
&|Q_{\pi_N}(u_0+V)| \\
&\le\sum_{{\substack{N_1,N_2,N_3,N_4\lesssim N \\ N_2= N^{(1)}, N_4\sim N^{(4)}, N_1\sim N_2}}} C_0|Q_{\pi_N}(P_{N_1}(u_0+V),P_{N_2}(u_0+V),P_{N_3}(u_0+V),P_{N_4}(u_0+V))|
\end{align*}
for some combinatorial constant $C_0$. 
For ease of notation, we write
\begin{align*}
Q_{\underline{N}}(f_1,f_2,f_3,f_4)=Q_{\pi_N}(P_{N_1}(f_1),P_{N_2}(f_2),P_{N_3}(f_3)&,P_{N_4}(f_4)),
\end{align*}
and again denote $Q_{\underline{N}}(f)=Q_{\underline{N}}(f,f,f,f)$.

Fix $\eps>0$ to be chosen later, then pick $\delta>0$ so that $$\sum_{\substack{N_1,N_2,N_3,N_4\lesssim N \\ N_2\sim N^{(1)}, N_4=N^{(4)}, N_1\sim N_2}} 6\delta N_2^{-\eps}<1.$$ 
For some universal constant $C_0>0$, from \eqref{BD} we then have that
\begin{align}
&\log\int \exp\left(M|Q_{\pi_N}(u)|\right)\mathbbm{1}_{B_R(0)}(u)d\mu_{s,N}(u) \notag \\ 
&\leq\sum_{\substack{N_1,N_2,N_3,N_4\lesssim N \\ N_2\sim N^{(1)}, N_4=N^{(4)}, N_1\sim N_2}}\bb{E}\left[\sup_{V\in H^s}C_0M|Q_{\underline{N}}(u_0+V)|\mathbbm{1}_{B_R(0)}(\pi_N(u_0+V))-\frac{6\delta N_2^{-\varepsilon}}{2}\|V\|_{H^s}^2\right]. \label{BDtoest}
\end{align}

We now expand $Q_{\underline{N}}(u_0+V)$:
\begin{align*}
&Q_{\underline{N}}(u_0+V,u_0+V,u_0+V,u_0+V)\\
&= Q_{\underline{N}}(u_0,u_0,u_0,u_0+V) \tag{I} \\
&\phantom{=\ }+ Q_{\underline{N}}(u_0,u_0,V,u_0+V) + Q_{\underline{N}}(u_0,V,u_0,u_0+V) + Q_{\underline{N}}(V,u_0,u_0,u_0+V) \tag{II}\\
&\phantom{=\ }+ Q_{\underline{N}}(V,V,u_0+V,u_0+V) \tag{III} \\
&\phantom{=\ }+ Q_{\underline{N}}(V,u_0,V,u_0+V)+ Q_{\underline{N}}(u_0,V,V,u_0+V). \tag{IV}
\end{align*}
We now bound each term.

\textbf{\underline{\textrm{(I)}}}: We express \rm{(I)} as a random kernel acting on the low-frequency function $P_{N_4}(u_0+V)$, namely 
\begin{align*}
Q_{\underline{N}}(u_0,u_0,u_0,u_0+V)=\sum_{0\leq n_4<N}K_{\underline{N}}(n_4)\cj{\big({P_{N_4}(u_0+V)}\big)\ft{\phantom{X}}(n_4)},
\end{align*}
where
\begin{equation*}
K_{\underline{N}}(n_4)=\sum_{\substack{n_1-n_2+n_3=n_4 \\ 0\leq n_1,n_2,n_3<N}}\frac{i\Psi_s(\underline{\mathbf{n}})\phi_{N_1}(n_1)\phi_{N_2}(n_2)\phi_{N_3}(n_3)g_{n_1}\overline{g_{n_2}}g_{n_3}}{2\jb{n_1}^s\jb{n_2}^s\jb{n_3}^s}\phi_{\sim N_4}(n_4).
\end{equation*}
We see that
\begin{align*}
|Q_{\underline{N}}(u_0,u_0,u_0,u_0+V)|&\mathbbm{1}_{B_R(0)}(\pi_N(u_0+V)) \\
&\leq\|K_{\underline{N}}(n_4)\|_{\ell^2_{n_4}}\|P_{N_4}\pi_N(u_0+V)\|_{L^2}\mathbbm{1}_{B_R(0)}(\pi_N(u_0+V)) \\
&\leq RN_4^{-\sigma}\|K_{\underline{N}}(n_4)\|_{\ell^2_{n_4}},
\end{align*}
which is a quantity independent of $V$. Now 
using \eqref{Psisbd}, we have
\begin{align*}
&\bb{E}\bigg[\sup_{V\in H^s}C_0M|Q_{\underline{N}}(u_0,u_0,u_0,u_0+V)|\mathbbm{1}_{B_R(0)}(\pi_N(u_0+V))\bigg] \\
&\lesssim MR\Big(N_4^{-2\sigma}\bb{E}\left[\sum_{n_4}|K_{\underline{N}}(n_4)|^2\right]\Big)^{\frac{1}{2}} \\
&\lesssim MR\Bigg(N_4^{-2\sigma}\sum_{\substack{n_1-n_2+n_3-n_4=0 \\ n_i\sim N_i \\ 0\leq n_i<N}}\frac{|\Psi_s(\underline{\mathbf{n}})|^2}{\jb{n_1}^{2s}\jb{n_2}^{2s}\jb{n_3}^{2s}}\Bigg)^{\frac{1}{2}} \\
&\lesssim MR\left(N_4^{-2\sigma}N_2N_3N_4N_2^{4s-2}N_3^2N_2^{-4s}N_3^{-2s}\right)^{\frac{1}{2}},
\end{align*}
so 
\begin{equation} \label{Iest}
    \E \big[\sup_{V\in H^s}|\rm{(I)}|\mathbbm{1}_{B_R(0)}(\pi_N(u_0+V))\big] \les MR \left(N_2^{-1}N_3^{3-2s}N_4^{1-2\sigma}\right)^{\frac{1}{2}},
\end{equation}
which is summable in $N_i$.

\textbf{\underline{\textrm{(II)}}}: We again use a random kernel representation. Note that each of the terms in \rm{(II)} can be written as 
\begin{align*}
\sum_{0\leq n_3,n_4<N}K_{\underline{N}}(n_3,n_4)\jb{n_3}^s\jb{n_4}^\s \widehat{P_{N_{\alpha(3)}}(V)}(n_3)\overline{\widehat{P_{N_4}(u_0+V)}(n_4)},
\end{align*}
where  
\begin{equation*}
K_{\underline{N}}(n_3,n_4)=\sum_{\substack{n_{\al(1)}-n_{\al(2)}+ n_{\al(3)}- n_4= 0 \\ 0\leq n_1,n_2<N}}\hspace{-30pt}
\frac{i\Psi_s(\underline{\mathbf{n}})\phi_{N_{\al(1)}}(n_1)\phi_{N_{\al(2)}}(n_2)\wt{g_{n_1}}\wt{g_{n_2}}}{2\jb{n_1}^s\jb{n_2}^s\jb{n_3}^s\jb{n_4}^{\s}}\phi_{\sim N_{\al(3)}}(n_3)\phi_{\sim N_4}(n_4),
\end{equation*}
$\al$ is a permutation of $\{1,2,3\}$, and  $\wt g_{n_1}, \wt g_{n_2}$ are normal Gaussian random variables 
that satisfy $\E[\wt g_{n_1}\wt g_{n_2}]=0$ unless $n_1 = n_2$ and $\al(2) = 2$, in which case $\wt g_{n_1} = \cj{\wt g_{n_1}}$. Note that in this case, we must also have $\Psi_s(\n) =0$. Therefore,
\begin{align*}
    K_{\underline{N}}(n_3,n_4)
    =\sum_{\substack{n_{\al(1)}-n_{\al(2)}+ n_{\al(3)}- n_4= 0 \\ 0\leq n_1,n_2<N\\\E[\wt g_{n_1}\wt g_{n_2}]\neq0}}\hspace{-30pt}
\frac{i\Psi_s(\underline{\mathbf{n}})\phi_{N_{\al(1)}}(n_1)\phi_{N_{\al(2)}}(n_2)\wt{g_{n_1}}\wt{g_{n_2}}}{2\jb{n_1}^s\jb{n_2}^s\jb{n_3}^s\jb{n_4}^{\s}}\phi_{\sim N_{\al(3)}}(n_3)\phi_{\sim N_4}(n_4).
\end{align*}
We denote by $\|K\|_{\mathrm{op}}$ the optimal constant $C$ such that the inequality 
$$ \sum_{0\leq n_3,n_4<N} K(n_3,n_4) a(n_3) b(n_4) \le C \|a\|_{l^2}\|b\|_{l^2}$$
holds. 
Then it follows that 
\begin{align*}
    |\rm(II)|\mathbbm{1}_{B_R(0)}(\pi_N(u_0+V)) &\les 
    \max_{\al}\|K_{\underline{N}}\|_{\mathrm{op}} \|V\|_{H^s}\|\pi_N(u_0+V)\|_{H^\s}\mathbbm{1}_{B_R(0)}(\pi_N(u_0+V))\\
    &\les R \max_{\al}\|K_{\underline{N}}\|_{\mathrm{op}} \|V\|_{H^s}.
\end{align*}
Taking Hilbert-Schmidt norms, and using \eqref{Psisbd} again, we see that
\begin{align*}
&\E[\|K_{\underline{N}}\|_{\mathrm{op}}^2] \\
& \les \E[\|K_{\underline{N}}\|_{l^2(n_3,n_4)}^2] \\
& \le \sum_{\substack{n_{\al(1)}-n_{\al(2)}+ n_{\al(3)}- n_4= 0 \\ n_1\neq n_2}}
\frac{|\Psi_s(\n)|^2\phi_{N_{\al(1)}}(n_1)^2\phi_{N_{\al(2)}}^2(n_2)}{\jb{n_1}^{2s}\jb{n_2}^{2s}\jb{n_3}^{2s}\jb{n_4}^{2\s}}\phi_{\sim N_{\al(3)}}(n_3)^2\phi_{\sim N_4}(n_4)^2\\
&\les N_2^{-1}N_3^{3-2s} N_4^{1-2\s}.
\end{align*}
Therefore, for some universal constant $C>0$, by Young's inequality, we have 
\begin{equation}\label{IIest}
\begin{aligned}
&\E\left[\sup_{V\in H^s}C_0M|\rm{(II)}|\mathbbm{1}_{B_R(0)}(\pi_N(u_0+V))-\frac{\dl N_2^{-\eps}}{2}\|V\|_{H^s}^2\right]  \\
&\le  \E\left[\sup_{V\in H^s}CMR\|K_{\underline{N}}\|_{\mathrm{op}} \|V\|_{H^s} -\frac{\dl N_2^{-\eps}}{2}\|V\|_{H^s}^2\right]\\
&\le \frac{2C^2M^2R^2 N_2^\eps}{\dl} \E[\|K_{\underline{N}}\|_{\mathrm{op}}^2]\\
&\le \delta^{-1}C^2M^2R^2 N_2^{\eps-1} N_3^{3-2s} N_4^{1-2\s},
\end{aligned}
\end{equation}
which is summable in $N_i$ (for $\varepsilon$ small enough).

\textbf{\underline{\textrm{(III)}}}: To bound this term, we use \eqref{QpiN_CM} and Bernstein's inequality, and obtain (on the set $\pi_N(u_0+V)\in B_R(0)$):
\begin{align*}
|\rm{(III)}|&\les N_2^{2s-1}N_3\|P_{N_1}\pi_N V\|_{L^{2}}\|P_{N_2}\pi_N V\|_{L^{2}} \|P_{N_3}\pi_N (u_0+V)\|_{L^{\infty}}\|P_{N_4}\pi_N (u_0+V)\|_{L^{\infty}}\\
&\les N_2^{-1} N_3^{\frac32-\s}\|P_{\les N_2}\pi_NV\|_{H^s}^2 \|\pi_N(u_0+V)\|_{H^\s}^2 \\
&\les R^2 N_2^{-\min(1,\s-\frac{1}{2})} \|P_{\les N_2}\pi_NV\|_{H^s}^2.
\end{align*}
Therefore, for some universal constant $C$, we have 
$$|\rm{(III)}| \le C R^2  N_2^{-\min(1,\s-\frac12)} \|P_{\les N_2}\pi_NV\|_{H^s}^2. $$
Let $N_0 = N_0(R,M,\dl,\sigma,\varepsilon)$ be such that for every $N'\ge N_0$, 
$$C R^2 (N')^{-\min(1,\s-\frac12)} \le \frac{\dl}{2M} (N')^{-\eps}. $$
Note that this exists as long as $\eps < \min(1, \s-\frac12)$, and for some $b_1, b_2 >0$, we can choose $N_0$ such that 
$$ N_0 \les M^{b_1} R^{b_2}.  $$
Then, for a constant $C'=C'(\dl,\sigma,\varepsilon)$ that can change line to line, distinguishing the cases $N_2 < N_0$ and $N_2 \ge N_0$, we obtain that for some $a_1,a_2>0$,
\begin{align*}
C_0M|\rm{(III)}|\mathbbm{1}_{B_R(0)}(\pi_N(u_0+V)) &\le N_2^{-\varepsilon}C'M^{a_1} R^{a_2} \|\pi_NV\|_{H^{\s}}^2 + 
\frac{\dl}2 N_2^{-\eps} \|V\|_{H^s}^2 \\
&\le N_2^{-\varepsilon} M^{a_1} R^{a_2}C' (R+\|u_0\|_{H^\s})^2 + \frac{\dl}2 N_2^{-\eps} \|V\|_{H^s}^2.
\end{align*}
Therefore, by taking expectations, we get for some universal constant $C >0$,
\begin{equation}\label{IIIest}
    \E\left[\sup_{V\in H^s}C_0M|\rm{(III)}|\mathbbm{1}_{B_R(0)}(\pi_N(u_0+V))-\frac{\dl N_2^{-\eps}}{2}\|V\|_{H^s}^2\right] \le CM^A (1 +R^A)N_2^{-\varepsilon},
\end{equation}
which is summable in $N_i$.

\textbf{\underline{\textrm{(IV)}}}: We use once again \eqref{QpiN_CM}, and obtain (on the set $\pi_N(u_0+V)\in B_R(0)$):
\begin{align*}
    |\rm{(IV)}| &\les N_2^{2s-1}N_3 N_2^{-s}N_3^{-s}  \|P_{\les N_2} \pi_N V\|_{H^s}^2 \|P_{\approx N_2}u_0\|_{L^\infty} \|P_{N_4}\pi_N(u_0+V)\|_{L^\infty}\\
    &\les R N_2^{s-1}N_3^{1-s}N_2^{\frac34-s} \|P_{\les N_2} \pi_NV\|_{H^s}^2 \|u_0\|_{B^{s-\frac12,+}_{4,\infty}}\\
    &\les R N_2^{-\frac14} \|P_{\les N_2} \pi_NV\|_{H^s}^2 \|u_0\|_{B^{s-\frac12,+}_{4,\infty}}.
\end{align*}
Proceeding as for \rm{(III)}, we pick $N_0=N_0(R,\dl,u_0,M,\varepsilon)$ such that  for all $N'\geq N_0$
$$ CR (N')^{-\frac14}\|u_0\|_{B^{s-\frac12,+}_{4,\infty}} \le \frac{\dl}{2M} (N')^{-\eps}, $$
where $C$ is the implicit constant in the inequality above. Moreover we may pick
$$N_0\sim M^{a_1}R^{a_2}||u_0||_{B^{s-\frac12,+}_{4,\infty}}^{a_3}\delta^{-a_4},$$
for some $a_i>0$ depending only on $\varepsilon$.
Therefore, for $0< \ta <1 $ such that
$$ \ta s + (1-\ta) \s \ge s-\frac18, $$
and for constants $C=C(s,\s,\ta,\dl)>0$ and $A_1=A_1(\varepsilon,\theta)>0$, $A_2=A_2(\varepsilon,\theta)>0$ (that can change line to line), distinguishing the cases $N_2 \ge N_0$ and $N_2 < N_0$, 
\begin{align*}
   C_0M|\rm{(IV)}|\mathbbm{1}_{B_R(0)}(\pi_N(u_0+V)) &\le N_2^{-\varepsilon}CM^{A_1}R^{A_2} \|u_0\|_{B^{s-\frac12,+}_{4,\infty}}^{A_2} \|P_{\les N_2} \pi_NV\|_{H^{s-\frac18}}^2 + \frac{\dl}{2} N_2^{-\eps}\|V\|_{H^s}^2 \\
   &\le N_2^{-\varepsilon}CM^{A_1}R^{A_2} \|u_0\|_{B^{s-\frac12,+}_{4,\infty}}^{A_2} \|V\|_{H^s}^{2\ta}\|\pi_NV\|_{H^\s}^{2(1-\ta)} + \frac{\dl}{2} N_2^{-\eps}\|V\|_{H^s}^2\\
   &\le N_2^{-\varepsilon}CM^{A_1}R^{A_2} \|u_0\|_{B^{s-\frac12,+}_{4,\infty}}^{A_2} \|\pi_NV\|_{H^\s}^{2} + {\dl} N_2^{-\eps}\|V\|_{H^s}^2\\
   &\le N_2^{-\varepsilon}CM^{A_1}R^{A_2} \|u_0\|_{B^{s-\frac12,+}_{4,\infty}}^{A_2} (R+\|u_0\|_{H^\s})^{2} + {\dl} N_2^{-\eps}\|V\|_{H^s}^2.
\end{align*}
In particular, we have that for a (different) constant $C >0$,
\begin{equation} \label{IVest}
    \E\left[\sup_{V\in H^s}C_0M|\rm{(IV)}|\mathbbm{1}_{B_R(0)}(\pi_N(u_0+V))-\dl N_2^{-\eps}\|V\|_{H^s}^2\right] \le N_2^{-\varepsilon'}CM^{A_1}R^{A_2},
\end{equation}
which is summable in $N_i$.
Finally, from \eqref{BDtoest}, \eqref{Iest}, \eqref{IIest}, 
\eqref{IIIest}, \eqref{IVest}, we get that, for some constant $K>0,$ and some $A,B>0$,
\begin{align*}
    \log\int \exp\left(M|Q_{\pi_N}(u)|\right)\mathbbm{1}_{B_R(0)}(u)d\mu_{s,N}(u)
    \le K  M\max\{1,M^A\} (1 + R)^B. 
\end{align*}
\end{proof}

We now introduce the sets $E_{N,R,t}\subset \pi_N(H^\sigma)$ for $R\geq 1$ and $t\in\bb{R}$
\begin{equation}\label{ENRtdef}
E_{N,R,t}=\bigcap_{\tau\in I_t}\Phi_{\tau,N}\left(\pi_N\left(B_R(0)\right)\right)=\left\{u\in \pi_N(H^\sigma)\bigg| \sup_{\tau\in I_t} \|\Phi_{-\tau,N}(u)\|_{H^\sigma}\leq R\right\},
\end{equation}
where $I_t=[0,t]$ for $t \geq 0$ and $I_t=[t,0]$ for $t<0$. 
These sets will allow us to obtain a-priori estimates on the $L^p$ norm of $f_{t,N} \1_{E_{N,R,t}}$, in a manner not too dissimilar from the arguments in \cites{AmFi,forlano2022quasiinvariance}. 
\begin{proposition} \label{prop:fNLp}
For $s>1$ and $\frac{1}{2}<\sigma<s-\frac{1}{2}$ with $s-\frac{1}{2}-\sigma$ sufficiently small, there are constants $0<A,B,C<\infty$ such that for all $t\in\bb{R}$, $N\in\bb{N}$, $R\geq 1$, and $p\in(1,\infty)$, we have
\begin{equation} \label{fNLpbound}
\|f_{t,N}\mathbbm{1}_{E_{N,R,t}}\|_{L^p(H^\sigma,d\mu_{s,N})}\leq \left\|\exp\big(|tQ_{\pi_N}(u)|\big)\mathbbm{1}_{B_R(0)}(u)\right\|_{L^p(H^\sigma,d\mu_{s,N})}\leq\exp(C|1+pt|^AR^B).
\end{equation}
\end{proposition}
\begin{proof}
From \eqref{fNformula}, \eqref{fNintegral}, and by Jensen's inequality together with convexity of the exponential, we have
\begin{align*}
&\|f_{t,N}\mathbbm{1}_{E_{N,R,t}}\|_{L^p(H^\sigma,d\mu_{s,N})}^p\\
&=\int (f_{t,N}(u))^{p-1}\mathbbm{1}_{E_{N,R,t}}(u)f_{t,N}(u)d\mu_{s,N}(u) \\
&=\int (f_{t,N}(\Phi_{t,N}(u)))^{p-1}\mathbbm{1}_{E_{N,R,t}}(\Phi_{t,N}(u))d\mu_{s,N}(u) \\
&=\int \exp\left(-(p-1)\int_0^{-t}Q_{\pi_N}(\Phi_{t+\tau,N}(u))d\tau\right)\mathbbm{1}_{E_{N,R,t}}(\Phi_{t,N}(u))d\mu_{s,N}(u) \\
&=\int \exp\left(\frac{t}{|t|}(p-1)\int_{I_t}Q_{\pi_N}(\Phi_{\tau,N}(u))d\tau\right)\mathbbm{1}_{E_{N,R,t}}(\Phi_{t,N}(u))d\mu_{s,N}(u) \\
&\le \int \frac{1}{|t|}\int_{I_t}\exp\left(t(p-1)Q_{\pi_N}(\Phi_{\tau,N}(u))\right)\mathbbm{1}_{E_{N,R,t}}(\Phi_{t,N}(u))d\tau d\mu_{s,N}(u) \\
&=\int \frac{1}{|t|}\int_{I_t}\exp\left(t(p-1)Q_{\pi_N}(u)\right)\mathbbm{1}_{E_{N,R,t}}(\Phi_{N,t-\tau}(u))f_{\tau,N}(u)d\tau d\mu_{s,N}(u).
\end{align*}
For $\tau\in I_t$ we note 
\begin{align*}
\Phi_{N,t-\tau}(u)\in E_{N,R,t}\implies \sup_{r\in I_t-\tau}\|\Phi_{N,r}(u)\|_{H^\sigma}\leq R \implies u\in E_{N,R,\tau}.
\end{align*}
Combining this with the fact $E_{N,R,\tau}\subset B_R(0)$, we see
\begin{align*}
&\|f_{t,N}\mathbbm{1}_{E_{N,R,t}}\|_{L^p(H^\sigma,d\mu_{s,N})}^p\\
&\le \int \frac{1}{|t|}\int_{I_t}\exp\left[t(p-1)Q_{\pi_N}(u)\right]\mathbbm{1}_{B_R(0)}(u)\mathbbm{1}_{E_{N,R,\tau}}(u)f_{\tau,N}(u)d\tau d\mu_{s,N}(u) \\
&\le \left\|\exp\left[t(p-1)|Q_{\pi_N}(u)|\right]\mathbbm{1}_{B_R(0)}(u)\right\|_{L^{p'}(H^\sigma,d\mu_{s,N})}\sup_{\tau\in I_t}\|f_{\tau,N}\mathbbm{1}_{E_{N,R,\tau}}\|_{L^p(H^\sigma,d\mu_{s,N})}.
\end{align*}
Taking the supremum over $t$ in some interval $I_T$ and using $p'(p-1)=p$ yields the first inequality in \eqref{fNLpbound}. The second inequality then follows from Proposition \ref{fNexpbound}.
\end{proof}
\begin{remark}
    If one were able to obtain a similar result with a $H^{\frac{1}{2}}$-cutoff instead of a $H^\sigma$-cutoff, due to conservation of the $H^\frac12$ norm, one could then apply Bourgain's \mbox{(quasi-)invariant} measure argument (\cite{Bo94}, see \cite{forlano2022quasiinvariance}*{Theorem 6.1} for a precise statement), and obtain polynomial-in-time bounds on solutions for $\mu_s$-a.e.\ initial data, due to the conservation of the $\dot H^\frac12$ norm. However, these estimates appear out of reach. 
    We point out that in \cite{GRELLIER_GERARD_2018}, a dense sets of turbulent solutions for \eqref{1} has been constructed, but it is unclear if such solutions are typical or not. 
    It would be interesting to either show good bounds on the growth of Sobolev norms for a.e.\ initial data distributed according to $\mu_s$, or in alternative to exhibit a set of turbulent solutions with positive probability under such a Gaussian measure.
\end{remark}
With this uniform $L^p$ bound on the truncated densities, we are ready to prove Proposition \ref{thm:QI}.

\begin{proof}[Proof of Proposition \ref{thm:QI}.] We first fix $\frac{1}{2}<\sigma<s-\frac{1}{2}$ satisfying the conditions of Proposition \ref{prop:fNLp}, and we recall that $X_s=H^\sigma$ where we have global well-posedness. By Proposition \ref{prop:LWPconvergence}, we have that for all $u\in H^\sigma$ and $T>0$, for every $\s' < \s$,
$$\Phi_{t,N}(\pi_Nu)\to \Phi_t(u) \text{ in } C([-T,T],H^{\sigma'}) \text{ as } N\rightarrow\infty,$$ 
and this convergence is uniform on $B_R(0)$ for every $R\ge 0$.
We now fix $\frac{1}{2}<\sigma' <\s$ that also satisfies the hypoteses of Proposition \ref{prop:fNLp}.
By global well-posedness for \eqref{1} (see \cite{ASENS_2010_4_43_5_761_0}) and this uniform convergence, we have that for any $R\geq 1$ and $t\in\bb{R}$, there exists $C(R,t)>0$ such that for all $N$ big enough (depending on $R,t$),
\begin{equation*}
    \Phi_{t,N}\big(\pi_N(B_R(0))\big)\subset E'_{N,C(R,t),t},
\end{equation*}
where the latter set corresponds to \eqref{ENRtdef} with $\s'$ replacing $\s$.

From the estimate \eqref{fNLpbound}, for any $t\in\bb{R}$ and $R\geq 1$, there exists $f_{t,R}\in L^2(H^{\sigma},d\mu_s)$ such that (up to a subsequence) $f_{t,N}\mathbbm{1}_{E_{N,C(R,t),t}}\circ\pi_N\rightarrow f_{t,R}$ weakly in $L^2(H^{\sigma},d\mu_s)$. With this, we consider any non-negative, continuous and bounded function $F:H^{\sigma'}\rightarrow [0,\infty)$. By dominated convergence, we have
\begin{align*}
    &\int F(u_0)d(\Phi_t)_\#(\mathbbm{1}_{B_R(0)}\mu_s)(u_0)\\
    &=\int F(\Phi_t(u_0))d(\mathbbm{1}_{B_R(0)}\mu_s)(u_0) \\
    &=\lim_{N\rightarrow\infty} \int F(\Phi_{t,N}(\pi_Nu_0))\mathbbm{1}_{B_R(0)}(u_0)d\mu_s(u_0) \\
    &\leq\limsup_{N\rightarrow\infty} \int F(\Phi_{t,N}(\pi_Nu_0))\mathbbm{1}_{\pi_N(B_R(0))}(\pi_Nu_0)d\mu_s(u_0) \\
    &=\limsup_{N\rightarrow\infty} \int F(\pi_Nu_0)\mathbbm{1}_{\Phi_{t,N}(\pi_N(B_R(0)))}(\pi_Nu_0)f_{t,N}(\pi_Nu_0)d\mu_s(u_0) \\
    &\leq\limsup_{N\rightarrow\infty} \int F(\pi_N u_0)\mathbbm{1}_{E_{N,C(R,t),t}}(\pi_Nu_0)f_{t,N}(\pi_N u_0)d\mu_s(u_0).
\end{align*}
We note by the dominated convergence theorem and the uniform $L^2$ bound on $f_{t,N}\mathbbm{1}_{E_{N,C(R,t),t}}\circ\pi_N$:
\begin{align*}
    \lim_{N\rightarrow\infty}\int |F(\pi_Nu_0)-F(u_0)|\mathbbm{1}_{\Phi_{t,N}\pi_N(B_R(0))}(\pi_Nu_0)f_{t,N}(\pi_Nu_0)d\mu_s(u_0)=0,
\end{align*}
and so
\begin{align*}
    \int F(u_0)d(\Phi_t)_\#(\mathbbm{1}_{B_R(0)}\mu_s)(u_0)&\leq\limsup_{N\rightarrow\infty} \int F(u_0)\mathbbm{1}_{E_{N,C(R,t),t}}(\pi_Nu_0)f_{t,N}(\pi_N u_0)d\mu_s(u_0) \\
    &=\int F(u_0)f_{t,R}(u_0)d\mu_s(u_0).
\end{align*}
Therefore $(\Phi_t)_\#(\mathbbm{1}_{B_R(0)}\mu_s)\ll \mu_s$ for any $R\geq 1$. Since $R$ is arbitrary, we deduce that $(\Phi_t)_\#\mu_s\ll \mu_s$ as well.
\end{proof}

\section{Singularity: \texorpdfstring{$s < 1$}{s<1}} \label{sec:4}

\subsection{An abstract singularity result} We start this section by showing a condition that guarantees singularity for up to countably many times. Most of the remaining of the paper will be dedicated to showing that this condition is indeed satisfied by the flow of \eqref{1} whenever $\frac 12 < s < 1$ and $s \neq \frac34$. 
\begin{proposition}\label{sing_abstract}
Let $g: X_s\times X_s \to \R\cup\{-\infty,\infty\}$ be a measurable function that satisfies 
$$ g(x,y) > 0 \Rightarrow g(y,x) < 0. $$
Suppose that there exists $\tau = \tau(u_0) \ge 0$ such that $u_0 \in \WP(\tau)\cap\WP(-\tau)$, and for every $0 < |t| < \tau$, 
$$ g(\Phi_t(u_0), u_0) > 0, $$
and $\tau > 0$ for $\mu_s$-a.e.\ $u_0$.\footnote{Note that $\tau(u_0) = 0$ always (vacuously) satisfies the previous properties, so one just needs to check the existence of such a map on a set of full measure for $\mu_s$.}
Then there exists a countable set $\mathscr N \subseteq \R$ such that for every $t \in \R \setminus \mathscr N$,
$$ (\Phi_t)_\# (\1_{\WP(t)} \mu_s) \perp \mu_s.$$
\end{proposition}
\begin{proof}
First of all, we note that without loss of generality, we can assume that the function $\tau$ is measurable. Indeed, if it is not, we can simply redefine $\tau$ to be 
$$ \tau(u_0)= \sup\Big\{ \frac1 n: n\in \N, u_0 \in \WP\big(\frac1n\big), \{t: |t| \le \frac 1n, g(\Phi_t(u_0), u_0) > 0\} = \big[-\frac 1n, \frac 1n\big]\Big\}, $$
where we define $\tau(u_0) = 0$ if the family on which we take the $\sup$ is empty.
Note that this is measurable since $\WP(t)$ is an open set for every $t\in \R$, and the map $(t,u_0) \mapsto g(\Phi_t(u_0), u_0)$ is measurable due to continuity of the map $(t,u_0) \mapsto \Phi_t(u_0)$ (for  $u_0 \in \WP(n^{-1})$).

Suppose by contradiction that the set 
$$\mathscr N:=\R \setminus \big(\{ t: \mu_s(\WP(t)) > 0, (\Phi_t)_\# (\1_{\WP(t)} \mu_s) \perp \mu_s\} \cup \{ t: \mu_s(\WP(t)) = 0\}\big)$$ 
is uncountable. Since $\tau > 0$ $\mu_s$-a.s., for every $t \in \mathscr N$, there must exist $m_1,m_2 \in \N$ such that 
$$(\Phi_t)_\#(\1_{\WP(t)} \mu_s)(\{\tau> 1/m_1\})>1/m_2.  $$
Since $\mathscr N$ is uncountable, we deduce that there exist $\tau_0, \eps_0 > 0$ such that the set 
$$ \mathscr N_{\tau_0,\eps_0} = \{t: (\Phi_t)_\#(\1_{\WP(t)} \mu_s)(\{\tau> 2\tau_0\})>\eps_0\} $$
is uncountable as well. Since  $\mathscr N_{\tau_0,\eps_0} \subseteq \R$ is uncountable, it must have at least one accumulation point, i.e.\ there exist distinct $t_1,t_2,\dotsc,t_n,\dotsc$ such that $t_j \in \mathscr N_{\tau_0,\eps_0}$ for every $j\in \N$, and $t_* \in \R$ such that 
$$ \lim_{j \to \infty} t_j = t_*. $$
Therefore, up to extracting a subsequence, we can assume that $|t_j - t_*| < \tau_0$ for every $j\in \N$. Now consider the sets 
$$ E_j := \Phi_{t_j}^{-1}(\{u_0: \tau(u_0) > 2\tau_0\}\cap \WP(-t_j))  $$
Recalling that $t_j \in \mathscr N_{\tau_0,\eps_0}$ and that by definition $\Phi_{t_j}^{-1}(\WP(-t_j)) = \WP(t_j)$, we obtain that 
$$ \mu_s(E_j) = (\Phi_{t_j})_\#\mu_s(\{\tau> 2\tau_0\} \cap \WP(-t_j)) = (\Phi_{t_j})_\#(\1_{\WP(t_j)} \mu_s)(\{\tau> 2\tau_0\})>\eps_0. $$
Therefore, there must exist $j_1,j_2$ such that $E_{j_1} \cap E_{j_2} \neq \emptyset$, otherwise we would have 
$$ 1 \ge \mu_s\Big(\bigcup_j E_j\Big) = \sum_j \mu_s(E_j) \ge  \sum_j \eps_0 = \infty. $$ 
Let $u_0 \in E_{j_1} \cap E_{j_2}$. By definition of $E_j$, we have that $\tau(\Phi_{t_{j_k}}(u_0)) > 2\tau_0$ for $k=1,2$. Therefore, by definition of $\tau$, since $|t_{j_1} - t_{j_2}| \le |t_{j_1} - t_*| +|t_*-  t_{j_2}| < 2\tau_0$, we have that 
\begin{align*}
g( \Phi_{t_{j_1}}(u_0), \Phi_{t_{j_2}}(u_0))= g( \Phi_{t_{j_1}-t_{j_2}}(\Phi_{t_{j_2}}(u_0)), \Phi_{t_{j_2}}(u_0)) > 0. 
\end{align*}
Therefore, by hypothesis on $g$, we have that $g( \Phi_{t_{j_2}}(u_0), \Phi_{t_{j_1}}(u_0)) < 0$. However, proceeding similarly, 
\begin{align*}
g( \Phi_{t_{j_2}}(u_0), \Phi_{t_{j_1}}(u_0)) = g( \Phi_{t_{j_2}-t_{j_1}}(\Phi_{t_{j_1}}(u_0)), \Phi_{t_{j_1}}(u_0)) >0,
\end{align*}
which is a contradiction.
\end{proof}

Our main goal for this section will be showing the following. 
\begin{proposition}\label{thm2.0.1}
Let $\frac{1}{2}<s<1$ with $s\neq\frac{3}{4}$, and for $x, y \in X_s$, let  
\begin{equation}
 g(x,y):= \liminf_{N\to\infty} \frac{\|P_Nx\|^2_{\dot{H}^1}-\|P_Ny\|^2_{\dot{H}^1}}{(4s-3)N^{4-4s}}. \label{gdef}
\end{equation}
Then $g$ satisfies the hypotheses of Proposition \ref{sing_abstract}. More precisely, for every $x,y \in X_s$, we have that  $g(x,y) > 0$ implies that $g(y,x) < 0$, and for $\mu_s$-a.e.\ $u_0$, there exists $\tau(u_0) > 0$ such that for all $0<|t|< \tau(u_0)$ we have
\begin{equation*}
    \liminf_{N\to\infty} \frac{\|P_N\Phi_t(u_0)\|^2_{\dot{H}^1}-\|P_Nu_0\|^2_{\dot{H}^1}}{(4s-3)N^{4-4s}}>0.
\end{equation*}
\end{proposition}

\noi
We discuss briefly the main ideas that go into Proposition \ref{thm2.0.1}. It is not unreasonable to expect that, for small times $t$, we have 
\begin{equation}
\| P_N\Phi_t(u_0)) \|_{\dot H^1}^2 - \| P_Nu_0 \|_{\dot H^1}^2 
\approx t \frac{d}{dt} \| P_N \Phi_t(u_0))\|_{\dot H^1}^2\big|_{t=0} + \frac{t^2}2 \frac{d^2}{dt^2} \| P_N u_0 \|_{\dot H^1}^2\big|_{t=0} + O(t^3). \label{eq:texpansion}
\end{equation}
We then exploit the equation \eqref{1} to compute 
$$\frac{d}{dt} \| P_N \Phi_t(u_0))\|_{\dot H^1}^2\big|_{t=0}, 
 \quad  \frac{d^2}{dt^2} \| P_N \Phi_t(u_0))\|_{\dot H^1}^2\big|_{t=0},$$ 
as polynomial objects in the frequencies of $u_0$. By taking expectations, we obtain that 
\begin{align*}
\E\Big[\frac{d}{dt} \| P_N \Phi_t(u_0))\|_{\dot H^1}^2\big|_{t=0}] &= 0, \\
\E\Big[\frac{d^2}{dt^2} \| P_N \Phi_t(u_0))\|_{\dot H^1}^2\big|_{t=0}\Big] &\sim (4s-3) N^{4-4s}. 
\end{align*}
This suggests that for small times, $\| P_N\Phi_t(u_0)) \|_{\dot H^1}^2 - \| P_Nu_0 \|_{\dot H^1}^2 $ has a definite sign as $N \to \infty$ (depending on the sign of $4s-3$), which in turn leads us to the choice of $g$ in \eqref{gdef}. In particular, Theorem \ref{thm2.0.1} follows once we show that this heuristic is correct, and once we bound the various error terms that will appear in the $O(t^3)$ term with the correct power on $N$. 
It actually turns out that the expansion in \eqref{eq:texpansion} is not fully correct, and we will need to allow for extra error terms of the form $o(N^{4-4s})$ (that disappear in the limit in the definition of \eqref{gdef}). 

Note that estimating the remainder in \eqref{eq:texpansion} deterministically, i.e.\ using only the fact that the solution map $t \mapsto \Phi_t(u_0)$ is smooth in time with values in $X_s$, we are only able bound the error term as $O(N^{2s+1}t^3) \gg N^{4-4s}$. Therefore, in order to conclude, we will need to show appropriate multilinear random estimates for the solution of \eqref{1} at time $t$.

\subsection{Time derivatives of the norm of the solution} \label{sec:time_der}
Define the random objects
\begin{align}
    F_N(u_0)&=N^{4s-4}\frac{d}{dt}\|P_Nu(t)\|_{\dot{H}^1}^2\Big|_{t=0},  \label{FN} \\
    G_N(u_0)&=N^{4s-4}\frac{d^2}{dt^2}\|P_Nu(t)\|_{\dot{H}^1}^2\Big|_{t=0}. \label{GN}
\end{align}
Let $u(t)$ be a (local) solution of $(\ref{1})$ on $X_s$. By a direct computation, we have that
\begin{equation*}
    \frac{d}{dt}\|P_Nu(t)\|_{\dot{H}^1}^2=Q_N(u,u,u,u)(t),
\end{equation*}
where $Q_N$ is the quadrilinear map given by
\begin{equation} \label{QNdef}
    Q_N(u_1,u_2,u_3,u_4)=\sum_{\substack{n_1-n_2+n_3-n_4=0 \\ n_i\geq 0}}\frac{i}{2}\Psi_N({\n})\widehat{u_1}(n_1)\overline{\widehat{u_2}(n_2)}\widehat{u_3}(n_3)\overline{\widehat{u_4}(n_4)},
\end{equation}
where we define the multiplier $\Psi_N({\n})$ for ${\n}=(n_1,n_2,n_3,n_4)$ as
\begin{equation}\label{PsiNdef}
\Psi_N({\n})=\prod_{i=1}^4 \1_{\N_0}(n_i)\Big(n_1^2\phi_N(n_1)^2-n_2^2\phi_N(n_2)^2+n_3^2\phi_N(n_3)^2-n_4^2\phi_N(n_4)^2\Big).
\end{equation}
We note that (by construction) $\Psi_N$ is anti-symmetric, so if $n_1-n_2+n_3-n_4=0$ with $\{n_1,n_3\}\cap\{n_2,n_4\}\neq\emptyset$, then $\Psi_N({\n})=0$. We start by showing a simple bound on $\Psi_N$.

\begin{lemma}\label{lemma2.0.2}
For $n_i\in\bb{N}_0$ with $n_1-n_2+n_3-n_4=0$ and $n_i\sim N_i$, we have
\begin{equation*}
    |\Psi_N({\n})|\les \1_{N^{(1)} \gtrsim N} N \min\{N,N^{(3)}\}.
\end{equation*}
\end{lemma}
\begin{proof}
It is clear that if $N^{(1)} \ll N$, we then have $\Psi_N({\n} = 0$, so it is enough to show the bound. Similarly, from $n^2\Phi_N(n) \les N^2$, the bound
$$|\Psi_N({\n})|\les N^2 $$
is immediate, so we focus on showing that $|\Psi_N({\n})| \les N N^{(3)}.$
Since $n_i \ge 0$, in order to have $n_1-n_2+n_3-n_4 = 0$, we cannot have that the biggest two frequencies are $\{n_1,n_3\}$ or $\{n_2,n_4\}$. Therefore, up to swapping $n_1$ with $n_3$ and $n_2$ with $n_4$, we have that 
$$ \min(n_1,n_2) \ge \max(n_3,n_4). $$
Up to further swapping $n_1$ with $n_2$ and $n_3$ with $n_4$ (which does not change $|\Psi_N(\n)|$, we can put ourselves in the case 
$$ n_1 \ge n_2 \ge n_4 \ge n_3. $$
By the mean value theorem, noting that $\| \phi_N' \|_{L^\infty} \le \frac 1N$, we obtain that 
\begin{align*}
|\Psi_N({\n})|&\le |n_1^2\phi_N(n_1)^2-n_2^2\phi_N(n_2)^2| +  |n_4^2\phi_N(n_4)^2-n_3^2\phi_N(n_3)^2| \\
&\les N |n_1-n_2| + N |n_4-n_3| \\
&= 2N|n_4-n_3| \\
&\les N_4N \\
&= N^{(3)} N,
\end{align*}
and so the bound follows.
\end{proof}

We are now ready to estimate $F_N$ and $G_N$ in \eqref{FN}, \eqref{GN}. 
\begin{proposition}\label{propt=0FN}
For $\frac{1}{2}<s<1$, we have that 
$$ \E[F_N(u_0)] = 0, \quad \E[|F_N(u_0)|^2 ]  \les N^{2s-2}.$$
In particular, we have that
\begin{equation*}
\lim_{N\to\infty} F_N(u_0) = 0 \text{ }\mu_s\text{-almost surely},
\end{equation*}
\end{proposition}
\begin{proof} For $N$ dyadic we have
\begin{align}
N^{4-4s}F_N(u_0)=Q_N(u_0,u_0,u_0,u_0)=\sum_{\substack{n_1-n_2+n_3-n_4=0 \\ n_i\geq 0}}\frac{i}{2}\Psi_N({\n})\widehat{u_0}(n_1)\overline{\widehat{u_0}(n_2)}\widehat{u_0}(n_3)\overline{\widehat{u_0}(n_4)}
\end{align} 
Note that $F_N$ is antisymmetric under conjugation, so we have that 
$$ \cj{F_N(u_0)} = - F_N(u_0). $$
However, since $\Law(\ft{u_0}(n)) =  \Law(\cj{\ft{u_0}(n)})$, we obtain 
$$ \E[F_N(u_0)] = -\E[ \cj{F_N(u_0)}] = - \E[F_N(u_0)],   $$
so we must have that 
$$ \E[F_N(u_0)] = 0. $$
We now move to computing the variance of $F_N(u_0)$. Recalling that whenever $n_1 = n_2$ or $n_1 = n_4$ we have $\Psi_N(\n) = 0$, and by Lemma \ref{lemma2.0.2}, we get that 
\begin{align*}
\E \Big[|F_N(u_0)|^2\Big]= &\ N^{8s-8}\sum_{n_1-n_2+n_3-n_4=0}|\Psi_N({\n})|^2\prod_{j=1}^4\frac{1}{\jb{n_j}^{2s}} \\
&\les N^{8s-8}\sum_{N_1,N_2,N_3,N_4} \sum_{\substack{n_1-n_2+n_3-n_4=0\\n_j \sim N_j}}\frac{N^2\min\{N^2,(N^{(3)})^2\}}{N_1^{2s}N_2^{2s}N_3^{2s}N_4^{2s}} \\
&\les N^{8s-8}\sum_{N_1,N_2,N_3,N_4}\frac{N^2\min\{N^2,(N^{(3)})^2\}N^{(2)}N^{(3)}N^{(4)}}{N_1^{2s}N_2^{2s}N_3^{2s}N_4^{2s}} \\
&\les N^{8s-8}N^{6-6s}=N^{2s-2}.
\end{align*}
In particular, this implies that $\E\Big[\sum_N|F_N(u_0)|^2\Big]<\infty$, so we must have 
$$ \lim_{N\to \infty} F_N(u_0) = 0$$
$\mu_s$-almost surely.
\end{proof}
\begin{proposition}\label{propt=0GN}
For $\frac 12 < s < 1$, define 
\begin{equation} \label{Isdef}
I_s := (4s-2)\Big(\int_0^\infty \big(x \phi(x)\big)^2x^{1-4s}dx\Big)\Big(\int_{0}^{1}\int_0^1\int_0^1y^{2-2s}(1-y+(\sigma+\tau)y)^{4s-4}d\sigma d\tau dy\Big).
\end{equation}
Then $I_s > 0$, and we have
\begin{equation}\label{a.s.conv}
\lim_{N\to\infty} G_N(u_0) =8(4s-3)I_s\|u_0\|_{L^2(\T)}^2
\end{equation}
for $\mu_s$-almost every $u_0$.
\end{proposition}
\begin{proof}
We use the symmetry of $F_N$ to compute
    \begin{align*}
    \frac{d^2}{dt^2}\|P_Nu(t)\|_{\dot{H}^1}^2&=2\Re\Big(\sum_{n_1-n_2+n_3-n_4=0}\Psi_N({\n})\widehat{i\partial_tu}(n_1)\overline{\widehat{u}(n_2)}\widehat{u}(n_3)\overline{\widehat{u}(n_4)}\Big) \\
    &=2\sum_{\mathclap{\substack{n_1-n_2+n_3-n_4+n_5-n_6=0 \\  n_1-n_2+n_3\geq 0 \\ n_i\geq 0}}} \Psi_N(n_4-n_5+n_6,n_4,n_5,n_6)\widehat{u}(n_1)\overline{\widehat{u}(n_2)}\widehat{u}(n_3)\overline{\widehat{u}(n_4)}\widehat{u}(n_5)\overline{\widehat{u}(n_6)}.
\end{align*}
Note that the the coefficients $\Psi_N$ are real, and the real part of
$$\prod_{j=1,3,5} \ft{u}(n_j) \prod_{j=2,4,6} \cj{\ft{u}(n_j)} $$
is invariant under permutations in the set 
$$\s \in S:= \big\{\s \in S^6: \s(\{1,3,5\}) = \{1,3,5\} \text{ or }\s(\{1,3,5\}) = \{2,4,6\}\big\}\cong (S^3 \times S^3) \rtimes \Z_2 .$$ 
Therefore, we obtain that 
\begin{align*}
    \frac{d^2}{dt^2}\|P_Nu(t)\|_{\dot{H}^1}^2&= 2 \sum_{\mathclap{\substack{n_1-n_2+n_3-n_4+n_5-n_6=0 \\  n_1-n_2+n_3\geq 0 \\ n_i\geq 0}}} f_N(n_1,n_2,n_3,n_4,n_5,n_6)\widehat{u}(n_1)\overline{\widehat{u}(n_2)}\widehat{u}(n_3)\overline{\widehat{u}(n_4)}\widehat{u}(n_5)\overline{\widehat{u}(n_6)},
    \end{align*}
where $f_N$ is the symmetrisation of $\Psi_N$ under the action of $S$, i.e.\ 
\begin{equation}\label{1.3}
f_N(n_1,n_2,n_3,n_4,n_4,n_6)=\frac{1}{72}\sum_{\s \in S}\Psi_N(n_{\sigma(1)}-n_{\sigma(2)}+n_{\sigma(3)},n_{\sigma(4)},n_{\sigma(5)},n_{\sigma(6)}).
\end{equation}
We now show a bound on $f_N$ analogous to the one in Lemma \ref{lemma2.0.2}, i.e.\ 
\begin{equation} \label{fNbound}
|f_N(\n)|\les \mathbbm{1}_{N^{(1)}\gtrsim N} N\min(N,N^{(3)}),
\end{equation}
under the condition $n_1-n_2+n_3-n_4+n_5-n_6=0$.
To see this, by symmetry it suffices to bound $\Psi_N(n_{\sigma(4)}-n_{\sigma(5)}+n_{\sigma(6)},n_{\sigma(4)},n_{\sigma(5)},n_{\sigma(6)})$ by this quantity. This bound follows by Lemma \ref{lemma2.0.2}, in combination with the observation that we cannot have three of the values $n_4,n_5,n_6$ and $n_4-n_5+n_6$ being bigger than $100N^{(3)}$. Indeed, by definition of $N^{(j)}$, we must have $\min(n_4, n_5, n_6) \les N^{(3)}$, and if $n_1-n_2+n_3 = n_4-n_5+n_6 \ge 100N^{(3)}$, then $\max(n_1,n_2,n_3) \ge 6 N^{(3)}$, 
and so at most one of $n_4, n_5, n_6$ can be bigger that $6 N^{(3)}$ as well. 

We see that
\begin{equation*}
G_N=2N^{4s-4}\sum_{\substack{n_1-n_2+n_3-n_4+n_5-n_6=0  \\ n_i\geq 0}} f_N({\n})\prod_{j=1}^6\jb{n_j}^{-s}g_{n_1}\overline{g_{n_2}}g_{n_3}\overline{g_{n_4}}g_{n_5}\overline{g_{n_6}},
\end{equation*}
which we decompose based on the number of pairings:
\begin{equation*}
G_N=G_{N,0}+G_{N,1}+G_{N,3},
\end{equation*}
where
\begin{align*}
G_{N,0}&=2N^{4s-4}\sum_{\substack{n_1-n_2+n_3-n_4+n_5-n_6=0  \\ n_i\geq 0 \\ \{n_1,n_3,n_5\}\cap \{n_2,n_4,n_6\}=\emptyset}} f_N({\n})\widehat{u}(n_1)\overline{\widehat{u}(n_2)}\widehat{u}(n_3)\overline{\widehat{u}(n_4)}\widehat{u}(n_5)\overline{\widehat{u}(n_6)}, \\
G_{N,1}&=18N^{4s-4}\sum_{\substack{n_1-n_2+n_3-n_4=0  \\ n_i\geq 0 \\ \{n_1,n_3\}\cap \{n_2,n_4\}=\emptyset \\ n_5=n_6}} f_N({\n})\widehat{u}(n_1)\overline{\widehat{u}(n_2)}\widehat{u}(n_3)\overline{\widehat{u}(n_4)}|\ft{u}(n_5)|^2, \\
G_{N,3}&=12N^{4s-4}\sum_{n_1,n_2,n_3\geq 0}f_N(n_1,n_1,n_2,n_2,n_3,n_3)\prod_{j=1}^3|\ft u(n_j)|^2.
\end{align*}
The fact that $G_N=G_{N,0}+G_{N,1}+G_{N,3}$ can be easily verified using the symmetry of $f_N$ under the action of $S$. 
For the 0-pairing term, by independence of different frequencies we have that $\E[G_{N,0}(u_0)] = 0$, and we use \eqref{fNbound} to bound its variance:
\begin{align*}
\bb{E}[|G_{N,0}(u_0)|^2]&\les N^{8s-8}\sum_{\substack{n_1-n_2+n_3-n_4+n_5-n_6=0  \\ n_i\geq 0}} f_N({\n})^2\prod_{j=1}^6\jb{n_j}^{-2s} \\
&\les N^{8s-8}\sum_{\substack{N_1\geq N_2\geq N_3\geq N_4\geq N_5\geq N_6 \\ N_1\gtrsim N, N_1 \sim N_2}}N^2\text{min}(N^2,N_3^2)N_1^{-2s}\prod_{j=2}^6N_j^{1-2s} \\
&\les N^{8s-8}\sum_{\substack{N_1\geq N_2\geq N_3 \\ N_3\gtrsim N}} N^4N_1^{-2s}N_2^{1-2s}N_3^{1-2s}\\
&\phantom{\les\ }+N^{8s-8}\sum_{\substack{N_1\geq N_2\geq N_3 \\ N_1\gtrsim N\gg N_3}} N^2N_1^{-2s}N_2^{1-2s}N_3^{3-2s} \\
&\les N^{2s-2} + N^{6s-6} \\
&\les N^{2s-2},
\end{align*}
which is summable in $N$. Hence we obtain that 
$$ \lim_{N\to\infty} G_{N,0}(u_0) = 0 $$
for $\mu_s$-a.e.\ $u_0$.

We now move to estimating $G_{N,1}(u_0)$. We further decompose it as
\begin{equation*}
G_{N,1}(u_0) =G_{N,1,0}(u_0) +G_{N,1,2}(u_0),
\end{equation*}
defined respectively by\footnote{The attentive reader would notice that $G_{N,1,0}$ and  $G_{N,1,2}$ are the projection to homogenous Wiener chaos of order $0$ and order $2$ (respectively), hence the choice of notation.}
\begin{align*}
G_{N,1,0}&=18N^{4s-4}\sum_{\substack{n_1-n_2+n_3-n_4+n_5-n_6=0  \\ n_i\geq 0 \\ \{n_1,n_3\}\cap \{n_2,n_4\}=\emptyset \\ n_5=n_6}} f_N({\n})\prod_{j=1}^6\jb{n_j}^{-s}g_{n_1}\overline{g_{n_2}}g_{n_3}\overline{g_{n_4}}, \\
G_{N,1,2}&=18N^{4s-4}\sum_{\substack{n_1-n_2+n_3-n_4+n_5-n_6=0  \\ n_i\geq 0 \\ \{n_1,n_3\}\cap \{n_2,n_4\}=\emptyset \\ n_5=n_6}} f_N({\n})\prod_{j=1}^6\jb{n_j}^{-s}g_{n_1}\overline{g_{n_2}}g_{n_3}\overline{g_{n_4}}(|g_{n_5}|^2-1),
\end{align*}
where we fixed the representation 
$$ u_0 = \sum_{n \ge 0} \jb{n}^{-s} g_ne^{inx}. $$
Similarly to the situation for $G_{N,0}$, we have that $\E[G_{N,1,2}(u_0)] = 0$ and 
\begin{align*}
\E[|G_{N,1,2}(u_0)|^2] &\les  N^{8s-8}\sum_{\substack{n_1-n_2+n_3-n_4+n_5-n_6=0  \\ n_i\geq 0, n_5 = n_6}} f_N({\n})^2\prod_{j=1}^6\jb{n_j}^{-2s} \\
&\les N^{2s-2},
\end{align*}
so $\lim_{N\to\infty} G_{N,0}(u_0) = 0$ for $\mu_s$-a.e.\ $u_0$.
We now move to estimating $G_{N,1,0}$. It is again easy to check that $\E[G_{N,1,0}]=0$. Using again \eqref{fNbound}, note that for $n_1 \sim N_1, \dotsc, n_4 \sim N_4$ with $n_1-n_2+n_3-n_4 = 0$, letting $N^{(1)}\geq \dotsc\geq N^{(4)}$ be a decreasing rearrangement of $N_1\dotsc N_4$, recalling that this implies $N^{(1)} \sim N^{(2)}$, we get that
\begin{align*}
 &\sum_{m\ge 0} |f_N(n_1,n_2,n_3,n_4,m,m)| \jb{m}^{-2s} \\
 &\les \sum_{m \les N^{(3)}, N^{(1)} \gtrsim N} N \min(N,N^{(3)}) \jb{m}^{-2s} + \sum_{N^{(2)} \gg m \gg N^{(3)}, N^{(1)} \gtrsim N} N \min(N,\jb{m}) \jb{m}^{-2s} \\
 &\phantom{\les\ } + \sum_{m \gtrsim N^{(2)} \gtrsim N, \max(\jb{m}, N^{(1)}) \gtrsim N} N \min(N,N^{(2)}) \jb{m}^{-2s} \\
 &\sim N \min(N,N^{(3)})\1_{N^{(1)}\gtrsim N} + N^{3-2s}\1_{N^{(1)}\gtrsim N} + N\min(N,N^{(1)})(\max(N,N^{(1)}))^{1-2s}.
\end{align*}
Therefore,
\begin{align*} 
&\E[|G_{N,1,0}(u_0)|^2]\\
&\les  N^{8s-8}\sum_{\mathclap{\substack{ n_1,n_2,n_3,n_4,n_5,n_6\geq 0 \\ n_1-n_2+n_3-n_4=0 }}}|f_N(n_1,n_2,n_3,n_4,n_5,n_5)| \cdot |f_N(n_1,n_2,n_3,n_4,n_6,n_6)|\prod_{j=1}^6\jb{n_j}^{-2s} \\
&= N^{8s-8}\sum_{{\substack{ n_1,n_2,n_3,n_4\geq 0 \\ n_1-n_2+n_3-n_4=0 }}}\prod_{j=1}^4\jb{n_j}^{-2s} \Big(\sum_{m\ge 0} |f_N(n_1,n_2,n_3,n_4,m,m)|\jb{m}^{-2s}\Big)^2\\
\les &\  N^{8s-8}\sum_{\substack{N_1\sim N_2\geq N_3\geq N_4}}  \big(N^2 \min(N,N_3)^2 \1_{N_1 \gtrsim N} +  N^{6-4s}\1_{N^{(1)}\gtrsim N} \\
&\phantom{\ N^{8s-8}\sum_{\substack{N_1\sim N_2\geq N_3\geq N_4}}  \big()}
+ N^2\min(N,N_1)^2\max(N,N_1)^{2-4s}\big)N_1^{-1}\prod_{j=1}^4N_j^{1-2s}  \\
\les &\  N^{2s-2} + N^{-1} + N^{8s-8} \sum_{N_1}N^2\min(N,N_1)^2\max(N,N_1)^{2-4s}N_1^{1-4s}\\
\les &\ N^{2s-2} + N^{-1} N^{4s-4} \max(N^{3-4s}, 1),
\end{align*}
which is again summable in $N$. Therefore,  we get that $\lim_{N\to\infty} G_{N,0,2}(u_0) = 0$ for $\mu_s$-a.e.\ $u_0$, and so
\begin{equation*}
\lim_{N\to\infty} (G_N(u_0)-G_{N,3}(u_0))=0 
\end{equation*}
for $\mu_s$-a.e.\ $u_0$.

In order to evaluate the behaviour of $G_{N,3}$ in the limit $N \to \infty$, it is convenient to work with the non-symmetric form $\Psi_N$ instead of $f_N$, and so we write
$$G_{N,3}(u)=2N^{4s-4}\sum_{n_1,n_2,n_3\geq 0}(2-\mathbbm{1}_{\{n_1=n_3\}})\Psi_N(n_1-n_2+n_3,n_1,n_2,n_3)\prod_{j=1}^3|\ft u(n_j)|^2. $$
Indeed, note that all the pairing in \eqref{1.3} that have $n_{\s(5)} = n_{\s(4)}$ or $ n_{\s(5)} = n_{\s(6)}$ result in $\Psi_N = 0$, and so one must have that $n_{\s(5)} = n_{\s(2)}$, which becomes $n_5 = n_2$ after we undo the symmetrisation procedure.

First of all, we note that in the case $n_1 = n_3$, the third highest frequency between $2n_1 - n_2, n_1, n_2, n_1$ has size $\les N_1$. Therefore, 
\begin{align*}
&N^{4s-4}\Big|\sum_{n_1,n_2\geq 0}\Psi_N(2n_1-n_2,n_1,n_2,n_1) |\ft u_0(n_1)|^4|\ft u_0(n_2)|^2\Big | \\
&\les N^{4s-4} \sum_{n_1,n_2\geq 0, \max(\jb{n_1}, \jb{n_2}) \gtrsim N} N \min(N, \jb{n_1}) \frac{ |g_{n_1}|^4}{\jb{n_1}^{4s}} \frac{|g_{n_2}|^2}{\jb{n_2}^{2s}}
\end{align*}
By taking expectations, we obtain 
\begin{align*}
&N^{4s-4}\E\Big|\sum_{n_1,n_2\geq 0}\Psi_N(2n_1-n_2,n_1,n_2,n_1) |\ft u_0(n_1)|^4|\ft u_0(n_2)|^2\Big | \\
&\les N^{4s-4} \sum_{n_1,n_2\geq 0, \max(\jb{n_1}, \jb{n_2}) \gtrsim N} N \min(N, \jb{n_1}) \frac{ 1}{\jb{n_1}^{4s}} \frac{1}{\jb{n_2}^{2s}}\\
&\les N^{4s-4} \sum_{N_1,N_2: \max(N_1,N_2) \gtrsim N}N \min(N, N_1) N_1^{1-4s} N_2^{1-2s} \\
&\les N^{4s-4} \sum_{M \gtrsim N} N^{3-4s} M^{1 - 2s} + N^2 M^{1-4s} \\
&\les N^{-1},
\end{align*}
which is again summable. When $n_2 = n_3$, we simply get that $\Psi_N(n_1-n_2+n_3,n_1,n_2,n_3) = 0$.
Therefore, for $\mu_s$-a.e.\ $u_0$, we deduce that 
$$G_{N,3}(u_0)=4N^{4s-4}\sum_{n_1,n_2,n_3\geq 0:\ n_j \text{ are distinct}}\Psi_N(n_1-n_2+n_3,n_1,n_2,n_3)\prod_{j=1}^3|\ft u(n_j)|^2 + o(1). $$
Therefore, we can anti-symmetrise further and order the $n_j$. We obtain
\begin{align*}
    &G_{N} + o(1)\\
    &=  8N^{4s-4}\sum_{\mathclap{\substack{n_1>n_2>n_3\geq 0}}} \frac{\Psi_N(n_1-n_2+n_3,n_1,n_2,n_3)+\Psi_N(n_1-n_3+n_2,n_1,n_3,n_2)}{\jb{n_1}^{2s}\jb{n_2}^{2s}\jb{n_3}^{2s}}|g_{n_1}|^2|g_{n_2}|^2|g_{n_3}|^2.
\end{align*}
Note that here we neglected the term with $\Psi_N(n_2-n_1+n_3,n_2,n_1,n_3)$, because in that case we must either have that $n_3 \gtrsim N$, and so 
\begin{align*}
&\E\Big|8N^{4s-4}\sum_{\mathclap{\substack{n_1>n_2>n_3\gtrsim N}}} \frac{\Psi_N(n_2-n_1+n_3,n_2,n_1,n_3)}{\jb{n_1}^{2s}\jb{n_2}^{2s}\jb{n_3}^{2s}} |g_{n_1}|^2|g_{n_2}|^2|g_{n_3}|^2\Big| \\
&\les N^{4s-4} \sum_{N_1,N_2,N_3\gtrsim N} N^2 (N_1N_2N_3)^{1-2s} \\
&\les N^{4s-4} N^{2 +3 -6s} \les N^{1 -2s},
\end{align*}
which is summable in $N$. Otherwise, we have $n_3\ll N$, so in order to have $\Psi_N\neq0$ we must have $n_2<n_1\leq n_2+n_3$ and $\Psi_N(n_2-n_1+n_3,n_2,n_1,n_3)=\phi_N(n_1)^2n_1^2-\phi_N(n_2)^2n_2^2$, in which case
\begin{align*}
&\E\Big|8N^{4s-4}\sum_{\mathclap{\substack{n_1>n_2\gg n_3\geq 0}}} \frac{\Psi_N(n_2-n_1+n_3,n_2,n_1,n_3)}{\jb{n_1}^{2s}\jb{n_2}^{2s}\jb{n_3}^{2s}} |g_{n_1}|^2|g_{n_2}|^2|g_{n_3}|^2\Big| \\
&\les N^{4s-4} \sum_{N_1,N_2\approx N\gg N_3} NN_3N_1^{1-2s}(N_2N_3)^{1-2s} \\
&\les N^{1 -2s},
\end{align*}
which is again summable in $N$.
We now apply a further decomposition to the expression of $G_{N,3}$. We write 
\begin{gather*}
G_{N,3,\le 2} = 8N^{4s-4}\hspace{-12pt}\sum_{{\substack{n_1>n_2>n_3\geq 0}}} \frac{\Psi_N(n_1-n_2+n_3,n_1,n_2,n_3)+\Psi_N(n_1-n_3+n_2,n_1,n_3,n_2)}{\jb{n_1}^{2s}\jb{n_2}^{2s}\jb{n_3}^{2s}}|g_{n_3}|^2,\\
\begin{aligned}
G_{N,3,>2} &= 8N^{4s-4}\sum_{{\substack{n_1>n_2>n_3\geq 0}}} \frac{\Psi_N(n_1-n_2+n_3,n_1,n_2,n_3)+\Psi_N(n_1-n_3+n_2,n_1,n_3,n_2)}{\jb{n_1}^{2s}\jb{n_2}^{2s}\jb{n_3}^{2s}} \\
&\hspace{120pt}\times  \big(|g_{n_1}|^2|g_{n_2}|^2-1\big)|g_{n_3}|^2.
\end{aligned}
\end{gather*}
By the above, we have that 
$$ G_N = G_{N,3,\le2} + G_{N,3,>2} + o(1) $$
as $N\to \infty$. We now show that $G_{N,3,>2}$ is a further error term.
Recall that 
$$ \E[\big(|g_{n_1}|^2|g_{n_2}|^2-1\big)\big(|g_{n_1'}|^2|g_{n_2'}|^2-1\big)] = 0$$
whenever $n_1' \neq n_1, n_2$ and $n_2' \neq n_1,n_2$. Therefore, from the inequality $ab \les a^2 +b^2$, and Lemma \ref{lemma2.0.2}, we obtain 
\begin{align*}
&\E[|G_{N,3,>2}(u_0)|^2] \\
& \les N^{8s-8} \sum_{n_1} \Big(\sum_{\substack{n_2,n_3:\\ 0< n_3<n_2<n_1}} \frac{\Psi_N(n_1-n_2+n_3,n_1,n_2,n_3)+\Psi_N(n_1-n_3+n_2,n_1,n_3,n_2)}{\jb{n_1}^{2s}\jb{n_2}^{2s}\jb{n_3}^{2s}}\Big)^2\\
&\phantom{\les} + N^{8s-8} \sum_{n_2} \Big(\sum_{\substack{n_1,n_3:\\ 0< n_3<n_2<n_1}} \frac{\Psi_N(n_1-n_2+n_3,n_1,n_2,n_3)+\Psi_N(n_1-n_3+n_2,n_1,n_3,n_2)}{\jb{n_1}^{2s}\jb{n_2}^{2s}\jb{n_3}^{2s}}\Big)^2\\
&\phantom{\les} + N^{8s-8} \sum_{n_1} \Big(\sum_{\substack{n_2,n_3:\\ 0< n_3<n_2<n_1}} \frac{\Psi_N(n_1-n_2+n_3,n_1,n_2,n_3)+\Psi_N(n_1-n_3+n_2,n_1,n_3,n_2)}{\jb{n_1}^{2s}\jb{n_2}^{2s}\jb{n_3}^{2s}}\Big)\\
&\phantom{\les + N^{8s-8} \sum_{n_1}} \times  \Big(\sum_{\substack{n_2,n_3:\\ 0< n_3<n_1<n_2}} \frac{\Psi_N(n_2-n_1+n_3,n_2,n_1,n_3)+\Psi_N(n_1-n_3+n_2,n_1,n_3,n_2)}{\jb{n_1}^{2s}\jb{n_2}^{2s}\jb{n_3}^{2s}}\Big)\\
& \les N^{8s-8} \sum_{n_1} \Big(\sum_{\substack{n_2,n_3:\\ 0< n_3<n_2<n_1}} \frac{\Psi_N(n_1-n_2+n_3,n_1,n_2,n_3)+\Psi_N(n_1-n_3+n_2,n_1,n_3,n_2)}{\jb{n_1}^{2s}\jb{n_2}^{2s}\jb{n_3}^{2s}}\Big)^2\\
&\phantom{\les} + N^{8s-8} \sum_{n_2} \Big(\sum_{\substack{n_1,n_3:\\ 0< n_3<n_2<n_1}} \frac{\Psi_N(n_1-n_2+n_3,n_1,n_2,n_3)+\Psi_N(n_1-n_3+n_2,n_1,n_3,n_2)}{\jb{n_1}^{2s}\jb{n_2}^{2s}\jb{n_3}^{2s}}\Big)^2\\
& \les N^{8s-8} \sum_{N_1\ge N_2\ge N_3, N_1\gtrsim N} \big(N_1^{1-4s} N_2^{1-4s}N_3^{2-4s} N^2 \min(N,N_2)^2\big)(N_2 + N_1) \\
& \les N^{8s-8} \sum_{N_1\ge N_2\ge N_3, N_1\gtrsim N} N_1^{2-4s} N_2^{1-4s}N_3^{2-4s} N^2 \min(N,N_2)^2\\
&\les N^{8s-8} \sum_{N_2} N_2^{1-4s} N^2 \max(N,N_2)^{2-4s} \min(N,N_2)^2\\
&\les N^{8s-8} N^{4 - 4s + \max(3-4s,0)} \les N^{4s - 4} + N^{-1},
\end{align*}
which is again summable in $N$. This shows that 
$$ G_N = G_{N,3,\le2} + o(1)$$
as $N\to \infty$.
Therefore, in order to conclude \eqref{a.s.conv}, we just need to show that $G_{N,3,\le2}(u_0) = 8 (4s-3) I_s \| u_0\|_{L^2}^2 + o(1)$. For convenience, write 
\begin{equation}
 G_{N,3,\le2}(u_0) = 8 \sum_{n\ge 0} |\ft{u_0}(n)|^2 A_N(n), \label{ANdef0}
\end{equation}
i.e.\ 
\begin{equation*}
A_N(n)=\sum_{\substack{n_1>n_2>n}}\Big[N^{4s-4}\frac{\Psi_N(n_1-n_2+n,n_1,n_2,n)+\Psi_N(n_1-n+n_2,n_1,n,n_2)}{\jb{n_1}^{2s}\jb{n_2}^{2s}}\Big].
\end{equation*}
Note that $A_N(n)$ is a quantity independent from $u_0$. We first show that for every $n \ge 0$,
\begin{equation} \label{ANlimit}
 \lim_{N\to \infty} A_N(n) = (4s-3)I_s,
 \end{equation}
and that $I_s > 0$. 
We start by proving that, for every $0< \eps\le \eps_0(s) \ll1$, as $N\to\infty$ we have
\begin{align}
A_N(n)=N^{4s-4}\sum_{N^{1+\varepsilon}\geq n_1>n_2\geq N^{1-\varepsilon}}\frac{\Psi_N(n_1+n_2,n_1,0,n_2)+\Psi_N(n_1-n_2,n_1,n_2,0)}{n_1^{2s}n_2^{2s}}+o_N(1). \label{ANapprox}
\end{align}
Indeed, recalling the definition of $\Psi_N$ in \eqref{PsiNdef}, by the mean value theorem, we have that for $N\gg \jb{n}$,
\begin{align*}
&\Big|A_N(n)-N^{4s-4}\sum_{N^{1+\varepsilon}\geq n_1>n_2\geq N^{1-\varepsilon}}\frac{\Psi_N(n_1+n_2,n_1,0,n_2)+\Psi_N(n_1-n_2,n_1,n_2,0)}{n_1^{2s}n_2^{2s}}\Big| \\
&\les N^{4s-4} \sum_{N^{1+\varepsilon}\geq n_1>n_2\geq N^{1-\varepsilon}} N\frac{ \jb{n}}{n_1^{2s}n_2^{2s}} \\
&\phantom{\les } + N^{4s-4} \sum_{n_1 > N^{1+\eps}, n_1> n_2>n} \frac{N^2\1_{|n_2| \gtrsim N}}{\jb{n_1}^{2s}\jb{n_2}^{2s}}
+ N^{4s-4} \sum_{n_2 < N^{1-\eps}, n_1> n_2>n} \frac{N\jb{n_2}\1_{|n_1| \gtrsim N}}{\jb{n_1}^{2s}\jb{n_2}^{2s}}\\
&\les \jb{n} N^{4s-3} N^{(1-\eps)(2-4s)} + N^{2s-1} N^{(1+\eps)(1-2s)} + N^{2s-2} N^{(1-\eps)(2-2s)}= o(1)
\end{align*}
as $N \to \infty$, as long as $0< \eps\ll1$. 
Now, let $$h(x)=x^2\phi(x)^2.$$ Recalling the definition of $\Psi_N$ in \eqref{PsiNdef}, and by writing $\Psi_N$ in terms of $h$, by \eqref{ANapprox} we have
\begin{equation*}
A_N(n)=\frac{1}{N^2}\sum_{N^{1+\varepsilon}\geq n_1>n_2\geq N^{1-\varepsilon}}\frac{h\Big(\frac{n_1}{N}+\frac{n_2}{N}\Big)-2h\Big(\frac{n_1}{N}\Big)+h\Big(\frac{n_1}{N}-\frac{n_2}{N}\Big)}{(\frac{n_1}{N})^{2s}(\frac{n_3}{N})^{2s}}+o_N(1),
\end{equation*}
which we can view as an (improper) Riemann sum, and so
\begin{align*}
\lim_{N\to\infty}A_N(n)={\int_{x\geq y\geq 0}\frac{h(x+y)-2h(x)+h(x-y)}{x^{2s}y^{2s}}dxdy}.
\end{align*}
Both the existence of the integral and convergence of the Riemann sums can be justified by the fact $h$ is smooth and compactly supported, and then using the double mean value theorem:
\begin{align*}
&\frac{|h(x+y)-2h(x)+h(x-y)|}{x^{2s}y^{2s}}\\
\les &\ \|h''\|_{L^\infty} x^{-2s}y^{2-2s} \1_{0\le y \le x \le 4}
+\|h\|_{L^\infty}x^{-2s}y^{-2s} \1_{1\le y \le x},
\end{align*}
which is integrable. Together with the analogous estimates for $h'$, one can easily show convergence of the Riemann sum.

We now manipulate this integral in order to show its relationship with $I_s$, and verify that $I_s>0$. Using the fact $h$ is supported away from the origin, by change of variable, we have
\begin{align*}
&\ \lim_{\varepsilon\downarrow 0}\Big(\int_{x\geq y\geq\varepsilon x}\frac{h(x+y)-2h(x)+h(x-y)}{x^{2s}y^{2s}}dxdy\Big) \\
= &\ \lim_{\varepsilon\downarrow 0}\Big(\int_{\substack{x\geq 0 \\ 1\geq u\geq\varepsilon}}[h(x(1+u))-2h(x)+h(x(1-u))]x^{1-4s}u^{-2s}dxdu\Big) \\
= &\ \lim_{\varepsilon\downarrow 0}\Big(\int_{\substack{x\geq 0 \\ 1\geq u\geq\varepsilon}}h(x(1+u))x^{1-4s}u^{-2s}dxdu-2\int_{\substack{x\geq 0 \\ 1\geq u\geq\varepsilon}}h(x)x^{1-4s}u^{-2s}dxdu\Big. \\
&\hspace{7cm}\Big.+\int_{\substack{x\geq 0 \\ 1\geq u\geq\varepsilon}}h(x(1-u))x^{1-4s}u^{-2s}dxdu\Big) \\
= &\ \Big(\int_0^\infty h(x)x^{1-4s}dx\Big)\lim_{\varepsilon\downarrow 0}\Big(\int_{\varepsilon}^{1}u^{-2s}(1+u)^{4s-2}du-2\int_\varepsilon^1u^{-2s}du+\int_{\varepsilon}^1 u^{-2s}(1-u)^{4s-2}du\Big) \\
= &\ \Big(\int_0^\infty h(x)x^{1-4s}dx\Big)\lim_{\varepsilon\downarrow 0}\Big(\int_{\varepsilon}^{1}u^{-2s}\Big[(1+u)^{4s-2}-2+(1-u)^{4s-2}\Big]du\Big).
\end{align*}
We note that the final integral converges since in a neighbourhood of $0$, 
$$ u^{-2s} (1+u)^{4s-2}-2+(1-u)^{4s-2} \les u^{2-2s},  $$
which is integrable. Moreover, by using the fundamental theorem of calculus twice, we get that
\begin{align*}
&\int_{0}^{1}u^{-2s}\Big[(1+u)^{4s-2}-2+(1-u)^{4s-2}\Big]du\\
&=(4s-3){(4s-2)\Big(\int_{0}^{1}\int_0^1\int_0^1u^{2-2s}(1-u+(\sigma+\tau)u)^{4s-4}d\sigma d\tau du\Big)},
\end{align*}
and so \eqref{ANlimit} follows. The fact that $I_s > 0$ follows from the fact that on the interval $[0,1]^3$, we have that 
$ 1-u+(\sigma+\tau)u \ge 0, $ and $h \ge 0$ as well. 

In order to conclude the proof, we just need to show that we can swap limit and summation in \eqref{ANdef0}. By Lemma \ref{lemma2.0.2}, we have that 
\begin{align*}
|A(n)| & \les N^{4s-4} \sum_{n_1>n_2} \frac{N \min(N,n_2) \1_{n_1 \gtrsim N} }{\jb{n_1}^{2s} \jb{n_2}^{2s}} \\
&\les N^{4s-4}N^{3-2s}N^{1 - 2s} \les 1. 
\end{align*}
Therefore, for $\mu$-a.e.\ $u_0$, by dominated convergence,
\begin{equation*}
\lim_{N \to \infty} G_N(u_0) = \lim_{N \to \infty} G_{N,3,\le2}(u_0) = 8 \lim_{N\to\infty} \sum_{n\ge 0} |\ft{u_0}(n)|^2 A_N(n) = 8(4s-3)I_s\|u_0\|_{L^2}^2.
\end{equation*}

%
\end{proof}

\subsection{Paralinear decomposition of the flow}
\label{sec:paradec}
We now need to justify the fact that (for short times) we can approximate $\|P_N\Phi_t(u_0)\|_{\dot{H}^1}^2$ by its second order expansion in time.
In all of the following we let $\frac{1}{2}<s<1$ and recall our choice of $X_s$ as $X_s = B^{s-\frac12,+}_{p,\infty}$ with $p>\min\Big\{100,\frac{1}{s-\frac{1}{2}}\Big\}$. We also let 
$$B_R:=\Big\{u_0\in X_s \Big| \|u_0\|_{X_s}\le R\Big\}.$$
The first step to showing the validity of the expansion \eqref{eq:texpansion} is to linearise \eqref{1} at high-frequency. In particular, note the following (para-product) decomposition:
\begin{align}
    i\partial_tu=\Pi(|u|^2u)= &\ \sum_{\mathclap{N_1,N_2,N_3}}\Pi(P_{N_1}u\overline{P_{N_2}u}P_{N_3}u) \nonumber \\
    = &\ 2\sum_{\mathclap{N_1,N_2\ll N_3}}\Pi(P_{N_1}u\overline{P_{N_2}u}P_{N_3}u) +2\sum_{\mathclap{\substack{N_1\ll N_3 \\ N_2\sim N_3}}}\Pi(P_{N_1}u\overline{P_{N_2}u}P_{N_3}u) \nonumber \\
    &+\sum_{\mathclap{\substack{N_1\sim N_3 \\ N_2\ll N_3}}}\Pi(P_{N_1}u\overline{P_{N_2}u}P_{N_3}u)+\sum_{\mathclap{N_1,N_2\sim N_3}}\Pi(P_{N_1}u\overline{P_{N_2}u}P_{N_3}u). \label{eq:paradecomposition}
\end{align}
This first term will serve as our linearisation around high-frequency. For ease of notation, we introduce the trilinear para-product operator describing this Low-Low-High interaction:
\begin{equation*}
        \Pi_{L,L,H}(u,v,w)\vcentcolon=\sum_{N_1,N_2\ll N_3}\Pi(P_{N_1}u\overline{P_{N_2}v}P_{N_3}w).
    \end{equation*}
Now, consider the system of equations
\begin{equation}\label{eq:paraLWP}
    \begin{cases}
    i\partial_tu=\Pi(|u|^2u) \\
    i\partial_t X=2\Pi_{L,L,H}(u,u,X), \\
    Y=u-X, \\
    (u,X,Y)|_{t=0}=(u_0,u_0,0).
    \end{cases}
\end{equation}
Namely, $u = X+Y$ represents the solution to \eqref{1}, while $X$ flows under the linearised Low-Low-High interaction. The main advantage of this decomposition is that the term $X$ captures the main nonlinear obejct with regularity $s-\frac12$, while $Y$ will be a smoother remainder. 
In particular, we have the following. 
\begin{proposition} \label{prop:paraLWP}
For all $R\geq 1$, there exists $T_R>0$ such that for all $u_0\in B_R$, there exists a unique triple
$$(u,X,Y)\in C^2\big([-T_R,T_R]; X_s \times X_s \times B^{2s-1,+}_{p/2,\infty}\big)$$ 
that satisfies \eqref{eq:paraLWP} with
\begin{equation} \label{PLWPbound}
    \|(u,X,Y)\|_{C^2\Big([-T_R,T_R];B^{s-\frac{1}{2},+}_{p,\infty}\times B^{s-\frac{1}{2},+}_{p,\infty}\times B^{2s-1,+}_{p/2,\infty}\Big)}\leq C(1+R^5)
\end{equation}
for some universal constant $C >0$.
\end{proposition}
\begin{proof}
First, consider the following multilinear bound, for $N$ dyadic:
\begin{align*}
    \|P_N\Big(\Pi_{L,L,H}(u,v,w)\Big)\|_{L^q}&\les \sum_{N_1,N_2\ll N_3}\|P_N\Big(P_{N_1}u\overline{P_{N_2}v}P_{N_3}w\Big)\|_{L^q} \\
    &\les \sum_{\substack{N_1,N_2\ll N_3 \\ N_3\approx N}}\|P_{N_1}u\|_{L^\infty}\|P_{N_2}v\|_{L^\infty}\|P_{N_3}w\|_{L^q} \\
    &\les \|u\|_{B^{s-\frac{1}{2}}_{p,\infty}}\|v\|_{B^{s-\frac{1}{2}}_{p,\infty}}\sum_{N_3\approx N}\|P_{N_3}w\|_{L^q},
\end{align*}
where we just used H\"older's inequality and the embedding $B^{s-\frac{1}{2}}_{p,\infty} \hookrightarrow B^0_{\infty,1} $.
In particular, we obtain that for every $\s \in \R, q \ge 1$, 
\begin{equation} \label{PLWP1}
    \|\Pi_{L,L,H}(u,v,w)\|_{B^{\s,+}_{q,\infty}} \les \|u\|_{B^{s-\frac{1}{2}}_{p,\infty}}\|v\|_{B^{s-\frac{1}{2}}_{p,\infty}} \|w\|_{B^{\s,+}_{q,\infty}}.
\end{equation} 
Moreover, by \eqref{eq:paradecomposition}, for every $q> \frac 1{2s - 1}$, we have 
\begin{align*}
   & \| P_N\Big(\Pi(|u|^2 u) - 2\Pi_{L,L,H}(u, u,u)\Big) \|_{L^q} \\
   & = \Big\|P_N\Big(2\sum_{\mathclap{\substack{N_1\ll N_3 \\ N_2\sim N_3}}} \Pi(P_{N_1}u\overline{P_{N_2}u}P_{N_3}u) 
    +\sum_{\mathclap{\substack{N_1\sim N_3 \\ N_2\ll N_3}}}\Pi(P_{N_1}u\overline{P_{N_2}u}P_{N_3}u)+\sum_{\mathclap{N_1,N_2\sim N_3}}\Pi(P_{N_1}u\overline{P_{N_2}u}P_{N_3}u)\Big)\Big\| \\
    & \les \|u\|_{B^{s-\frac12,+}_{2q,\infty}}^2 \Big( 
    \sum_{\mathclap{\substack{N_1\ll N_3 \\ N_2\sim N_3\gtrsim N }}} ||P_{N_1}u||_{L^\infty}(N_2N_3)^{\frac12 -s }\\
    &\phantom{\les \|u\|_{B^{s-\frac12,+}_{2q,\infty}}^2 \Big() }
    +\sum_{\mathclap{\substack{N_1\sim N_3 \gtrsim N  \\ N_2\ll N_3}}}||P_{N_2}u||_{L^\infty}(N_1N_3)^{\frac12 -s } +\sum_{\mathclap{N_1\sim N_2\sim N_3 \gtrsim N}}||P_{N_1}u||_{L^\infty}(N_2N_3)^{\frac12 -s }  
    \Big) \\
    &\les N^{1-2s} \|u\|_{B^{s-\frac12,+}_{2q,\infty}}^3, 
\end{align*}
where we used H\"older's inequality $\| fgh \|_{L^q} \le \| f\|_{L^\infty} \|g\|_{L^{2q}} \|h\|_{L^{2q}}, $ and the embedding $ B^{s-\frac12,+}_{2q,\infty} \hookrightarrow B^0_{\infty,1}$. We obtain that 
\begin{equation} \label{PLWP2}
    \| P_N\Big(\Pi(|u|^2 u) - 2\Pi_{L,L,H}(u, u,u)\Big) \|_{B^{2s-1,+}_{q,\infty}} \les \|u\|_{B^{s-\frac12,+}_{2q,\infty}}^3.
\end{equation}
Now, rewrite the equations for $X, Y$ in \eqref{eq:paraLWP} as 
\begin{gather*}
    i\partial_t X=2\Pi_{L,L,H}(u,u,X), \\
    i\partial_t Y= \Pi(|u|^2 u) - 2\Pi_{L,L,H}(u, u,u) + 2\Pi_{L,L,H}(u, u,Y).
\end{gather*}
Then, by a standard Banach fixed point argument, from the estimates \eqref{PLWP1},\eqref{PLWP2}, we obtain that there exists a unique couple $(X,Y) \in C([-T_R,T_R], B^{s-\frac 12,+}_{p,\infty} \times B^{2s-1,+}_{p/2,\infty}),$ while the existence of $u$ follows from Proposition \ref{prop:LWP}. 
The fact that $(X,Y) \in C^2([-T_R,T_R], B^{s-\frac 12,+}_{p,\infty} \times B^{2s-1,+}_{p/2,\infty})$ follows from similar considerations to the ones in Proposition \ref{prop:LWP}, together with \eqref{PLWP1} and the analogous of \eqref{PLWP2} for the difference 
$$ \partial_t^2 \big( \Pi(|u|^2 u) - 2\Pi_{L,L,H}(u, u,u) \big).$$
The reason of the power $R^5$ in the RHS of \eqref{PLWPbound} is the fact that 
$ \partial_t^2 u|_{t=0}$ is a quintic expression in $R$.  
We omit the details.
\end{proof}
Recall that, in view of \eqref{eq:texpansion}, Proposition \ref{propt=0FN} and Proposition \ref{propt=0GN}, our goal is showing that, in an appropriate sense,
\begin{align*}
    &\| P_N\Phi_t(u_0)) \|_{\dot H^1}^2 - \| P_Nu_0 \|_{\dot H^1}^2 \\
    &= \frac{t^2}2 \frac{d^2}{dt^2} \| P_N u_0 \|_{\dot H^1}^2\big|_{t=0} + R(t,u_0), 
\end{align*}
with 
$$ |R(t, u_0)| \le \eps t^2 N^{4-4s} + o(N^{4-4s})  $$
for $t$ small enough (depending both on $\eps$ and $u_0$). 
Recalling the definition of $Q_N$ in \eqref{QNdef} and of $F_N,G_N$ in \eqref{FN}, \eqref{GN} respectively, such a statement reduces to showing that, for every fixed $\eps>0$, and for every $t$ small enough (depending on $u_0,\eps$), we have
$$ \big|N^{4s-4}Q_N(u(t),u(t),u(t),u(t)) - F_N(u_0) - t G_N(u_0)\big| \le \eps t.  $$
In order to see this, one can write $u(t)$ as $u(t) = X(t) + Y(t)$, and expand the terms in the expression for $Q_N$ by using quadrilinearity. 
The next lemma essentially states that the estimate above holds whenever the expansion of $Q$ contains at least one term in $Y$ that does not appear in the lowest frequency. 
\begin{lemma}
    Let $N_1, N_2,N_3,N_4$ be dyadic numbers, and let $f_1,\dots, f_4 \in L^4_+(\T)$. We have that 
    \begin{equation}\label{QN_CM}
        \big|Q_N(P_{N_1} f_1, P_{N_2} f_2, P_{N_3} f_3, P_{N_4} f_4)\big| \les N\min(N,N^{(3)}) \prod_{j=1}^4 \|f_j\|_{L^4},
    \end{equation}
    where we recall that $N^{(1)} \ge N^{(2)}\ge N^{(3)}\ge N^{(4)} $ is a reordering of $N_1,\dotsc,N_4$.
\end{lemma}
\begin{proof}
    Recall the formula \eqref{PsiNdef} for the multiplier $\Psi_N(\n)$ appearing in \eqref{QNdef}, 
$$\Psi_N(\n) = n_1^2\phi_N(n_1)^2 - n_2^2\phi_N(n_2)^2 + n_3^2\phi_N(n_3)^2 - n_4^2\phi_N(n_4)^2, $$
where we omitted the indicator of nonnegative frequencies due to the assumption $f_j \in L^4_+$. 

Exploiting the symmetries of $\Psi_N$, we can assume without loss of generality that $N_1\sim N_2 \gtrsim N$, and $N_3,N_4 \les N_1$.
Under the condition $n_1-n_2+n_3-n_4 =0$, and $n_j \sim N_j$, we can decompose $\Psi_N$ as 
\begin{align*}
    \Psi_N(\n) &= N(n_4-n_3) \Big(\frac{n_1^2\phi_N(n_1)^2 - n_2^2\phi_N(n_2)^2}{N(n_1-n_2)}\prod_{j=1}^4 \phi_{\les N}(n_j)\Big) +  n_3^2\phi_N(n_3)^2 - n_4^2\phi_N(n_4)^2 \\
    & =: N(n_4-n_3) h_N(n_1,n_2,n_3,n_4) +  n_3^2\phi_N(n_3)^2 - n_4^2\phi_N(n_4)^2. 
\end{align*}
It is fairly easy to check that $h_N(x_1,x_2,x_3,x_4) = h(\frac{x_1}{N},\frac{x_2}{N}, \frac{x_3}{N}, \frac{x_4}{N})$, and $h$ is a smooth, compactly supported function of $x,y$. Therefore, by Proposition \ref{prop:CM}, we obtain 
\begin{align*}
     &\big|Q_N(P_{N_1} f_1, P_{N_2} f_2, P_{N_3} f_3, P_{N_4} f_4)\big| \\
     & \le \Big| \sum _{n_1-n_2+n_3-n_4 =0} N h_N(n_1,n_2,n_3,n_4) (n_4-n_3) \ft{P_{N_1}f_1}(n_1) \overline{\ft{P_{N_2}f_2}(n_2)} \ft{P_{N_3}f_3}(n_3) \overline{\ft{P_{N_4}f_4}(n_4)}\Big| \\
     &\phantom{\le\ } +\Big| \sum _{n_1-n_2+n_3-n_4 =0} n_3^2\phi_N(n_3)^2 \ft{P_{N_1}f_1}(n_1) \overline{\ft{P_{N_2}f_2}(n_2)} \ft{P_{N_3}f_3}(n_3) \overline{\ft{P_{N_4}f_4}(n_4)}\Big| \\
     &\phantom{\le\ } +\Big| \sum _{n_1-n_2+n_3-n_4 =0} n_4^2\phi_N(n_4)^2 \ft{P_{N_1}f_1}(n_1) \overline{\ft{P_{N_2}f_2}(n_2)} \ft{P_{N_3}f_3}(n_3) \overline{\ft{P_{N_4}f_4}(n_4)}\Big| \\
     &\les \big(N \|P_{N_1}f_1\|_{L^4} \|P_{N_2}f_2\|_{L^4} (\|(P_{\les N} P_{N_3}f_3)'\|_{L^4} \|P_{N_4}f_4\|_{L^4} + \|P_{N_3}f_3\|_{L^4} \|(P_{\les N}P_{N_4}f_4)'\|_{L^4})\big) \\
     &\phantom{\les\ }+ \|P_{N_1}f_1\|_{L^4} \|P_{N_2}f_2\|_{L^4} \|(P_NP_{N_3}f_3)''\|_{L^4} \|P_{N_4}f_4\|_{L^4} \\
     &\phantom{\les\ }+   \|P_{N_1}f_1\|_{L^4} \|P_{N_2}f_2\|_{L^4} \|P_{N_3}f_3\|_{L^4} \|(P_NP_{N_4}f_4)''\|_{L^4}\big)  \\
     &\les \prod_{j=1}^4 \|f_j\|_{L^4}\big( N(\min(N,N_3) + \min(N,N_4)) + \min(N,N_3)^2 + \min(N,N_4)^2\big)\\
     &\les N \min(N,N^{(3)}) \prod_{j=1}^4 \|f_j\|_{L^4}.
\end{align*}
\end{proof}
In view of the previous lemma and of Proposition \ref{propt=0FN}, we are left with estimating the difference 
$$ \big|N^{4s-4}Q_N(X(t),X(t),X(t),u(t)) - t G_N(u_0)\big|$$
under the extra assumption that the terms in $u(t)$ in the expression above appear at the lowest frequency. 
However, one can check that for generic $f_1, \dotsc, f_4 \in B^{s-\frac12,+}_{p,\infty}$, one cannot expect an estimate which is any better than 
$$Q_N(f_1,f_2,f_3,f_4)\les N^{\frac72 -3s}.$$
Therefore, in order to conclude, we need to exploit the random oscillations of $X(t)$ in order to show that the expression 
$$ Q_N(X(t),X(t),X(t),u(t))$$
has the correct growth in the parameter $N$ (and that has the correct dependency in the parameter $t$). In order to achieve this, we will need some further reductions. As it turns out, it is actually convenient to treat $Y(t)$ together with further error terms. In particular, the goal of this subsection is to show the following.

\begin{proposition}
\label{lemma3.4}
For $u_0\in X_s$ and $N$ dyadic, denote 
$$u_{0,N}=\sum_{M\ll N}P_Mu_0 = P_{\ll N}u_0.$$ 
Let $u_N(t)=\Phi_t(u_{0,N})$.  
Then there exists $T_0 = T_0(\|u_0\|_{X_s})>0$ such that for every $N$ dyadic, 
    \begin{equation*}
        P_Nu(t)=P_N\Pi (X_N)(t)+v_N(t),
    \end{equation*}
    where 
    \begin{equation} \label{XNdef}
        X_N(t)=\exp\Big(-2i\int_0^t|P_{\ll N}u_N|^2(\tau)d\tau\Big)P_{\approx N}u_0,
    \end{equation}
    and $u$, $v_N$ satisfy:
    \begin{equation}\label{paraLWPestimates}
    \begin{gathered}
                \|u(t)\|_{C^2\big([-T_0,T_0];X_s\big)} \les 1+ \|u_0\|_{X_s}^5,\\
                \|v_N(t)\|_{C^2\big([-T_0,T_0];L^\frac{p}{3}\big)}\les (1+ \|u_0\|_{X_s}^5) N^{1-2s}.
    \end{gathered}
    \end{equation}
\end{proposition}
We defer the proof of this proposition to the end of this subsection, since we will need a number of preliminary lemmas. 

Firstly, we have the following bound on the difference of solutions $u(t)$ and $u_N(t)$. 
\begin{lemma}\label{lemma3.2}
For every $u_0 \in X_s$, there exists $T_0 = T_0(\|u_0\|_{X_s})>0$ such that for every $N$ dyadic,
\begin{equation*}
    \|u(t)-u_N(t)\|_{C^2\Big([-T_R,T_R];B^{0,+}_{p,1}\Big)}\les N^{\frac{1}{2}-s} \|u_0\|_{X_s}(1+\|u_0\|_{X_s}^4) ,
\end{equation*}
where $p > \max(100, \frac{1}{s-\frac12})$ is the same parameter that appears in the definition \eqref{Xsdef} of $X_s$.
\end{lemma}
\begin{proof}
Let $R=1+ \|u_0\|_{X_s}$. Note that, by definition,
$$ \|u_{0,N}\|_{X_s} \les \|u_0\|_{X_s}. $$
Therefore, by Proposition \ref{prop:LWP} we may pick $C,\wt{T_R}>0$ such that $\|u(t)\|_{C^1\Big([-\wt{T_R},\wt{T_R}];X_s\Big)}\leq CR^5$ and $\|u_N(t)\|_{C^1\Big([-\wt{T_R},\wt{T_R}];X_s\Big)}\leq CR^5$ for all $N$.
We now show a product estimate on Besov spaces (for functions not necessarily supported on non-negative frequencies). For every $0< \eps \ll s-\frac12-\frac1p$, we have that 
\begin{align*}
    \|fgh\|_{B^{0}_{p,1}}= &\ \sum_M\|P_M(fgh)\|_{L^p} \\
    \leq& \sum_M\sum_{N_1,N_2,N_3}\|P_M(P_{N_1}fP_{N_2}gP_{N_3}h)\|_{L^p} \\
    \les &\  \|g\|_{B^\varepsilon_{\infty,\infty}}\|h\|_{B^\varepsilon_{\infty,\infty}} \sum_M\Big(\sum_{\substack{N_1,N_2,N_3 \\ N_1\les \max(N_2,N_3) \\ M\les \max(N_2,N_3)}}+\sum_{\substack{N_1,N_2,N_3 \\ N_1\gg \max(N_2,N_3) \\ N_1\sim M}}\Big)\|P_{N_1}f\|_{L^p}N_2^{-\varepsilon}N_3^{-\varepsilon} \\
    \les &\  \|g\|_{B^{s-\frac{1}{2}}_{p,\infty}}\|h\|_{B^{s-\frac{1}{2}}_{p,\infty}}\Big(\sum_MM^{-\varepsilon}\|f\|_{B^0_{p,1}}+\sum_M\sum_{N_1\sim M} \|P_{N_1}f\|_{L^p}\Big) \\
    \les &\  \|f\|_{B^{0}_{p,1}}\|g\|_{B^{s-\frac{1}{2}}_{p,\infty}}\|h\|_{B^{s-\frac{1}{2}}_{p,\infty}}.
\end{align*}
With this, we see for $|t|\leq \widetilde{T_R}$:
\begin{align*}
    \|u(t)-u_N(t)\|_{B^{0}_{p,1}}\leq&\ \|u_0-u_{0,N}\|_{B^{0}_{p,1}}+\int_0^t\|\Pi(|u|^2u(\tau)-|u_N|^2u_N(\tau))\|_{B^{0}_{p,1}}d\tau \\
    \les &\  RN^{\frac{1}{2}-s}+\int_0^t\|(u-u_N)(u+u_N)\overline{u}\|_{B^{0}_{p,1}}+\|\overline{(u-u_N)}u_N^2\|_{B^{0}_{p,1}}d\tau \\
    \les &\  \|u_0\|_{X^s}N^{\frac{1}{2}-s}+R^5|t|\sup_{|\tau|\leq |t|}\|u(\tau)-u_N(\tau)\|_{B^{0}_{p,1}},
\end{align*}
and so it follows that picking a $T_0 = T_0(R) >0$ small enough, we obtain
$$     \|u(t)-u_N(t)\|_{C^0\Big([-T_0,T_0];B^{0,+}_{p,1}\Big)}\les N^{\frac{1}{2}-s} \|u_0\|_{X_s}. $$ 
The analogous $C^2$ bound follows by taking derivatives and a similar argument.
\end{proof}

\begin{lemma}[Commutator estimates]\label{lemma3.3}
    For $N$ dyadic, consider the following trilinear maps
    \begin{gather*}
        \mathcal{C}_N^1[f,g,h]=P_N\Pi(P_{\ll N}(f)\overline{P_{\ll N}(g)}h)-P_{\ll N}(f)\overline{P_{\ll N}(g)}P_N\Pi(h)\\
        \mathcal{C}_N^2[f,g,h]=P_N\Pi\big(\Pi_{L,L,H}(f,g,h)-P_{\ll N}(f)\overline{P_{\ll N}(g)}h\big)
    \end{gather*}
    For $j=1,2$, we have the following estimate
    \begin{equation*}
        \|\mathcal{C}_N^j[f,g,h]\|_{L^{\frac{p}{3}}}\les N^{1-2s}\|f\|_{B^{s-\frac{1}{2}}_{p,\infty}}\|g\|_{B^{s-\frac{1}{2}}_{p,\infty}}\|h\|_{B^{s-\frac{1}{2}}_{p,\infty}},
    \end{equation*}
    which holds for all $f,g,h$ not-necessarily supported on non-negative frequencies.
\begin{proof}
    We compute for $n\in\bb{Z}$
    \begin{align*}
        &(\mathcal{C}_N^1[f,g,h])\ft{\phantom{X}}(n)\\
        =&\ \sum_{\substack{n_1,n_2,n_3\in\bb{Z} \\ n_1-n_2+n_3=n}} [\mathbbm{1}_{\bb{N}_0}(n)\phi_N(n)-\mathbbm{1}_{\bb{N}_0}(n_3)\phi_N(n_3)]\widehat{P_{\ll N}f}(n_1)\overline{\widehat{P_{\ll N}g}(n_2)}\widehat{h}(n_3) \\
        = &\ \mathbbm{1}_{\bb{N}_0}(n)\sum_{\substack{n_1,n_2,n_3\in\bb{Z} \\ n_1-n_2+n_3=n}} [\phi_N(n)-\phi_N(n_3)]\widehat{P_{\ll N}f}(n_1)\overline{\widehat{P_{\ll N}g}(n_2)}\widehat{\Pi P_{\approx N} h}(n_3) \\
        = &\ \frac{\mathbbm{1}_{\bb{N}_0}(n)}{N}\sum_{\substack{n_1,n_2,n_3\in\bb{Z} \\ n_1-n_2+n_3=n}} (n_1-n_2)\int_0^1\phi'\Big(\frac{n_3+\sigma(n_1-n_2)}{N}\Big)d\sigma\widehat{P_{\ll N}f}(n_1)\overline{\widehat{P_{\ll N}g}(n_2)}\widehat{\Pi P_{\approx N} h}(n_3) \\
        = &\ -i\frac{\mathbbm{1}_{\bb{N}_0}(n)}{N}\sum_{\substack{n_1,n_2,n_3\in\bb{Z} \\ n_1-n_2+n_3=n}} m_N(n_1,n_2,n_3)\widehat{(P_{\ll N}f)'}(n_1)\overline{\widehat{P_{\ll N}g}(n_2)}\widehat{\Pi P_{\approx N} h}(n_3) \\
        &+i\frac{\mathbbm{1}_{\bb{N}_0}(n)}{N}\sum_{\substack{n_1,n_2,n_3\in\bb{Z} \\ n_1-n_2+n_3=n}} m_N(n_1,n_2,n_3)\widehat{P_{\ll N}f}(n_1)\overline{\widehat{(P_{\ll N}g)'}(n_2)}\widehat{\Pi P_{\approx N} h}(n_3),
    \end{align*}
    where $m_N$ is the Fourier multiplier given by
    \begin{equation*}
        m_N(\xi_1,\xi_2,\xi_3)=\int_0^1\phi'\Big(\frac{\xi_3+\sigma(\xi_1-\xi_2)}{N}\Big)d\sigma\prod_{j=1}^3\Big(\sum_{M\leq 100N}\phi_M(\xi_j)\Big),
    \end{equation*}
    We note that the $m_N$ are smooth and bounded on $\bb{R}^3$ (uniformly in $N$), and are supported on frequencies $|\xi_{max}|\les N$. Using the fact $\Big(\sum_{M\leq 100N}\phi_M(x)\Big)$ is supported on $(-200N,200N)$, we note that for all multi-indices $\alpha$,
    \begin{equation*}
        |\partial^\alpha m_N(\xi_1,\xi_2,\xi_3)|\les N^{-|\alpha|}\mathbbm{1}_{\{|\xi_{max}|\les N\}}\les \|(\xi_1,\xi_2,\xi_3)\|^{-\alpha},
    \end{equation*}
    and so the $m_N$ are Coifmann-Meyer multipliers, with norms uniform in $N$. So by Proposition \ref{prop:CM}, we have that uniform in $N$,
    \begin{align*}
        \|\mathcal{C}_N^1[f,g,h]\|_{L^{\frac{p}{3}}}&\les\frac{1}{N}(\|P_{\ll N}f'\|_{L^p}\|P_{\ll N}g\|_{L^p}+\|P_{\ll N}f\|_{L^p}\|P_{\ll N}g'\|_{L^p})\|\Pi P_{\approx N} h\|_{L^p} \\
        &\les N^{\frac{1}{2}-s}N^{-1}\sum_{N_1\ll N} N_1N_1^{\frac{1}{2}-s}\|f\|_{B^{s-\frac{1}{2}}_{p,\infty}}\|g\|_{B^{s-\frac{1}{2}}_{p,\infty}}\|h\|_{B^{s-\frac{1}{2}}_{p,\infty}} \\
        &\les N^{1-2s}\|f\|_{B^{s-\frac{1}{2}}_{p,\infty}}\|g\|_{B^{s-\frac{1}{2}}_{p,\infty}}\|h\|_{B^{s-\frac{1}{2}}_{p,\infty}}.
    \end{align*}
    The estimate for $\mathcal{C}_N^2$ is simpler, noting that
    \begin{align*}
    \mathcal{C}_N^2(f,g,h)&=P_N\Pi\left(\sum_{M\approx N}(P_{\ll M}(f)\overline{P_{\ll M}(g)}P_Mh)-P_{\ll N}(f)\overline{P_{\ll N}(g)}h\right) \\
    &=P_N\Pi\left(\sum_{M\approx N}\left(P_{\ll M}(f)\overline{P_{\ll M}(g)}-P_{\ll N}(f)\overline{P_{\ll N}(g)}\right)P_Mh\right) \\
    &=P_N\Pi\Bigg(\sum_{M\approx N}\Big((P_{\ll M}(f)-P_{\ll N}(f))\overline{P_{\ll M}(g)}\\
    &\phantom{=P_N\Pi\Bigg()}+P_{\ll N}(f)\big(\overline{P_{\ll M}(g)}-\overline{P_{\ll N}(g)}\big)\Big)P_Mh\Bigg),
    \end{align*}
    from which the estimate follows.
\end{proof}
\end{lemma}

\begin{proof}[Proof of Proposition \ref{lemma3.4}]
Let $R = 1+ \|u_0\|_{X_s}$. Recall the decomposition $u=X+Y$, where $(u,X,Y)$ solve \eqref{eq:paraLWP}. By Proposition \ref{prop:paraLWP} and Lemma \ref{lemma3.2}, we may pick a constant $C>0$ and $T_0 = T_0(R)>0$ such that  
\begin{equation} \label{paradec1}
   \begin{gathered}
    \|X(t)\|_{C^2\Big([-T_R,T_R];X_s\Big)} \le CR^5,\\
    \|Y(t)\|_{C^2\Big([-T_R,T_R];B^{2s-1,+}_{p/2,\infty}\Big)} \le CR^5,\\ \|u_N(t)\|_{C^2\Big([-T_R,T_R];X_s\Big)}\le CR^5,\\
    \|u(t)-u_N(t)\|_{C^2\Big([-T_R,T_R];B^{0,+}_{p,1}\Big)}\le C N^{\frac{1}{2}-s} R^5,
\end{gathered} 
\end{equation}

for all $N$ dyadic. We also note that $X_N$ as defined in \eqref{XNdef} solves the equation
\begin{equation*}
    \begin{cases}
        i\partial_t X_N=2|P_{\ll N}u_N|^2X_N, \\
        X_N(0)=P_{\approx N}u_0,
    \end{cases}
\end{equation*}
for which one can easily prove the analogous of Proposition \ref{prop:LWP}.
Therefore, up to making $T_0$ smaller, we also have 
\begin{equation}\label{paradec2}
    \|X_N(t)\|_{C^2\Big([-T_0,T_0];X_s\Big)}\leq CR^5
\end{equation}
for all $N$ dyadic.
Then we decompose
\begin{equation*}
    v_N=P_N\Pi(X)-P_N\Pi(X_N)+P_NY.
\end{equation*}
From \eqref{paradec1}, we get
\begin{equation*}
        \|P_NY(t)\|_{C^2\Big([-T_0,T_0];L^\frac{p}{3}\Big)}\les R^5 N^{1-2s}.
    \end{equation*}
and so it suffices to bound $\widetilde{v_N}=P_N\Pi (X-X_N)$. We compute
\begin{align*}
    &\widetilde{v_N}\\
    &= -2i\int_0^tP_N\Pi(\Pi_{L,L,H}(u,u,X))(\tau)-P_N\Pi(|P_{\ll N} u_N|^2X_N)(\tau)d\tau \\
    &=  -2i\int_0^t|P_{\ll N} u_N|^2\widetilde{v_N}d\tau \\
    &\phantom{=\ }-2i\int_0^t{P_N\Pi(|P_{\ll N} u_N|^2(X-X_N))-|P_{\ll N} u_N|^2P_N\Pi(X-X_N)}d\tau \\
    &\phantom{=\ }-2i\int_0^tP_N\Pi(\Pi_{L,L,H}(u,u,X)-|P_{\ll N}u_N|^2X)d\tau.\\
    & = -2i\int_0^t\Big(|P_{\ll N} u_N|^2\wt{v_N}d\tau + {\mathcal{C}_N^1(u_N,u_N,X-X_N)} \\
    &\phantom{= -2i\int_0^t \Big()} +P_N\Pi(\Pi_{L,L,H}(u,u,X)-|P_{\ll N}u_N|^2X)\Big)d\tau.
\end{align*}
First note that, by \eqref{paradec1}, for $|t|\le T_0$,
\begin{equation*}
    \Big|\Big|\int_0^t|P_{\ll N} u_N|^2\widetilde{v_N}d\tau\Big|\Big|_{L^\frac{p}{3}}\les R^{10}T_0\|\widetilde{v_N}\|_{C^0\Big([-T_0,T_0];L^\frac{p}{3}\Big)}.
\end{equation*}
From the commutator estimate in Lemma \ref{lemma3.3}, and \eqref{paradec1}, \eqref{paradec2}, we see that
\begin{equation*}
    \Big|\Big|\int_0^t\mathcal{C}_N^1(u_N,\overline{u_N},X-X_N)d\tau\Big|\Big|_{{C^2\Big([-T_0,T_0];L^\frac{p}{3}\Big)}}\les R^{15}T_0 N^{1-2s}.
\end{equation*}
Next we compute
\begin{align*}
    & P_N\Pi(\Pi_{L,L,H}(u,u,X)-|P_{\ll N}u_N|^2X)\\
    &= P_N\Pi(\Pi_{L,L,H}(u,u,X)-\Pi_{L,L,H}(u_N,u_N,X))\\
    &\phantom{=\ }+P_N\Pi(\Pi_{L,L,H}(u_N,u_N,X)-|P_{\ll N}u_N|^2X)\\
    &= P_N\Pi(\Pi_{L,L,H}(u,u,X)-\Pi_{L,L,H}(u_N,u_N,X)) + \mathcal{C}_N^2(u_N,u_N,X).
\end{align*}
Again from Lemma \ref{lemma3.3}, \eqref{paradec1}, and \eqref{paradec2}, we have 
\begin{align*}
   \Big\|\int_0^t \mathcal{C}_N^2(u_N,u_N,X) d\tau\|_{C^2\Big([-T_0,T_0];L^\frac{p}{3}\Big)}  \les R^{15}T_0 N^{1-2s} 
\end{align*}
By \eqref{paradec1}, and Lemma \ref{lemma3.2}, we have that 
\begin{align*}
    &\|P_N\Pi(\Pi_{L,L,H}(u,u,X)-\Pi_{L,L,H}(u_N,u_N,X))\|_{C^1\Big([-T_0,T_0];L^\frac{p}{3}\Big)}\\ 
    &\les \sum_{\substack{N_1,N_2\ll N_3 \\ N_3\sim N}}\|(P_{N_1}(u-u_N)\overline{P_{N_2}(u_N)}+\overline{P_{N_2}(u-u_N)}P_{N_1}u)P_{N_3}X\|_{C^1\Big([-T_0,T_0];L^\frac{p}{3}\Big)} \\
    &\les R^{10} N^{\frac{1}{2}-s}\|u(t)-u_N(t)\|_{C^1\Big([-T_0,T_0];B^{0}_{p,1}\Big)} \\
    &\les R^{10} N^{1-2s}.
\end{align*}
Therefore
\begin{equation*}
    \Big|\Big|\int_0^tP_N\Pi(\Pi_{L,L,H}(u,u,X)-|P_{\ll N}u_N|^2X)d\tau\Big|\Big|_{{C^2\Big([-T_0,T_0];L^\frac{p}{3}\Big)}}\les R^{15} T_0 N^{1-2s}.
\end{equation*}
Putting all this together we can deduce that, up to making $T_0$ even smaller (but depending only on $R$), 
\begin{equation*}
    \|\widetilde{v_N}\|_{C^2\Big([-T_0,T_0];L^\frac{p}{3}\Big)}\les R^5 N^{1-2s},
\end{equation*}
and so
\begin{equation*}
    \|{v_N}\|_{C^2\Big([-T_0,T_0];L^\frac{p}{3}\Big)}\les R^5 N^{1-2s}
\end{equation*}
as well.
\end{proof}

\subsection{Estimates for \texorpdfstring{$Q_N$}{QN}}
\label{sec:estimates}
We now verify that the decomposition $u(t) = P_N\Pi(X_N)(t) + v_N(t)$ given by Proposition \ref{lemma3.4} allows us to estimate $Q_N(u(t),u(t),u(t),u(t))$ appropriately. For ease of notation, let
\begin{equation*}
    \mathscr{E}_N(t)=\exp\Big(-2i\int_0^t|P_{\ll N}u_N|^2(\tau)d\tau\Big),
\end{equation*}
and
\begin{equation*}
    E_N(t)=P_{\ll N} (\mathscr{E}_N(t)).
\end{equation*}
We will also define, for $M,M'$ dyadic,  
\begin{align*}
    X_{M}^+&=P_M\Pi(X_M) = P_M\Pi\big(\mathscr{E}_M(t)P_{\approx M}u_0\big),\\
    X_{M,M'} &= E_{M'}(t)P_Mu_0 = P_{\ll M'}(\mathscr{E}_{M'}(t)) P_M u_0.
\end{align*}
We have the following estimates. 
\begin{lemma}
    Let $M'\le M$ dyadic, $u_0 \in X_s$, and let $T_0=T_0(\|u_0\|_{X_s})$ be as in Proposition \ref{prop:paraLWP}. Then it holds 
    \begin{equation} \label{XMdecoupling}
        \|X_{M}^+ - X_{M,M'}\|_{C^2([-T_0,T_0],L^4)} \les M^{\frac12-s}(M')^{\frac12-s}(1+\|u_0\|_{X_s})^4 \|u_0\|_{X_s}.
    \end{equation}
\end{lemma}
\begin{proof}
    We write 
    \begin{align*}
        X_{M}^+ - X_{M,M'}
        &= P_M\Pi\big(\mathscr{E}_M(t)P_{\approx M}u_0\big) - P_{\ll M'}(\mathscr{E}_{M'}(t)) P_M u_0 \\
        &= P_M\Pi\big((\mathscr{E}_M(t)-\mathscr{E}_{M'}(t))P_{\approx M}u_0\big) \tag{I} \\
        &\phantom{=\ } + P_M\Pi\big((\mathscr{E}_{M'}(t) - P_{\ll M'}\mathscr{E}_{M'}(t)) P_{\approx M}u_0\big) \tag{II} \\
        &\phantom{=\ } + P_{M} \big(P_{\ll M'}\mathscr{E}_{M'}(t) P_{\approx M}u_0\big) - P_{\ll M'}\mathscr{E}_{M'}(t) P_{M}P_{\approx M}u_0, \tag{III}
    \end{align*}
where in the last line we used that $\Pi(P_{\ll M}\mathscr{E}_{M'}(t) P_{\approx M}u_0) = P_{\ll M}\mathscr{E}_{M'}(t) P_{\approx M}u_0$ and that $P_MP_{\approx M}=P_M$.

To estimate \rm{(I)}, from Lemma \ref{lemma3.2}, we deduce that 
\begin{equation*}
    \|\mathscr{E}_M(t)-\mathscr{E}_{M'}(t)\|_{C^2([-T_0,T_0],L^8)} \les (M')^{\frac12-s} (1+\|u_0\|_{X_s})^4,
\end{equation*}
from which by H\"older's inequality, 
\begin{align*}
    \|\rm{(I)}\|_{C^2([-T_0,T_0],L^4)} &\les \|\mathscr{E}_M(t)-\mathscr{E}_{M'}(t)\|_{C^2([-T_0,T_0],L^8)} \|P_M u_0\|_{L^8}\\
    &\les M^{\frac12-s}(M')^{\frac12-s}(1+\|u_0\|_{X_s})^4 \|u_0\|_{X_s}.
\end{align*}
For \rm{(II)}, note that $\mathscr{E}_{M'}(t)$ solves the equation
$$ i\partial_t \mathscr{E}_{M'}(t) = 2|P_{\ll M'}u_{M'}|^2\mathscr{E}_{M'}(t),$$
from which we immediately deduce that 
\begin{align}
    \|\mathscr{E}_{M'}\|_{C^2([-T_0,T_0],X_s)} &\les 1+  \|u_{M'}\|_{C^0([-T_0,T_0],X_s)}\|u_{M'}\|_{C^1([-T_0,T_0],X_s)} \notag \\
    &\les (1+ \|u_0\|_{X_s}^4) \label{EMbound}
\end{align}
Therefore, we obtain that 
\begin{align*}
   \|\mathscr{E}_{M'}(t) - P_{\ll M'}\mathscr{E}_{M'}(t)\|_{C^2([-T_0,T_0],L^8)} &\les \sum_{L \gtrsim M'} \|P_L \mathscr{E}_{M'}\|_{C^2([-T_0,T_0],L^8)} \\
   &\les \sum_{L \gtrsim M'}L^{\frac12-s}\|\mathscr{E}_{M'}\|_{C^2([-T_0,T_0],X_s)} \\
   &\les (M')^{\frac12-s} (1+\|u_0\|_{X_s})^4.
\end{align*}
Therefore,  
\begin{align*}
     \|\rm{(II)}\|_{C^2([-T_0,T_0],L^4)} &\les \|\mathscr{E}_{M'}(t) - P_{\ll M}\mathscr{E}_{M'}(t)\|_{C^2([-T_0,T_0],L^8)} \|P_M u_0\|_{L^8}\\
    &\les M^{\frac12-s}(M')^{\frac12-s}(1+\|u_0\|_{X_s})^4 \|u_0\|_{X_s}.
\end{align*}
To estimate \rm{(III)}, consider the multiplier 
$$ m_M(n_1,n_2) = M \frac{(\phi_M(n_1+n_2) - \phi_M(n_1))}{n_2} \phi_{\les M}(n_1)\phi_{\les M}(n_2). $$
It is easy to check that this is a Coifman-Meyer multiplier with $\|m_M\|_{CM,2} \les 1.$ In particular, we deduce that for every $f,g:\T\to\C$ supported on frequencies $|n|\lesssim M$ , we have 
$$ \|P_M(fg) - fP_M(g)\|_{L^4} \les M^{-1} \|f'\|_{L^8} \|g\|_{L^8}. $$
Therefore, 
\begin{align*}
    \|\rm{(III)}\|_{C^2([-T_0,T_0],L^4)} &\les \sum_{L\ll M'} M^{-1} \| \partial_x P_{L} \mathscr{E}_{M'}\|_{C^2([-T_0,T_0],L^8)} \|P_{\approx M} u_0\|_{L^8} \\
    &\les \sum_{L\ll M'} \frac{L}{M}\|P_{L} \mathscr{E}_{M'}\|_{C^2([-T_0,T_0],L^8)}\|P_{\approx M} u_0\|_{L^8} \\
    &\les \sum_{L\ll M'} \frac{L}{M} L^{\frac12-s} M^{\frac12-s} \|\mathscr{E}_{M'}\|_{C^2([-T_0,T_0],X_s)} \|u_0\|_{X_s} \\
    &\les (M')^{\frac32 -s} M^{-\frac12 -s } (1+\|u_0\|_{X_s})^4 \|u_0\|_{X_s} \\
    &\les M^{\frac12-s}(M')^{\frac12-s}(1+\|u_0\|_{X_s})^4 \|u_0\|_{X_s}.
\end{align*}
\end{proof}
\begin{lemma}
     Let $M'\le M$ dyadic, $u_0 \in X_s$, and let $T_0=T_0(\|u_0\|_{X_s})$ be as in Proposition \ref{prop:paraLWP}. Then it holds 
     \begin{equation} \label{XMcommutator}
         \| \partial_x^2P_N^2( X_{M,M'}) - E_{M'(t)} \partial_x^2 P_N^2 P_M u_0\|_{C^2([-T_0,T_0],L^4)} \les N^{\frac32-s}(M')^{\frac32-s}(1+\|u_0\|_{X_s})^4 \|u_0\|_{X_s}.
     \end{equation}
\end{lemma}
\begin{proof}
    Recall that 
    $$ X_{M,M'} = E_{M'}(t) P_M u_0=P_{\ll M'}(\mathscr{E}_{M'}(t))P_M u_0.$$ 
    Moreover, as in the previous proof, we have that (see \eqref{EMbound})
    $$ \|\mathscr{E}_{M'}\|_{C^2([-T_0,T_0],X_s)} \les (1+ \|u_0\|_{X_s}^4). $$
    For $f,g$ smooth, it is easy to check that the bilinear map 
    $$ \partial_x^2 P_{N}^2((P_{\ll 4M}f)g) - P_{\ll 4M}f\cdot \partial_x^2 P_{N}^2(g)  $$
    is given by the multiplier 
    \begin{align*}
        c_M(n_1,n_2) &= (n_1^2 \phi_{N}(n_1)^2 - (n_1+n_2)^2\phi_{N}(n_1+n_2)^2)\phi_{\ll 4M}(n_2) \\
        &= Mn_2 \cdot \phi_{\ll 4M}(n_2) \frac{n_1^2 (\phi_{N}(n_1)^2 - \phi_{N}(n_1+n_2)^2)}{Mn_2} \\
        &\phantom{=\ }-(2n_1n_2 +n_2^2) \phi_{N}(n_1+n_2)^2 \phi_{\ll 4M}(n_2).
    \end{align*}
    Via a scaling argument, it is easy to check that for $N\sim M$, we have
    $$\Big\| \phi_{\ll 4M}(n_2) \frac{n_1^2 (\phi_{N}(n_1)^2 - \phi_{N}(n_1+n_2)^2)}{Mn_2} \Big\|_{CM,2} \les 1.$$
    Therefore, by Coifman-Meyer's theorem, we obtain that for $M\sim N$,
    \begin{align*}
        & \|\partial_x^2 P_{N}^2((P_{\ll 4M}f)g) - P_{\ll 4M}f\cdot \partial_x^2 P_{N}^2(g)\|_{L^4} \\
        &\les M \|f'\|_{L^8}\|g\|_{L^8} + \|P_{N}^2 ((P_{\ll 4M}f') g')\|_{L^4} + \|P_{N}^2 ((P_{\ll 4M}f'') g)\|_{L^4}\\
        &\les N \|f'\|_{L^8}\|g\|_{L^8} + \|f'\|_{L^8}\|g'\|_{L^8} + \|f''\|_{L^8}\|g\|_{L^8}.
    \end{align*}
    We now apply this to $f= \partial_t^k P_L \mathscr{E}_{M'}(t)$ for $k\in \{0,1,2\}$  and $L \ll M'$, and $g = P_M u_0$. Note that when $M\not\approx N$, the LHS of \eqref{XMcommutator} is equal to $0$. Therefore, we can assume that $M \approx N$, and obtain 
    \begin{align*}
        & \| \partial_x^2 P_N^2 X_{M,M'} - E_{M'(t)} \partial_x^2 P_N^2P_M u_0\|_{C^2([-T_0,T_0],L^4)} \\
        &\les  \sum_{L\ll M'} (M L + L^2) \|P_L \mathscr{E}_{M'}(t)\|_{C^2([-T_0,T_0],L^4)} \|P_Mu_0\|_{L^4}\\
        &\les \sum_{L\ll M'} ML \cdot L^{\frac12 -s} M^{\frac12-s} 
        \|\mathscr{E}_{M'}(t)\|_{C^2([-T_0,T_0],X_s)} \|u_0\|_{X_s}\\
        &\les N^{\frac32-s}(M')^{\frac32-s}(1+\|u_0\|_{X_s})^4 \|u_0\|_{X_s}
    \end{align*}
    \end{proof}
The previous deterministic estimates allow us to isolate the ``diverging" part of $Q_N$ for general $u_0 \in X_s$.\footnote{As we will see soon, this is an improper name, since the ``diverging" part of $Q_N$ will actually be of smaller order as a power of $N$. Nevertheless, for general functions $u_0 \in X_s$, this is not the case, and the ``diverging" part can actually be much bigger than the other terms.} For ease of notation we introduce the conjugation relation
\begin{equation*}
    \mathfrak c_j(z)=\begin{cases}
    z,\text{ for $j$ even,} \\
    \overline{z},\text{ for $j$ odd.}
    \end{cases}
\end{equation*}
\begin{lemma}[Decoupling of $Q_N$]\label{lemma3.5}
     Let $u_0 \in X_s$, and let $T_0=T_0(\|u_0\|_{X_s})$ be as in Proposition \ref{prop:paraLWP}. For $N_1,N_2,N_3$ dyadic, define  
     \begin{equation} \label{YNdef}
     \begin{aligned}
     &Y_{N_1,N_2,N_3}(x)\\
     &:= \frac i2
     \sum_{n_1,n_2,n_3 \in \Z} (n_1^2 \phi_{N}(n_1)^2 - n_2^2 \phi_{N}(n_2)^2 + n_3^2 \phi_{N}(n_3)^2)\prod_{j=1}^3 \phi_{N_j}(n_j) \mathfrak c_j(\ft{u_0}(n_j)) e^{i(n_1-n_2+n_3)x}.         
     \end{aligned}
     \end{equation}  
     Then, for all $N$ dyadic, we have
\begin{align*}
& \begin{multlined}
    \Big|\Big|Q_N(u,u,u,u)(t)\\-4\sum_{\substack{N_1,N_2,N_3\gg N_4 
 \\ N^{(1)}\sim N^{(2)} \gtrsim N\gg N_4\\ N^{(3)} \les N}}
 \int_{\T} |E_{N^{(3)}}|^2E_{N^{(3)}}(t,x) \cj{P_{N_4}u_{\min(N,N^{(3)})}(t,x)} Y_{N_1,N_2,N_3}(x)\Big|\Big|_{C^2([-T_0,T_0])}
\end{multlined} \\
 &\les N^{4-4s}(1+\|u_0\|_{X_s}^{20}),
     \end{align*}
where we recall that $N^{(1)} \ge N^{(2)}\ge N^{(3)}\ge N^{(4)} $ is a reordering of $N_1,\dotsc,N_4$.\footnote{It is actually possible to improve the power in the RHS to $(1+\|u_0\|_{X_s}^{8})$. However, this is inessential to the argument of this paper.}
\end{lemma}
\begin{proof}
    By symmetry, we decompose
    \begin{align*}
       & Q_N(u,u,u,u)\\
    &=  
    4\sum_{\substack{N_1,N_2,N_3\gg N_4 
 \\ N\gg N_4}}Q_N(P_{N_1}u,P_{N_2}u,P_{N_3}u,P_{N_4}u)  \\ 
 &\phantom{=\ }+\sum_{\substack{N_1,N_2,N_3,N_4 \\ \min(N,N^{(3)})\les N^{(4)}}}Q_N(P_{N_1}u,P_{N_2}u,P_{N_3}u,P_{N_4}u) \\
 &= 
 4\sum_{\substack{N_1,N_2,N_3\gg N_4 
 \\ N^{(1)}\sim N^{(2)} \gtrsim N\gg N_4\\ N^{(3)} \les N}} \int_{\T} |E_{N^{(3)}}|^2E_{N^{(3)}}(t,x) \cj{P_{N_4}u_{\min(N,N^{(3)})}(t,x)} Y_{N_1,N_2,N_3}(x)\\
  &\phantom{=\ } + 4\smash{\sum_{\substack{N_1,N_2,N_3\gg N_4 
 \\ N^{(1)}\sim N^{(2)} \gtrsim N\gg N_4\\ N^{(3)} \les N}}}\Big(Q_N(X_{N_1,N^{(3)}},X_{N_2,N^{(3)}},X_{N_3,N^{(3)}},P_{N_4}u_{\min(N,N^{(3)})})\\
 &\phantom{=\ } \hspace{100pt} - \int_{\T} |E_{N^{(3)}}|^2E_{N^{(3)}}(t,x) \cj{P_{N_4}u_{\min(N,N^{(3)})}(t,x)} Y_{N_1,N_2,N_3}(x) \Big)\tag{I}\\
  &\phantom{=\ } +4\smash{\sum_{\substack{N_1,N_2,N_3\gg N_4 
 \\ N\gg N_4}}} \Big(Q_N(X_{N_1}^+,X_{N_2}^+,X_{N_3}^+,P_{N_4}u_{\min(N,N^{(3)})})  \\
&\phantom{=\ } \hspace{100pt}
-Q_N(X_{N_1,N^{(3)}},X_{N_2,N^{(3)}},X_{N_3,N^{(3)}},P_{N_4}u_{\min(N,N^{(3)})})\Big) \tag{II}\\
  &\phantom{=\ } +
 4\smash{\sum_{\substack{N_1,N_2,N_3\gg N_4 
 \\ N\gg N_4}}}\Big(Q_N(P_{N_1}u,P_{N_2}u,P_{N_3}u,P_{N_4}u)\\
&\phantom{=\ } \hspace{100pt}
 - Q_N(X_{N_1}^+,X_{N_2}^+,X_{N_3}^+,P_{N_4}u_{\min(N,N^{(3)})})\Big) 
 \tag{III}\\
  &\phantom{=\ }+\sum_{\substack{N_1,N_2,N_3,N_4 \\ \min(N,N^{(3)})\les N^{(4)}}}Q_N(P_{N_1}u,P_{N_2}u,P_{N_3}u,P_{N_4}u)\tag{IV}.
    \end{align*}
Let $T_0 = T_0(R)$ be as in Proposition \ref{lemma3.4}.
From \eqref{QN_CM} and \eqref{paraLWPestimates}, we have that 
\begin{align*}
    &\|\mathrm{IV}\|_{C^2([-T_0,T_0])} \\
    & \les \sum_{\substack{N_1,N_2,N_3,N_4 \\ \min(N,N^{(3)})\les N^{(4)}\\ N^{(1)}\sim N^{(2)} \gtrsim N}} N \min(N,N^{(3)}) \prod_{j=1}^4 N_j^{\frac12 -s }\|u(t)\|_{C^2([-T_0,T_0],B^{{s-\frac12}, +}_{4,\infty})}^4 \\
    &\les N\sum_{\substack{N_1,N_2,N_3,N_4 \\ \min(N,N^{(3)})\les N^{(4)}\\ N^{(1)}\sim N^{(2)} \gtrsim N}} \min(N,N^{(3)}) (N^{(1)})^{\frac12-s}(N^{(2)})^{\frac12-s}(N^{(3)})^{\frac12-s}(N^{(4)})^{\frac12-s} (1+\|u_0\|_{X_s}^{20})\\
    &\les N\sum_{\substack{N^{(3)} \ge N^{(4)} \\ \min(N,N^{(3)})\les N^{(4)}}} \min(N,N^{(3)})\max(N,N^{(3)})^{1-2s} (N^{(3)})^{\frac12 - s}(N^{(4)})^{\frac12-s} (1+\|u_0\|_{X_s}^{20}) \\
    &\les \|u_0\|_{X_s}^4\Big(N^{2-2s}\sum_{N^{(3)} \sim N^{(4)} \les N}  (N^{(3)})^{\frac32 - s}(N^{(4)})^{\frac12-s} +N^2\sum_{N^{(3)} \gtrsim N^{(4)} \gtrsim N} (N^{(3)})^{\frac32 - 3s}(N^{(4)})^{\frac12-s}\Big) \\
    &\les (N^{2-2s}+ N^{4-4s}) (1+\|u_0\|_{X_s}^{20}) \les N^{4-4s} (1+\|u_0\|_{X_s}^{20}).
\end{align*}
In order to estimate \rm{(III)}, for $j=1,2,3$, decompose $P_{N_j}u(t) = X^+_{N_j} + v_{N_j}(t)$, as in Proposition \ref{lemma3.4}. Recalling that 
$$\|v_{N_j}(t)\|_{C^{2}([-T_0,T_0], L^\frac p3 )} \les N^{1-2s}\|u_0\|_{X_s}, $$
and that $\frac p3 \ge 4$, by \eqref{QN_CM} and Lemma \ref{lemma3.2} we obtain 
\begin{align*}
    &\|\mathrm{(III)}\|_{C^2([-T_0,T_0])} \\
    &\les \sum_{\substack{N_1,N_2,N_3\gg N_4 
 \\ N^{(1)} \sim N^{(2)} \gtrsim N \gg N_4}} N \min(N,N^{(3)})^{\frac32-s} (N^{(1)})^{\frac12-s}(N^{(2)})^{\frac12-s}(N^{(3)})^{\frac12-s}(N^{(4)})^{\frac12-s} (1+\|u_0\|_{X_s}^{20}) \\
 &\les \sum_{N^{(1)} \sim N^{(2)} \gtrsim N } (N^{(1)})^{\frac12-s}(N^{(2)})^{\frac12-s} N^{3-2s} (1+\|u_0\|_{X_s}^{20}) \\
 &\les N^{4-4s} (1+\|u_0\|_{X_s}^{20}).
\end{align*}
For estimating \rm{(II)}, we proceed similarly, using \eqref{XMdecoupling} in lieu of the estimate on $v_N$. Using \eqref{QN_CM} once again, we obtain 
\begin{align*}
      &\|\mathrm{(II)}\|_{C^2([-T_0,T_0])} \\  
      &= 4\Big\|\smash{\sum_{\substack{N_1,N_2,N_3\gg N_4 
 \\ N\gg N_4}}} \Big(Q_N(X_{N_1}^+,X_{N_2}^+,X_{N_3}^+,P_{N_4}u_{\min(N,N^{(3)})})  \\
&\phantom{=\ } \hspace{70pt}
-Q_N(P_{\approx N_1}X_{N_1,N^{(3)}},P_{\approx N_2}X_{N_2,N^{(3)}},P_{\approx N_3}X_{N_3,N^{(3)}},P_{N_4}u_{\min(N,N^{(3)})})\Big)\Big\|_{C^2_t}\\
          &\les \sum_{\substack{N_1,N_2,N_3\gg N_4 
 \\ N^{(1)} \sim N^{(2)} \gtrsim N \gg N_4}} N \min(N,N^{(3)}) (N^{(1)})^{\frac12-s}(N^{(2)})^{\frac12-s}(N^{(3)})^{1-2s}(N^{(4)})^{\frac12-s} (1+\|u_0\|_{X_s}^{20}) \\
  &\les N^{4-4s} (1+\|u_0\|_{X_s}^{20}).
\end{align*}
Finally, we move to the estimate of \rm{(I)}. Recall the definition of $\Psi_N$ in \eqref{PsiNdef}, which we now view as inducing a differential operator in spatial representation. Then, by definition of $X_{M,M'}$, we have that 
\begin{align*}
&\begin{multlined}
Q_N(X_{N_1,N^{(3)}},X_{N_2,N^{(3)}},X_{N_3,N^{(3)}},P_{N_4}u_{\min(N,N^{(3)})}) \\
- \int_{\T} |E_{N^{(3)}}|^2E_{N^{(3)}}(t,x) \cj{P_{N_4}u_{\min(N,N^{(3)})}(t,x)} Y_{N_1,N_2,N_3}(x)dx
\end{multlined}\\
&= \frac{i}{2}\int_{\T} \overline{P_{N_4}u_{\min(N,N^{(3)})}} \sum_{j=1}^3 (-1)^{j} \Big(\mathfrak c_j\big(\partial_x^2 P_N^2 {X_{N_j,N^{(3)}}}\big) \prod_{k \in \{1,2,3\}\setminus \{j\}} \mathfrak c_k\big({X_{N_k,N^{(3)}}}\big) \\ 
&\hspace{130pt}- \mathfrak c_j\big(E_{N^{(3)}}(t)\partial_x^2 P_N^2 P_{N_j} u_0\big) \prod_{k \in \{1,2,3\}\setminus \{j\}} \mathfrak c_k\big(E_{N^{(3)}}(t)P_{N_k} u_0\big)\Big)dx\\
&\begin{multlined}
    = \frac{i}{2}\int_{\T} \overline{P_{N_4}u_{\min(N,N^{(3)})}} \sum_{j=1}^3 (-1)^{j} \mathfrak c_j\big(\partial_x^2  P_N^2 {X_{N_j,N^{(3)}}} - E_{N^{(3)}}(t)\partial_x^2 P_N^2 P_{N_j} u_0\big) \\
    \times \prod_{k \in \{1,2,3\}\setminus \{j\}} \mathfrak c_k\big({X_{N_k,N^{(3)}}}\big) dx.
\end{multlined}
\end{align*}
By \eqref{XMcommutator}, we obtain that 
\begin{align*}
    &\begin{multlined}
\Big\|Q_N(X_{N_1,N^{(3)}},X_{N_2,N^{(3)}},X_{N_3,N^{(3)}},P_{N_4}u_{\min(N,N^{(3)})}) \\
- \int_{\T} |E_{N^{(3)}}|^2E_{N^{(3)}}(t,x) \cj{P_{N_4}u_{\min(N,N^{(3)})}(t,x)} Y_{N_1,N_2,N_3}(x)dx\Big\|_{C^2([-T_0,T_0])}
\end{multlined}\\
&\les \|P_{N_4} u_{\min(N,N^{(3)})}\|_{C^2([-T_0,T_0], L^4)} \|E_{N^{(3)}}\|_{C^2([-T_0,T_0],L^\infty)}^2 \\
&\hspace{30pt}\times \sum_{j=1}^3 
\big\|\partial_x^2  {P_N X_{N_j,N^{(3)}}} - E_{N^{(3)}}(t)\partial_x^2 P_N P_{N_j} u_0)\|_{C^2([-T_0,T_0], L^4)} \prod_{k \in \{1,2,3\}\setminus \{j\}} \|P_{N_k} u_0\|_{L^4}\\
&\les (1+\|u_0\|_{X_s}^{20}) (N^{(3)})^{\frac 32-s} N^{\frac32 -s} (N^{(2)}N^{(3)})^{\frac 12-s} N_4^{\frac12 -s} \\
&\les N^{\frac32-s} (N^{(1)})^{\frac12-s} (N^{(3)})^{2-2s}N_4^{\frac12 -s}.
\end{align*}

Therefore, 
\begin{align*}
    |\rm{(I)}| &\les \sum_{\substack{N_1,N_2,N_3\gg N_4 
 \\ N^{(1)}\sim N^{(2)} \gtrsim N\gg N_4\\ N^{(3)} \les N}}  N^{\frac32-s} (N^{(1)})^{\frac12-s} (N^{(3)})^{2-2s}N_4^{\frac12 -s} \\
 &\les N^{4-4s}.
\end{align*}

\end{proof}
In view of this estimate, we are left with bounding the expression. 
$$ \sum_{\substack{N_1,N_2,N_3\gg N_4 
 \\ N^{(1)}\sim N^{(2)} \gtrsim N\gg N_4\\ N^{(3)} \les N}}
 \int_{\T} |E_{N^{(3)}}|^2E_{N^{(3)}}(t,x) \cj{P_{N_4}u_{\min(N,N^{(3)})}(t,x)} Y_{N_1,N_2,N_3}(x)dx. $$
In principle, if the various functions in the expression above are general elements of the Besov space $B^{s-\frac12,+}_{p,\infty}$, the best estimate that one could prove is 
$$\sum_{\substack{N_1,N_2,N_3\gg N_4 
 \\ N^{(1)}\sim N^{(2)} \gtrsim N\gg N_4\\ N^{(3)} \les N}}
 \int_{\T} |E_{N^{(3)}}|^2E_{N^{(3)}}(t,x) \cj{P_{N_4}u_{\min(N,N^{(3)})}(t,x)} Y_{N_1,N_2,N_3}(x)dx \les N^{\frac72 - 3s}, $$
 which is not sufficient for our goals. However, exploiting the (random) structure of $u_0$, we can prove the following.
\begin{lemma}[Random estimate for $Q_N$]\label{lemma3.6} Fix $R>0$, and let $T_0(R)$ be as in Proposition \ref{lemma3.4}. For every $0\le |t| \le T_0(R)$, and for every $k\in\{0,1,2\}$, we have that 
\begin{equation} \label{QNrandom}
\begin{aligned}
    &\E\bigg[\Big|\frac{d^k}{dt^k}\Big(\sum_{\substack{N_1,N_2,N_3\gg N_4 
 \\ N^{(1)}\sim N^{(2)} \gtrsim N\gg N_4\\ N^{(3)} \les N}}
 \hspace{-10pt}
 \int_{\T} |E_{N^{(3)}}|^2E_{N^{(3)}}(t,x) \cj{P_{N_4}u_{\min(N,N^{(3)})}(t,x)} Y_{N_1,N_2,N_3}(x)dx\Big)\Big|\1_{B_R}(u_0)\bigg] \\
 &\les (1+R^{20}) N^{3 - 3s}. 
\end{aligned}
\end{equation}
\end{lemma}
\begin{proof}
Note that, by definition, the function $Y_{N_1,N_2,N_3}$ does not depend on time. Therefore, when we take derivatives, we obtain that 
\begin{align*}
& \frac{d^k}{dt^k}\Big(
 \int_{\T} |E_{N^{(3)}}|^2E_{N^{(3)}}(t,x) \cj{P_{N_4}u_{\min(N,N^{(3)})}(t,x)} Y_{N_1,N_2,N_3}(x)dx\Big)\\
& = \int_\T f_1 f_2 f_3 f_4 Y_{N_1,N_2,N_3} dx\\
\end{align*} 
where $f_1,\dotsc f_4$ are functions of $P_{\ll N^{(3)}}u_0$ which satisfy 
$$ \|f_j\|_{C^2([-T_0,T_0],X_s)} \les 1 + R^{5} $$
as long as $\|P_{\ll N^{(3)}}u_0 \|_{X_s} \les R$, 
and have Fourier support in $\{n: |n| \ll 2N^{(3)}\}$. In particular, we have that 
$$ \int_\T f_1 f_2 f_3 f_4 Y_{N_1,N_2,N_3} dx = \int_\T f_1 f_2 f_3 f_4 P_{\ll 16N^{(3)}} Y_{N_1,N_2,N_3} dx, $$
and that $Y_{N_1,N_2,N_3}$ is independent from $(f_1,\dotsc,f_4)$, since $Y$ is a function of $P_{N_1}u_0, P_{N_2}u_0, P_{N_3}u_0$, which are independent from $P_{\ll N^{(3)}} u_0$.

For $ |k| \ll 16 N^{(3)}$, we now estimate $\E|(Y_{N_1,N_2,N_3})\ft{\phantom{X}}(k)|^2$. By \eqref{YNdef}, we have 
\begin{align*}
    &\E|(Y_{N_1,N_2,N_3})\ft{\phantom{X}}(k)|^2\\
    & = \E\Big|
     \sum_{n_1-n_2+n_3 = k} (n_1^2 \phi_{N}(n_1)^2 - n_2^2 \phi_{N}(n_2)^2 + n_3^2 \phi_{N}(n_3)^2)\prod_{j=1}^3 \phi_{N_j}(n_j) \mathfrak c_j(\ft{u_0}(n_j))\Big|^2 \\
     & \les \sum_{\substack{n_1-n_2+n_3 = k \\ n_2\neq n_1,n_3}} \frac{(n_1^2\phi_{N}(n_1)^2 - n_2^2 \phi_{N}(n_2)^2 + n_3^2 \phi_{N}(n_3)^2)^2 \prod_{j=1}^3 \phi_{N_j}(n_j)^2}{\jb{n_1}^{2s}\jb{n_2}^{2s}\jb{n_3}^{2s}} \\ 
     &\les N^2 \min(N,N^{(3)})^2 (N_1N_2N_3)^{-2s} N^{(1)}{N^{(3)}}\\
     & \les N^2 \min(N,N^{(3)})^2 (N^{(3)})^{1-2s} (N^{(1)})^{1-4s}.
\end{align*}
Here we used that under the condition $|n_1-n_2+n_3| =|k|  \ll N^{(3)}$, we have the bound 
$$|n_1^2\phi_{N}(n_1)^2 - n_2^2 \phi_{N}(n_2)^2 + n_3^2 \phi_{N}(n_3)^2| \les N \min(N,N^{(3)}), $$
which follows from Lemma \ref{lemma2.0.2}, where we put $n_4 = k$.\footnote{In principle, Lemma \ref{lemma2.0.2} does not hold when $n_4 < 0$, while it is certainly possible for us to have $k <0$. However, one can check that in the the proof of Lemma \ref{lemma2.0.2}, the lowest frequency is allowed to be negative without any modification in the argument.}

Moreover, for $k\neq h$ we have that 
\begin{align*}
    \E\Big[(Y_{N_1,N_2,N_3})\ft{\phantom{X}}(k) \cj{(Y_{N_1,N_2,N_3})\ft{\phantom{X}}(h)}\Big] = 0.
\end{align*}
Indeed, it is easy to check that $k\neq h$ implies that we cannot pair the frequencies $\ft{u_0(n_j)}, \ft{u_0(n_j')}$ so that $\{n_1,n_3\} = \{n_1', n_3'\}$, $n_2 = n_2'$ and $n_1-n_2+n_3 =k$, $n_1'-n_2'+n_3' = h$, and so the expectation must vanish. 

Therefore, by Plancherel, recalling that $Y_{N_1,N_2,N_3}$ is independent from $P_{\ll 2N^{(3)}} u_0$ and $(f_j)_{j=1,\dotsc,4}$ are functions of $P_{\ll N^{(3)}} u_0$, for a universal constant $C>0$, we have 
\begin{align*}
    &\E\Big[\Big|\Big(\int_\T f_1 f_2 f_3 f_4 Y_{N_1,N_2,N_3} dx\Big)\1_{B_R}(u_0)\Big|^2\Big] \\
    &\les \E\Big[\Big|\Big(\int_\T f_1 f_2 f_3 f_4 Y_{N_1,N_2,N_3} dx\Big)\1_{B_{CR}}(P_{\ll N^{(3)}} u_0)\Big|^2\Big] \\
    &\les \E\Big[\Big| \sum_{|k| \ll 16 N^{(3)}} \Big(\prod_{j=1}^4 f_j\Big)\ft{\phantom{X}}(k) \cj{(Y_{N_1,N_2,N_3})\ft{\phantom{X}}(k)} \Big|^2 \1_{B_{CR}}(P_{\ll N^{(3)}} u_0)\Big] \\
    & = \sum_{|k| \ll 16 N^{(3)}} \E\Big[\Big|  \Big(\prod_{j=1}^4 f_j\Big)\ft{\phantom{X}}(k)\Big|^2 \1_{B_{CR}}(P_{\ll N^{(3)}} u_0)\Big] \E|(Y_{N_1,N_2,N_3})\ft{\phantom{X}}(k)|^2\\
    &\les \sum_{|k| \ll 16 N^{(3)}} \E\Big[\Big|  \Big(\prod_{j=1}^4 f_j\Big)\ft{\phantom{X}}(k)\Big|^2 \1_{B_{CR}}(P_{\ll N^{(3)}} u_0)\Big] N^2 \min(N,N^{(3)})^2 (N^{(3)})^{1-2s} (N^{(1)})^{1-4s} \\
    &= \E\Big[\Big\|\prod_{j=1}^4 f_j\Big\|_{L^2}^2 \1_{B_{CR}}(P_{\ll N^{(3)}} u_0)\Big] N^2 \min(N,N^{(3)})^2 (N^{(3)})^{1-2s} (N^{(1)})^{1-4s}.
\end{align*}
Therefore, 
\begin{align*}
&\E\bigg[\Big|\frac{d^k}{dt^k}\Big(\sum_{\substack{N_1,N_2,N_3\gg N_4 
 \\ N^{(1)}\sim N^{(2)} \gtrsim N\gg N_4\\ N^{(3)} \les N}}
 \hspace{-10pt}
 \int_{\T} |E_{N^{(3)}}|^2E_{N^{(3)}}(t,x) \cj{P_{N_4}u_{\min(N,N^{(3)})}(t,x)} Y_{N_1,N_2,N_3}(x)dx\Big)\Big|\1_{B_R}(u_0)\bigg] \\
 &\les \sum_{\substack{N_1,N_2,N_3\gg N_4 
 \\ N^{(1)}\sim N^{(2)} \gtrsim N\gg N_4\\ N^{(3)} \les N}} \hspace{-10pt} \E\Big[\Big\| |E_{N^{(3)}}|^2E_{N^{(3)}} \cj{P_{N_4}u_{\min(N,N^{(3)})}}  \Big\|_{C^2([-T_0(R),T_0(R), L^2)} \1_{B_{CR}}(P_{\ll N^{(3)}} u_0)\Big] \\
 &\phantom{\les  \sum_{\substack{N_1,N_2,N_3\gg N_4 
 \\ N^{(1)}\sim N^{(2)} \gtrsim N\gg N_4\\ N^{(3)} \les N}}  } \times N \min(N,N^{(3)}) (N^{(3)})^{\frac12-s} (N^{(1)})^{\frac12-2s} \\
 &\les (1+R^{20}) \sum_{\substack{N_1,N_2,N_3\gg N_4 
 \\ N^{(1)}\sim N^{(2)} \gtrsim N\gg N_4\\ N^{(3)} \les N}} N_4^{\frac12 -s}N \min(N,N^{(3)}) (N^{(3)})^{\frac12-s} (N^{(1)})^{\frac12-2s} \\
 &\les (1+R^{20}) N^{3 - 3s}.
\end{align*}

\end{proof}

\subsection{Singularity of the evolved measure}

We are now ready to prove Proposition \ref{thm2.0.1} and Proposition \ref{thm:sing}.

\begin{proof}[Proof of Proposition \ref{thm2.0.1}]
We first show that $g(x,y) > 0$ implies that $g(y,x) < 0$. Indeed, by \eqref{gdef}, we have that 
\begin{align*}
    g(y,x) &= \liminf_{N\to \infty} \frac{\|P_Ny\|^2_{\dot{H}^1}-\|P_Nx\|^2_{\dot{H}^1}}{(4s-3)N^{4-4s}} \\
    &\le \limsup_{N\to \infty} \frac{\|P_Ny\|^2_{\dot{H}^1}-\|P_Nx\|^2_{\dot{H}^1}}{(4s-3)N^{4-4s}}\\
    &= \limsup_{N\to \infty} -\Big(\frac{\|P_Nx\|^2_{\dot{H}^1}-\|P_Ny\|^2_{\dot{H}^1}}{(4s-3)N^{4-4s}}\Big) \\
    &= - \liminf_{N\to \infty} \frac{\|P_Nx\|^2_{\dot{H}^1}-\|P_Ny\|^2_{\dot{H}^1}}{(4s-3)N^{4-4s}}\\
    &= -g(x,y) < 0.
\end{align*}
We now move to showing the existence of $\tau(u_0)>0$ for $\mu_s$-a.e.\ $u_0 \in X_s$. We simply pick 
$$ \tau(u_0) = \min(T_0(\|u_0\|_{X_s}), \delta(\|u_0\|_{X_s}, \|u_0\|_{L^2}^2)),$$
where $T_0$ is as in Proposition \ref{prop:paraLWP}, and $\delta \ll 1$ is to be determined later in the proof. For ease of notation, let 
$$ h_N(t) = \frac{\|P_N(\Phi_t(u_0))\|^2_{\dot{H}^1}-\|P_Nu_0\|^2_{\dot{H}^1}}{N^{4-4s}}. $$
By Proposition \ref{prop:LWP}, we deduce that $h_N(t) \in C^3([-T_0,T_0])$. Therefore, by Taylor's theorem with integral remainder, we have that 
\begin{equation}\label{hNexpansion}
   \begin{aligned}
    h_N(t) &= h_N(0) + th_N'(0) + \frac{t^2}2 h_N''(0) + \int_0^t h_N'''(\sigma) \frac{\sigma^3}{6} d\sigma \\
    &= tF_N(u_0) + \frac{t^2}2 G_N(u_0) + N^{4s-4}\int_0^t \Big(\frac{d^2}{d\s^2}Q_N(u(\s),u(\s),u(\s),u(\s))\Big) \frac{\sigma^3}{6} d\sigma,
\end{aligned} 
\end{equation}
where $F_N,G_N$ and $Q_N$ are defined in \eqref{FN}, \eqref{GN} and \eqref{QNdef} respectively. By Lemma \ref{lemma3.5} and Lemma \ref{lemma3.6}, we can write 
$$ N^{4s-4}\Big(\frac{d^2}{d\s^2}Q_N(u(\s),u(\s),u(\s),u(\s))\Big) = a_1(\s) + a_2(\s),$$ 
with 
$$\|a_1(\s)\|_{L^\infty([-T_0,T_0])} \les 1 + \|u_0\|_{X_s}^{20}, \quad \E|a_2(\s)\1_{\{u_0\in B_R\}}| \les N^{s-1}(1+R^{20}).$$
In particular, since $N^{s-1}$ is summable, we deduce that 
\begin{align*}
    \lim_{N\to \infty} \int_{-T_0(u_0)}^{T_0(u_0)} |a_2(\s)| \frac{\sigma^3}{6} d\sigma = 0
\end{align*}
for $\mu_s$-a.e.\ $u_0 \in X_s$, and for every $|t| < \tau(u_0)$,
\begin{align*}
   \int_{-|t|}^{|t|} |a_1(\s)| \frac{\sigma^3}{6} d\sigma \les t^3 (1+\|u_0\|_{X_s})^{20} \les t^2 \delta (1+\|u_0\|_{X_s})^{20}.   
\end{align*}
Therefore, we can pick $\dl \ll |4s-3| (1+\|u_0\|_{X_s})^{-20}\|u_0\|_{L^2}^2 $, so we can guarantee that 
\begin{align*}
   \int_{-|t|}^{|t|} |a_1(\s)| \frac{\sigma^3}{6} d\sigma \le |4s-3|I_s\|u_0\|_{L^2}^2 t^2,   
\end{align*}
where $I_s$ is defined in \eqref{Isdef}. Making this choice, we obtain 
\begin{equation*}
    \limsup_{N\to\infty}\Big|\int_0^t h_N'''(\sigma) \frac{\sigma^3}{6} d\sigma| \le |4s-3| I_s\|u_0\|_{L^2}^2 t^2.
\end{equation*}
Therefore, from \eqref{hNexpansion}, Proposition \ref{propt=0FN}, and \eqref{a.s.conv}, for every $0< |t| \le \tau(u_0)$ we have that 
\begin{align*}
    g(\Phi_t(u_0),u_0) &= \liminf_{N\to \infty} \frac{h_N(t)}{4s-3} \\
    &= \lim_{N\to\infty} \frac{F_N(u_0)}{4s-3}t + \lim_{N\to\infty} \frac{G_N(u_0)}{4s-3}\cdot \frac{t^2}{2} + \frac{1}{4s-3} \liminf_{N\to\infty} \int_0^t h_N'''(\sigma) \frac{\sigma^3}{6} d\sigma \\
    &\ge 8I_s \|u_0\|_{L^2}^2 t^2 - I_s\|u_0\|_{L^2}^2 t^2 \\
    &\ge I_s \|u_0\|_{L^2}^2 t^2 >0. 
\end{align*}
\end{proof} 
\begin{proof}[Proof of Proposition \ref{thm:sing}]
    This is an immediate corollary of Proposition \ref{thm2.0.1} and Proposition \ref{sing_abstract}.
\end{proof}
We now complete the paper by showing that Proposition \ref{thm:sing} implies Theorem \ref{thm:main}, (ii).
\begin{proof}[Proof of Theorem \ref{thm:main}, (ii)]
    Fix $\dl >0$, and define 
    $$\mathscr N_\dl := \Big\{ t \in \R: (\Phi_t)_\#\mu_s \not \perp \1_{\WP(\dl) \cap \WP(-\dl)} \mu_s\Big\}.$$
     We first show that $\mathscr N_\dl$ is countable for every $\dl > 0$. If  $\mu_s(\WP(\dl) \cap \WP(-\dl))=0$, then by definition $\mathscr N_\dl = \emptyset$, so without loss of generality, we can assume that $\mu_s(\WP(\dl) \cap \WP(-\dl))>0$. For $t \in \mathscr N_\dl$, pick a representative for the Radon-Nykodim derivative $\frac{d (\Phi_t)_\# \mu_s}{d\mu_s} $ of (the absolutely continuous part of) $(\Phi_t)_\# \mu_s$ with respect to $\mu_s$, such that the singular part of $(\Phi_t)_\# \mu_s$ with respect to $\mu_s$ is concentrated on the set
     $$ \Big\{\frac{d (\Phi_t)_\# \mu_s}{d\mu_s} = \infty\Big\}. $$
Define the set 
    $$ E_t:= \Phi_{-t}\Big(\big\{ 0 < \frac{d (\Phi_t)_\# \mu_s}{d\mu_s} \1_{\WP(\dl) \cap \WP(-\dl)} < \infty\big\}\Big).  $$
    Now fix an interval $I \subseteq \R$ with $|I|\le \dl$, and define 
    $$1 \ge a_{\dl,I} := \sup_{\substack{\mathscr T\subseteq \mathscr N_\dl \cap I,\\
    \mathscr{T} \text{ countable}}} \mu_s\big( \bigcup_{t \in \mathscr T} E_t\big). $$
    It is fairly easy to check that the $\sup$ in the definition of $a_{\dl,I}$ is a maximum, 
    i.e.\ there exists a countable set $\mathscr T_{\dl, I}\subseteq \mathscr N_\dl \cap I$ such that 
    $$ a_{\dl,I} = \mu_s\big( \bigcup_{t \in \mathscr T_{\dl,I}} E_t\big). $$
    Call $E_I:= \bigcup_{t \in \mathscr T_{\dl,I}} E_t$.
    From the definition of $a_{\dl, I}$, we deduce that, for every $t \in (\mathscr N_\dl \cap I) \setminus \mathscr T_{\dl, I}, $ 
    \begin{align*}
        0 & = \mu_s(E_I^c \cap E_t) \\
        &= \mu_s\Big(E_I^c \cap \Phi_{-t}\big(\big\{ 0 < \frac{d (\Phi_t)_\# \mu_s}{d\mu_s}\1_{\WP(\dl) \cap \WP(-\dl)} < \infty\big\}\big)\Big) \\
        &= (\Phi_t)_{\#}(\1_{E_I^c}\mu_s)\Big(\big\{ 0 < \frac{d (\Phi_t)_\# \mu_s}{d\mu_s}\1_{\WP(\dl) \cap \WP(-\dl)} < \infty\big\}\Big),
    \end{align*}
    and so for every $t \in (\mathscr N_\dl \cap I) \setminus \mathscr T_{\dl, I}$, $(\Phi_t)_{\#}(\1_{E_I^c}\mu_s) \perp \1_{\WP(\dl) \cap \WP(-\dl)}\mu_s$. From the definition $\mathscr N_\dl$,  we deduce that 
    \begin{equation} \label{Ndlinclusion}
            \begin{aligned}
         (\mathscr N_\dl \cap I) \setminus \mathscr T_{\dl, I} &\subseteq \{\tau \in I: (\Phi_\tau)_{\#} (\1_{E_I}\mu_s)\not\perp \mu_s \}\\
         &= \bigcup_{t \in \mathscr T_{\dl,I}} \{\tau \in I: (\Phi_\tau)_{\#} (\1_{E_t}\mu_s)\not\perp \mu_s \}.
     \end{aligned}  
    \end{equation}

     Now let $t \in \mathscr T_{\dl,I}$, and $\tau\in I$. By definition of $E_t$, we have that 
     $$\Phi_t(E_t) \subseteq \WP(\dl) \cap \WP(-\dl) \subseteq \WP(\tau-t),$$ and 
     $$(\Phi_t)_{\#} (\1_{E_t}\mu_s) \ll \mu_s. $$
Therefore, by Proposition \ref{thm:sing}, we deduce that 
\begin{align*}
    \{\tau \in I: (\Phi_\tau)_{\#} (\1_{E_t}\mu_s)\not\perp \mu_s \} & = \{\tau \in I: (\Phi_{\tau-t})_{\#} ((\Phi_t)_{\#}(\1_{E_t}\mu_s))\not\perp \mu_s \} \\
    &= \{\tau \in I: (\Phi_{\tau-t})_{\#} (\1_{\WP(\tau-t)}(\Phi_t)_{\#}(\1_{E_t}\mu_s))\not\perp \mu_s \}\\
    &\subseteq \{\tau \in I: (\Phi_{\tau-t})_{\#} (\1_{\WP(\tau-t)}\mu_s)\not\perp \mu_s \}
\end{align*}
is countable. By \eqref{Ndlinclusion}, recalling that the set $\mathscr T_{\dl,I}$ is countable, we deduce that $\mathscr N_\dl \cap I$ is countable as well. Since $I$ was an arbitrary interval with $|I|\le \dl$, taking unions over a countable family of $I$, we obtain that $\mathscr N_\dl$ is countable. 

Now, define 
$$ \mathscr N := \Big\{ t \in \R: (\Phi_t)_\#\mu_s \not \perp \mu_s\Big\}.$$
By Proposition \ref{prop:paraLWP}, we have that $\mu_s = \lim_{n \to \infty} \1_{\WP(\frac1n) \cap \WP(-\frac1n)} \mu_s.$ Therefore, 
$$ \mathscr N = \bigcup_{n \in \N} \mathscr N_{\frac1n} $$
is a countable union of countable sets, so it must be countable as well.
\end{proof}

\bibliography{references}

\end{document}